\documentclass[11pt]{aart}

\usepackage[letterpaper, hmargin=1in, top=1in, bottom=1.1in, footskip=0.4in]{geometry}

\usepackage{titlesec}
\titleformat{\section}[block]{\filcenter\normalfont\bfseries\large}{\thesection.}{.5em}{}\titlespacing*{\section}{0pt}{2\baselineskip}{1\baselineskip}
\titleformat{\subsection}[runin]{\normalfont\bfseries}{\thesubsection.}{.4em}{}[.]\titlespacing{\subsection}{0pt}{2ex plus .1ex minus .2ex}{.8em}
\titleformat{\subsubsection}[runin]{\normalfont\itshape}{\thesubsubsection.}{.3em}{}[.]\titlespacing{\subsubsection}{0pt}{1ex plus .1ex minus .2ex}{.5em}
\titleformat{\paragraph}[runin]{\normalfont\itshape}{\theparagraph.}{.3em}{}[.]\titlespacing{\paragraph}{0pt}{1ex plus .1ex minus .2ex}{.5em}

\usepackage[labelfont=bf,font=small,labelsep=period]{caption}
\usepackage[T1]{fontenc}
\usepackage{lmodern}

\usepackage{amsmath}
\usepackage{amssymb}
\usepackage{amsfonts}
\usepackage{latexsym}
\usepackage{amsthm}
\usepackage{amsxtra}
\usepackage{amscd}
\usepackage{bbm}
\usepackage{mathrsfs}
\usepackage{bm}
\usepackage[utf8x]{inputenc}
\usepackage{mathrsfs}  
\usepackage{dsfont} 
\usepackage{enumitem}
\usepackage{subcaption}
\usepackage{wrapfig}
\usepackage{pdfpages}
\usepackage{pgfplots}

\usepackage{tikz}
\usetikzlibrary{arrows,decorations.pathmorphing,backgrounds,positioning,fit,petri}

\usepackage{graphicx, color}

\definecolor{darkred}{rgb}{0.9,0,0.3}
\definecolor{darkblue}{rgb}{0,0.3,0.9}

\definecolor{vdarkred}{rgb}{0.6,0,0.2}
\definecolor{vdarkblue}{rgb}{0,0.2,0.6}
\usepackage[pdftex, colorlinks, linkcolor=vdarkblue,citecolor=vdarkred]{hyperref}

\usepackage{booktabs}
\usepackage[nottoc,notlof,notlot]{tocbibind}
\usepackage{cite} 

\numberwithin{equation}{section}
\numberwithin{figure}{section}

\theoremstyle{plain} 
\newtheorem{theorem}{Theorem}[section]
\newtheorem*{theorem*}{Theorem}
\newtheorem{lemma}[theorem]{Lemma}
\newtheorem*{lemma*}{Lemma}
\newtheorem{corollary}[theorem]{Corollary}
\newtheorem*{corollary*}{Corollary}
\newtheorem{proposition}[theorem]{Proposition}
\newtheorem*{proposition*}{Proposition}

\newtheorem*{conjecture*}{Conjecture}

\theoremstyle{definition} 
\newtheorem{definition}[theorem]{Definition}
\newtheorem{convention}[theorem]{Convention}
\newtheorem*{definition*}{Definition}

\newtheorem*{example*}{Example}
\newtheorem{remark}[theorem]{Remark}
\newtheorem*{remark*}{Remark}

\newtheorem*{assumption*}{Assumption}

\newcommand{\f}[1]{\boldsymbol{\mathrm{#1}}} 
\renewcommand{\r}{\mathrm}  
\newcommand{\bb}{\mathbb} 
\renewcommand{\cal}{\mathcal}

\newcommand{\ul}[1]{\underline{#1} \!\,} 
\newcommand{\wh}{\widehat}
\newcommand{\wt}{\widetilde}
\newcommand{\op}{\operatorname}

\renewcommand{\P}{\mathbb{P}}
\newcommand{\E}{\mathbb{E}}
\newcommand{\R}{\mathbb{R}}
\newcommand{\C}{\mathbb{C}}
\newcommand{\N}{\mathbb{N}}
\newcommand{\Z}{\mathbb{Z}}

\newcommand{\ee}{\mathrm{e}}

\newcommand{\dd}{\mathrm{d}}
\newcommand{\col}{\mathrel{\vcenter{\baselineskip0.75ex \lineskiplimit0pt \hbox{.}\hbox{.}}}}
\newcommand*{\deq}{\mathrel{\vcenter{\baselineskip0.5ex \lineskiplimit0pt\hbox{\scriptsize.}\hbox{\scriptsize.}}}=}

\newcommand{\eqdist}{\overset{\r d}{=}}

\renewcommand{\leq}{\leqslant}
\renewcommand{\geq}{\geqslant}
\renewcommand{\epsilon}{\varepsilon}

\newcommand{\ceil}[1]  {\lceil  #1 \rceil}

\newcommand{\qq}[1]{[\![#1]\!]}

\newcommand{\ind}[1]{\mathbbm 1_{#1}}

\newcommand{\pb}[1]{\bigl(#1\bigr)}
\newcommand{\pB}[1]{\Bigl(#1\Bigr)}
\newcommand{\pbb}[1]{\biggl(#1\biggr)}

\newcommand{\qb}[1]{\bigl[#1\bigr]}

\newcommand{\qbb}[1]{\biggl[#1\biggr]}

\newcommand{\h}[1]{\{#1\}}
\newcommand{\hb}[1]{\bigl\{#1\bigr\}}

\newcommand{\hbb}[1]{\biggl\{#1\biggr\}}

\newcommand{\abs}[1]{\lvert #1 \rvert}
\newcommand{\absb}[1]{\bigl\lvert #1 \bigr\rvert}

\newcommand{\absbb}[1]{\biggl\lvert #1 \biggr\rvert}

\newcommand{\absa}[1]{\left\lvert #1 \right\rvert}

\newcommand{\norm}[1]{\lVert #1 \rVert}

\newcommand{\normbb}[1]{\biggl\lVert #1 \biggr\rVert}

\newcommand{\scalar}[2]{\langle#1 \mspace{2mu}, #2\rangle}
\newcommand{\scalarb}[2]{\bigl\langle#1 \mspace{2mu}, #2\bigr\rangle}

\newcommand{\scalarbb}[2]{\biggl\langle#1 \,\mspace{2mu},\, #2\biggr\rangle}

\newcommand{\cond}{\,\vert\,}
\newcommand{\condb}{\,\big\vert\,}
\newcommand{\condB}{\,\Big\vert\,}

\DeclareMathOperator{\diag}{diag}

\DeclareMathOperator{\var}{Var}
\DeclareMathOperator{\supp}{supp}

\DeclareMathOperator{\sign}{sign}

\DeclareMathOperator{\spec}{spec}

\newcommand{\eps}{\varepsilon}

\newcommand*{\defeq}{\mathrel{\vcenter{\baselineskip0.5ex \lineskiplimit0pt\hbox{\scriptsize.}\hbox{\scriptsize.}}}=}

\newcommand{\id}{I_N} 					

\newcommand{\ord}{\mathcal{O}} 

\newcommand{\Cnu}{\mathcal{C}}

\title{Extremal eigenvalues of critical Erd\H{o}s-R\'enyi graphs} 
\author{}
\author{Johannes Alt \and Rapha\"el Ducatez \and Antti Knowles}
\date{\today}

\begin{document}
\maketitle

\begin{abstract}
We complete the analysis of the extremal eigenvalues of the adjacency matrix $A$ of the Erd\H{o}s-R\'enyi graph $G(N,d/N)$ in the critical regime $d \asymp \log N$ of the transition uncovered in \cite{BBK1, BBK2}, where the regimes $d \gg \log N$ and $d \ll \log N$ were studied. We establish a one-to-one correspondence between vertices of degree at least $2d$ and nontrivial (excluding the trivial top eigenvalue) eigenvalues of $A / \sqrt{d}$ outside of the asymptotic bulk $[-2,2]$. This correspondence implies that the transition characterized by the appearance of the eigenvalues outside of the asymptotic bulk takes place at the critical value $d = d_* = \frac{1}{\log 4 - 1} \log N$. For $d < d_*$ we obtain rigidity bounds on the locations of all eigenvalues outside the interval $[-2,2]$, and for $d > d_*$ we show that no such eigenvalues exist. All of our estimates are quantitative with polynomial error probabilities.

Our proof is based on a tridiagonal representation of the adjacency matrix and on a detailed analysis of the geometry of the neighbourhood of the large degree vertices. An important ingredient in our estimates is a matrix inequality obtained via the associated nonbacktracking matrix and an Ihara-Bass formula \cite{BBK2}. Our argument also applies to sparse Wigner matrices, defined as the Hadamard product of $A$ and a Wigner matrix, in which case the role of the degrees is replaced by the squares of the $\ell^2$-norms of the rows.
\end{abstract}

\section{Introduction}

This paper is about the extremal eigenvalues of sparse random matrices, such as the adjacency matrix of the Erd\H{o}s-R\'enyi graph. In spectral graph theory, obtaining precise bounds on the locations of the extreme eigenvalues, in particular on the spectral gap, is of fundamental importance and has attracted much attention in the past thirty years. See for instance \cite{Chu, HLW06, Alo98} for reviews.

The Erd\H{o}s-R\'enyi graph $G = G(N,d/N)$ is the simplest model of a random graph, where each edge of the complete graph on $N$ vertices is kept independently with probability $d/N$, with $0 < d < N$. Its adjacency matrix $A$ is the canonical example of a sparse random matrix, and its spectrum has been extensively studied in the random matrix theory literature. In the regime $d \equiv d_N \to \infty$ as $N \to \infty$, the empirical eigenvalue measure of $A / \sqrt{d}$ converges to the semicircle law supported on $[-2,2]$ \cite{Wig2, TVW}.

The behaviour of the extremal eigenvalues is more subtle, and has been investigated in several recent works \cite{BBK1, BBK2, EKYY1, EKYY2, FKoml, KS03, FO05, LHY17, LS1, Vu07, HLY17}. In particular, in \cite{KS03} it is shown that the largest eigenvalue $\lambda_1(A)$ of $A$ is asymptotically equivalent to the maximum of $d$ and the square root of the largest degree of $G$. A more difficult question is that of the other eigenvalues, $\lambda_2(A), \dots, \lambda_N(A)$, which determine in particular the gap $\lambda_1(A) - \lambda_2(A)$ between the largest and second-largest eigenvalues. By a standard eigenvalue interlacing argument, the analysis of the extremal eigenvalues $\lambda_2(A), \dots, \lambda_N(A)$ of $A$ is equivalent to the analysis of the eigenvalues $\lambda_1(\ul A), \dots, \lambda_N(\ul A)$ of the centred adjacency matrix $\ul A \deq A -  \E A$.

An important motivation for the present work is a transition in the behaviour of the extremal eigenvalues of $\ul A$ observed in \cite{BBK1, BBK2}. In \cite{BBK1} it is shown that in the regime $d \gg \log N$ the extremal eigenvalues $\lambda_1(\ul A/\sqrt{d})$ and $\lambda_N(\ul A/\sqrt{d})$ converge with high probability to the edges $+2$ and $-2$ of the semicircle law's support. Conversely, in \cite{BBK2} it is shown that in the regime $d \ll \log N$, the extremal eigenvalues $\lambda_1(\ul A/\sqrt{d})$ and $\lambda_N(\ul A/\sqrt{d})$ are asymptotically of order $\pm\sqrt{\eta / \log \eta}$ with $\eta \deq \frac{\log N}{d}$, placing them far outside of the interval $[-2,2]$.

Based on the two different behaviours observed in \cite{BBK1, BBK2}, we therefore expect a transition in the behaviour of the extremal eigenvalues on the \emph{critical} scale $d \asymp \log N$, where the extremal eigenvalues leave the support of the semicircle law.

In this paper we give a detailed analysis of this transition around the critical scale $d \asymp \log N$, which was left open by the works \cite{BBK1,BBK2}, by deriving quantitative high-probability bounds on the locations of all eigenvalues of $A/\sqrt{d}$ and $\ul A  /\sqrt{d}$ that lie outside the interval $[-2,2]$. Our analysis covers also the neighbouring sub- and supercritical regimes, $d \ll \log N$ and $d \gg \log N$, and in particular provides a complete picture of the transition between these two regimes.  Our approach also works for sparse Wigner matrices of the form $X = (X_{xy})_{x,y = 1}^N$, where $X_{xy} = W_{xy} A_{xy}$ and $(W_{xy} \col x \leq y)$ are uniformly bounded independent random variables with zero expectation and unit variance.

We remark that the critical scale $d \asymp \log N$ is the same as the well-known connectivity threshold for the Erd\H{o}s-R\'enyi, which happens precisely at the value $d = \log N$. In contrast, although the transition in the locations of the extremal eigenvalues of $\ul A$ happens on the same scale $d \asymp \log N$, it happens at a different numerical value, $d = b_* \log N$, where $b_* \deq \frac{1}{\log 4 - 1} \approx 2.59$.
The underlying cause is the same for both transitions: the lack of concentration of the degree sequence, which yields isolated vertices on the one hand and vertices of large degree on the other hand.

Indeed, the mechanism underlying the emergence of eigenvalues outside the support of the semicircle distribution for sufficiently sparse matrices is the appearance of vertices of large degree. This was already observed and exploited in \cite{BBK2} in the subcritical regime $d \ll \log N$. The intuition is that for sufficiently small $d$, the concentration of the degrees of the vertices around their mean $d$ fails, and we observe a number of vertices whose degree is much larger than $d$. This mechanism is also at the heart of our analysis. In fact, our main result is a high-probability correspondence between vertices of large degree and extremal eigenvalues. Roughly, we show that the following holds with probability at least $1 - N^{-\nu}$ for any fixed $\nu > 0$.
\begin{enumerate}
\item
Every vertex $x$ with degree $D_x$ larger than $(2 + o(1))d$ gives rise to exactly one eigenvalue of $\ul A / \sqrt{d}$ in $[2 + o(1), \infty)$ and one in $(-\infty, -2 - o(1)]$. These eigenvalues are located near $\pm \Lambda(D_x / d)$ respectively, where $\Lambda(t) \deq \frac{t}{\sqrt{t - 1}}$. The error is bounded by an inverse power of $d$.
\item
There are no other eigenvalues in $(-\infty, - 2 - o(1)] \cup [2 + o(1), \infty)$.
\end{enumerate}
Using standard results on the degree distribution of the Erd\H{o}s-R\'enyi graph (for the reader's convenience we review the necessary results in Appendix \ref{sec:degrees}), we can then easily conclude rigidity estimates for all eigenvalues of $\ul A /\sqrt{d}$ and $A / \sqrt{d}$ in the region $\R \setminus [-2-o(1), 2 + o(1)]$. Setting $d = b \log N$ for a fixed $b < b_*$, one can check (see Remark \ref{rem:rigidity} below) that with high probability there are $N^{1 - b / b_* + o(1)}$ such eigenvalues.

Our proof is based on the tridiagonal representation \cite{Tro84} of the matrix $\ul A$ around some vertex $x$. Thus, for any vertex $x \in [N]$ we consider the unit vector $\f 1_x$ supported at $x$ and rewrite $\ul A$ in the basis obtained by orthonormalizing the vectors $\f 1_x, \ul A \f 1_x, \ul A^2 \f 1_x, \dots$. The resulting matrix $M$ is tridiagonal and its spectrum coincides with that of $\ul A$. Denoting by $S_i(x)$ the sphere of radius $i$ around $x$, the key intuition behind our proof is that even though $D_x = \abs{S_1(x)}$ does not concentrate in the critical and subcritical regimes, the \emph{quotients} $\abs{S_{i+1}(x)} / \abs{S_{i}(x)}$, $i \geq 1$, do. Moreover, we note that balls of sufficiently small radius have only a bounded number of cycles with high probability, and can therefore be approximated by trees after a removal of a bounded number of edges. Thus, we expect the tridiagonal matrix $M$ to be close to that of a tree whose root $x$ has $D_x$ children and all other vertices $d$ children (see \eqref{eq:M_T_I_x_tridiagonal} and Figure \ref{fig:regular_tree} below). The spectrum of this latter matrix may be analysed using transfer matrix methods. We remark that this approximation requires precise information about the geometry of the neighbourhoods of vertices, and is only correct for vertices of large enough degree. 

In practice, we proceed as follows. For clarity, let us focus only on the positive eigenvalues. Our proof then consists of two major steps: deriving lower and upper bounds on the extremal eigenvalues of $\ul A$. For the lower bounds, we construct approximate eigenvectors $\f v^{(x)}$ of $\ul A$ around vertices $x$ of high degree, whose definition is motivated by the fact that $\f v^{(x)}$ would be an exact eigenvector if the approximation by a regular tree sketched above were exact. In addition to showing that these $\f v^{(x)}$ are indeed approximate eigenvectors with a quantitatively controlled error bound, we need to show that all of the associated eigenvalues in $[2 + o(1), \infty)$ are distinct. We do this by a careful pruning of the graph, with the property that all balls (in the pruned graph) of suitable radii around the vertices $\cal V_2 \deq \{x \in [N] \col D_x \geq 2 d\}$ are disjoint, and that the degrees of the difference between the original and pruned graphs are not too large. Since $\f v^{(x)}$ is supported in a sufficiently small ball around $x$, this will imply that the family $(\f v^{(x)})_{x \in \cal V_2}$ is orthogonal, and hence the associated eigenvalues of $\ul A / \sqrt{d}$ are distinct.

For the matching upper bounds on the extremal eigenvalues, a fundamental input is an Ihara-Bass type formula and a bound on the spectral radius of the \emph{nonbacktracking matrix} associated with $\ul A$ derived in \cite{BBK1}. This argument allows us to completely bypass typically very complicated combinatorial arguments needed in the moment method for estimating matrix norms. Thanks to the Ihara-Bass formula, the moment method is performed only on the level of the nonbacktracking matrix; this was already performed in \cite{BBK1} using a moment method that was very simple thanks to the nonbacktracking property. In particular, the lack of concentration of the degrees, which has a crucial impact on the extremal eigenvalues of $\ul A$, has no impact on the extremal eigenvalues (in absolute value) of the nonbacktracking matrix of $\ul A$. The outcome of this observation is the matrix inequality $\ul A / \sqrt{d} \leq I_N + D + o(1)$, where $D$ is the diagonal matrix with entries $D_x / d$. We apply this inequality to estimate the norm of the matrix $\ul A / \sqrt{d}$ restricted to vertices with degrees at most $2d$, and show that it is bounded by $2 + o(1)$. To that end we need to derive, for the maximal eigenvector of the restricted matrix, a delocalization bound at vertices with degree at least $(1 + o(1)) d$. This delocalization bound is derived using a careful analysis of the tridiagonal matrix associated with the restricted adjacency matrix. In fact, all of this analysis has to be done with the pruned adjacency matrix described above in order to obtain simultaneous upper bounds on all eigenvalues down to $2 + o(1)$. We refer to Section \ref{sec:ideas_proof} below for a more detailed summary of the proof.

The argument sketched above can also be easily applied to the sparse Wigner matrices $X$ described above, essentially by replacing the degree $D_x$ of a vertex by the square $\ell^2$-norm of the $x$-th row of $X$. The details are explained in Section \ref{sec:general_sparse_random_matrices} below.

Our method is rather general and in particular it is not tied to the homogeneity of the Erd\H{o}s-R\'enyi graph. We therefore expect it to be applicable to many other sparse random matrix models, such as inhomogeneous Erd\H{o}s-R\'enyi graphs and stochastic block models.

We remark that a related result for the eigenvalues induced by the top degree $D_{\rm max}$ appeared in the independent work \cite{TY1} while we were finalizing the current manuscript. In \cite{TY1}, the authors show that, for any fixed $l \in \N$, the largest / smallest $l$ eigenvalues of $\ul A$ are with high probability equal to $\pm (1 + o(1)) \Lambda(D_{\rm max} / d \vee 2) \sqrt{d}$. This corresponds to a qualitative version of our main result restricted to the top $O(1)$ eigenvalues, which all have the same asymptotic value. In this paper, we obtain quantitative rigidity bounds for all eigenvalues in $\R \setminus [-2-o(1), 2 + o(1)]$. For $d = b \log N$ with some fixed $b < b_*$, 
there are {with high probability} $N^{1 - b/b_* + o(1)}$ such eigenvalues. 
The precise location $d = {b_*} \log N$ for the transition in the behaviour of the top eigenvalues of the Erd\H{o}s-R\'enyi graph was also established in \cite{TY1}. 
Their argument also works for sparse Wigner matrices $X$ described above.  
The proof of \cite{TY1} differs substantially from ours; it relies on suitably chosen trial vectors and an intricate moment method argument controlled using cleverly constructed data structures.

We conclude the introduction with a brief outline of the paper. In Section \ref{sec:results} we state our results. The rest of the paper is devoted to the proofs. In Section \ref{sec:notations}, we introduce notations used throughout the paper and in Section \ref{sec:ideas_proof} give a more detailed summary of the proof. In Section~\ref{sec:large_degree_induces_eigenvalues}, we show that a vertex of degree greater than $2 d$ induces two approximate eigenvectors for the adjacency matrix.  The subsequent Section \ref{sec:proof_upper_bound_adjacency} is devoted to a quadratic form bound on the adjacency matrix in terms of the degree matrix. Lower and upper bounds on large eigenvalues of the adjacency matrix are established in 
 Section~\ref{sec:lower_bound_eigenvalues} and Section~\ref{sec:upper_bound_eigenvalues}, respectively. In the short Section \ref{sec:proofs_finish} we put everything together and conclude our main results for the Erd\H{o}s-R\'enyi graph. In Section \ref{sec:general_sparse_random_matrices}, we explain the minor changes required to handle sparse Wigner matrices. In the appendices, we collect some basic results on tridiagonal matrices and the degree distribution of Erd\H{o}s-R\'enyi graphs.

\paragraph{Convention} We regard $N$ as the fundamental large parameter. All quantities that are not explicitly \emph{fixed} may depend on $N$; we almost always omit the argument $N$ from our notation.

\section{Results}  \label{sec:results}
Let $A = (A_{xy})_{x,y \in [N]} \in \{0,1\}^{N\times N}$ be the adjacency matrix of the homogeneous Erd{\H o}s-R\'enyi graph with vertex set $[N] \defeq \{1, \ldots, N\}$ and edge probability $d/N$. That is, $A =A^*$, $A_{xx}=0$ for all $x \in [N]$, and  $( A_{xy} \col x < y)$
 are independent $\op{Bernoulli}(d/N)$ random variables.
Throughout this paper, $N$ is a large parameter and $d \equiv d_N$ depends on $N$. 
For each $x \in [N]$, we define the \emph{normalized degree} $\alpha_x$ of $x$ through
\begin{equation} \label{eq:def_alpha_x} 
 \alpha_x \defeq \frac{1}{d} \sum_{y \in [N]} A_{xy}. 
\end{equation}
We also consider the centred adjacency matrix $\underline{A} \defeq A - \E A$.  For any Hermitian matrix $M = M^* \in \R^{N\times N}$, we denote by $\lambda_1(M) \geq \lambda_2(M) \geq \ldots \geq \lambda_N(M)$ its eigenvalues. 

For $t \geq 2$ we define
\begin{equation} \label{eq:def_Lambda}
\Lambda(t) \deq \frac{t}{\sqrt{t-1}}.
\end{equation}
We denote by $\sigma \col [N] \to [N]$ a random permutation such that 
\begin{equation} \label{eq:def_sigma_permutation} 
\alpha_{\sigma(1)} \geq \alpha_{\sigma(2)} \geq \cdots \geq \alpha_{\sigma(N)}. 
\end{equation}
We can now state our main result.

\begin{theorem}\label{thm:correspondence_eigenvalue_large_degree}
Fix $0 < \kappa < 1/2$. 
Suppose that $0 < \theta \leq 1/2$ and $(\log N)^{4 / (5 - 2 \theta)} \leq d \leq N/2$. Define the random index
\begin{equation*}
L \deq \max \h{l \geq 1 \col \alpha_{\sigma(l)} \geq 2 + (\log d)^{-\kappa}}
\end{equation*}
with the convention that $L = 0$ if $\alpha_{\sigma(1)} < 2 + (\log d)^{-\kappa}$. Then there is a universal constant $c > 0$ such that for any $\nu > 0$ there is a constant $\cal C \equiv \cal C_{\nu,\kappa}$ such that the following holds with probability at least $1 - \cal C N^{-\nu}$.
\begin{enumerate}
\item
For $1 \leq l \leq L$ we have
\begin{equation*}
\absb{\lambda_l(\ul A) - \sqrt{d} \Lambda(\alpha_{\sigma(l)})} + \absb{\lambda_{N + 1 - l}(\ul A) + \sqrt{d} \Lambda(\alpha_{\sigma(l)})} \leq \cal C \pB{d^{-c (\Lambda(\alpha_{\sigma(l)}) - 2)}+ d^{-\theta/3}} \sqrt{d}.
\end{equation*}
\item
For $l = L + 1$ we have
\begin{equation*}
\max\{ \lambda_l(\underline{A}), -\lambda_{N + 1 - l}(\underline{A}) \} \leq \pB{2 + \Cnu  (\log d)^{-2\kappa }} \sqrt{d}.
\end{equation*}
\end{enumerate}
\end{theorem}

\begin{remark}
\begin{enumerate}[label=(\roman*)] 
\item In the supercritical regime $d \gg \log N$, Theorem \ref{thm:correspondence_eigenvalue_large_degree} is established in \cite{BBK2}, and in the subcritical regime $d \ll \log N$ it is established in \cite{BBK1}, in both cases with quantitative error bounds. Hence, in Theorem \ref{thm:correspondence_eigenvalue_large_degree} it would be sufficient to assume that $d \asymp \log N$. We allow a larger range $d \geq (\log N)^{4 / (5 - 2 \theta)}$ so as to obtain a simple statement that extends to all three regimes, showcasing the full behaviour through the transition at criticality. 
\item 
A simple analysis of the degrees shows that $L = 0$ in the supercritical regime, while 
$L$ is a fractional power of $N$ in the subcritical regime (see Appendix~\ref{sec:degrees}). 
\end{enumerate} 
\end{remark}

As a consequence of Theorem \ref{thm:correspondence_eigenvalue_large_degree}, for any $\nu > 0$ there is a constant $\cal C  \equiv \cal C_\nu> 0$ such that, with probability at least $1 - \cal C N^{-\nu}$,
\begin{equation} \label{norm_estimate}
\norm{\ul A} = \Lambda\pb{\alpha_{\sigma(1)} \vee 2} (1 + o(1)) \sqrt{d}.
\end{equation}

Another easy consequence is the corresponding statement for the non-centred adjacency matrix $A$, which follows by eigenvalue interlacing.

\begin{corollary} \label{cor:non_centred}
Under the same conditions and notations as in Theorem \ref{thm:correspondence_eigenvalue_large_degree}, the following holds with probability at least $1 - \cal C N^{-\nu}$.
\begin{enumerate}
\item
For $1 \leq l \leq L$ we have
\begin{multline*}
\absb{\lambda_{l + 1}(A) - \sqrt{d} \Lambda(\alpha_{\sigma(l)})} + \absb{\lambda_{N + 1 - l}(A) + \sqrt{d} \Lambda(\alpha_{\sigma(l)})}
\\
\leq \cal C  \pB{d^{-c (\Lambda(\alpha_{\sigma(l)}) - 2)}+ d^{-\theta/3} + (\alpha_{\sigma(l)} - \alpha_{\sigma(l+1)})} \sqrt{d}.
\end{multline*}
\item
For $l = L + 1$ we have
\begin{equation*}
\max\{ \lambda_{l+1}(A), -\lambda_{N + 1 - l}(A) \} \leq \pb{2 + \Cnu (\log d)^{-2\kappa}} \sqrt{d}.
\end{equation*}
\end{enumerate}
\end{corollary}

Note that the additional error term $(\alpha_{\sigma(l)} - \alpha_{\sigma(l+1)})$ is of order $1/d$ with high probability (see Proposition \ref{lem:degree_distr} below). It is well known that the largest eigenvalue $\lambda_1(A)$ is an outlier far outside the bulk spectrum; in fact a trivial perturbation argument using \eqref{norm_estimate} implies that $\abs{\lambda_1(A) - d} \leq \Lambda\pb{\alpha_{\sigma(1)} \vee 2} (1 + o(1)) \sqrt{d}$ with probability at least $1 - \cal C N^{-\nu}$, where $\nu > 0$ and $\cal C \equiv \cal C_\nu$.

Theorem \ref{thm:correspondence_eigenvalue_large_degree} (and its non-centred counterpart) can be combined with a standard analysis of the distribution of the degree sequence $D_{\sigma(1)}, D_{\sigma(2)}, \dots$ of the Erd\H{o}s-R\'enyi graph. For the convenience of the reader, in Appendix \ref{sec:degrees} we collect some basic results about the degree distribution. As an illustration, we state such an application for the extremal eigenvalues of $A$.

For its statement, we need the following facts from Appendix \ref{sec:degrees}. For any $d > 0$ and $1 \leq l \leq \frac{N}{C \sqrt{d}}$, the equation
\begin{equation*}
d(\beta \log \beta - \beta + 1) + \frac{1}{2} \log (2 \pi \beta d)  = \log (N / l)
\end{equation*}
has a unique solution $\beta_l(d)$ in $[1,\infty)$. (Here $C$ is a universal constant.) The interpretation of $\beta_l(d)$ is the typical value of the normalized degree $\alpha_{\sigma(l)}$. 

Then Corollary \ref{cor:non_centred} and Proposition \ref{lem:degree_distr} imply the following result.

\begin{corollary} \label{cor:er_beh}
Under the same conditions and notations as in Theorem \ref{thm:correspondence_eigenvalue_large_degree}, the following holds. Define the deterministic index
\begin{equation*}
\cal L(d) \deq \max \h{l \geq 1 \col \beta_l(d) \geq 2 + (\log d)^{-\kappa}}
\end{equation*}
with the convention that $\cal L(d) = 0$ if $\beta_1(d) < 2 + (\log d)^{-\kappa}$.
\begin{enumerate}
\item
For $1 \leq l \leq \cal L(d)$ we have with probability $1 - o(1)$
\begin{equation*}
\absb{\lambda_{l + 1}(A) - \sqrt{d} \Lambda(\beta_l(d))} + \absb{\lambda_{N + 1 - l}(A) + \sqrt{d} \Lambda(\beta_l(d))}
\leq C \pB{d^{-c (\Lambda(\beta_l(d)) - 2)}+ d^{-\theta/3}} \sqrt{d}.
\end{equation*}
\item
For $l = \cal L(d) + 1$ we have with probability $1 - o(1)$
\begin{equation*}
\max\{ \lambda_{l+1}(A), - \lambda_{N + 1 - l}(A) \} \leq \pb{2 + C (\log d)^{-2\kappa }} \sqrt{d}.
\end{equation*}
\end{enumerate}
(Here $C \equiv C_\kappa$ is a constant depending on $\kappa$.)
\end{corollary}

The errors $o(1)$ in the probabilities can be easily made quantitative by a slight refinement of the argument in Appendix \ref{sec:degrees}. 
See Figure \ref{fig:graph} for an illustration of Corollary \ref{cor:er_beh}.
An analogous result holds for the matrix $\ul A$, whose details we omit. 

\begin{figure}[!ht]
\begin{center}
{\small 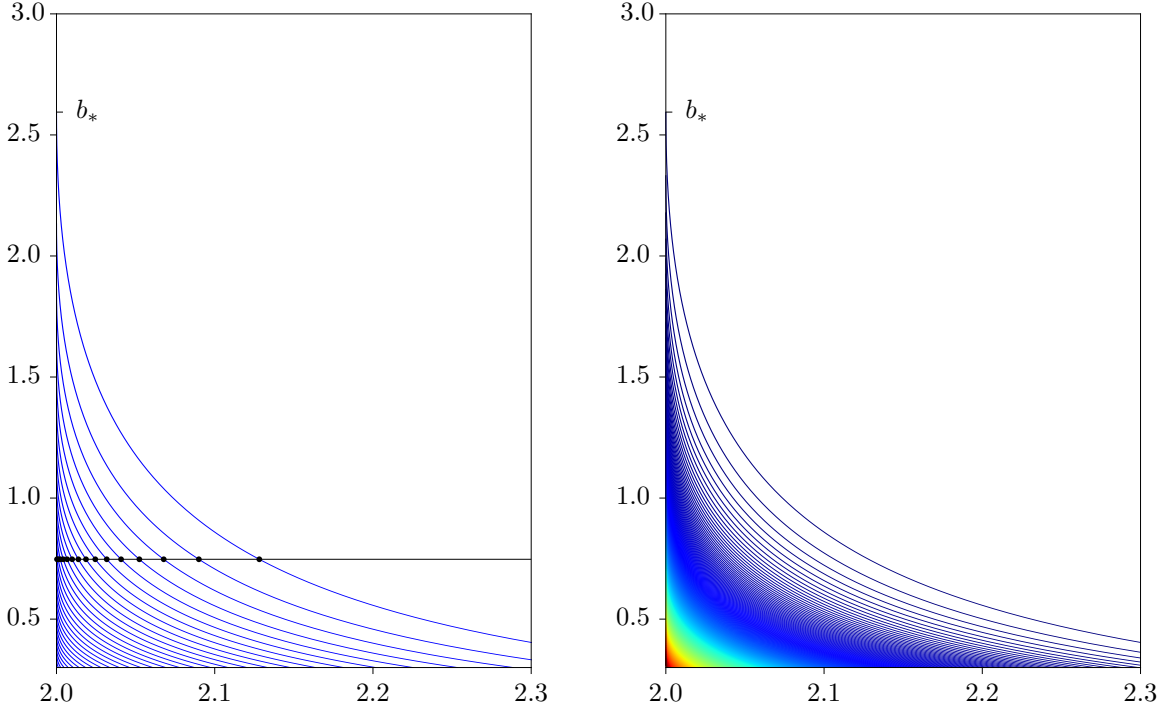}
\end{center}
\caption{An illustration of the typical values of the nontrivial eigenvalues $\lambda_2( A / \sqrt{d}), \lambda_3( A / \sqrt{d}), \dots$ in the interval $(2,\infty)$ (horizontal axis) as a function of $b = d / \log N$ (vertical axis). For each $l = 1,2, \dots$ we plot the function $b \mapsto \Lambda(\beta_l(b \log N))$. Left: $N = 50$; the typical eigenvalue configuration of $ A / \sqrt{d}$ in the interval $(2,\infty)$ for $d = b \log N$ is given by a horizontal slice of the graph at $b$, indicated by black dots. Right: $N = 1000$; we colour the graphs $b \mapsto \Lambda(\beta_l(b \log N))$ depending on $l$ to distinguish them from each other. Note that for $b > b_* \deq \frac{1}{\log 4 - 1} \approx 2.59$ there are no typical eigenvalues in $(2,\infty)$, and for $d = b \log N$ with fixed $b < b_*$ there are $N^{1 - b/b_* + o(1)}$ typical eigenvalues in $(2,\infty)$. \label{fig:graph}}
\end{figure}

\begin{remark}
There is a typical normalized degree $\beta_l(d)$ greater than or equal to $2$ if and only if $\beta_1(d) \geq 2$. Thus, we introduce the critical value $d_*$ as the unique solution of  $\beta_1(d_*) = 2$. It is easy to see that
\begin{equation*}
d_* = b_* \log N  + O(\log \log N), \qquad b_* \deq \frac{1}{\log 4 - 1}.
\end{equation*}
Since $\cal L(d) = 0$ for $d > d_*$ and $\cal L(d) \geq 1$ for $d \leq d_*$, we conclude from Corollary \ref{cor:er_beh} that $\lambda_2(A) / \sqrt{d}$ converges to $2$ in probability if and only if $\liminf \frac{d - d_*}{\log N} \geq 0$.
\end{remark}
\begin{remark} \label{rem:rigidity}
Fix $b < b_*$ and set $d = b \log N$. From the definition of $\beta_l(d)$, we deduce that $\abs{\{l \col \beta_l(d) \geq 2 + o(1)\}} = N^{1 - b/b_* + o(1)}$. Hence, using Corollary \ref{cor:er_beh}, we conclude that with probability $1 - o(1)$ the matrix $A / \sqrt{d}$ has $N^{1 - b/b_* + o(1)}$ eigenvalues in $\R \setminus [-2-o(1), 2 + o(1)]$.
\end{remark}

Our final result is a version of our results for sparse Wigner matrices. Let $A = (A_{xy})$ be as above and $W = (W_{xy})$ be an independent Wigner matrix with bounded entries. That is, $W$ is Hermitian and its upper triangular entries $(W_{xy} \col x \leq y)$ are independent complex-valued random variables with mean zero and variance one, $\E \abs{W_{xy}}^2 = 1$, and $\abs{W_{xy}} \leq K$ almost surely for some constant $K$. Then we define the sparse Wigner matrix $X = (X_{xy})$ as the Hadamard product of $A$ and $W$, with entries $X_{xy} \deq A_{xy} W_{xy}$.

\begin{theorem}\label{thm:eigenvalues_general_sparse_matrices}
Theorem \ref{thm:correspondence_eigenvalue_large_degree} holds also for the eigenvalues $\lambda_l(X)$ of a sparse Wigner matrix $X$ instead of $\lambda_l(\ul A)$, provided that the normalized degree $\alpha_x$ is replaced by
\begin{equation} \label{def_alpha_general}
\alpha_x = \frac{1}{d} \sum_{y \in [N]} \abs{X_{xy}}^2.
\end{equation}
Here, the constant $\Cnu$ from Theorem~\ref{thm:correspondence_eigenvalue_large_degree} depends on $K$ in addition to $\nu$ and $\kappa$. 
\end{theorem} 

Similarly, versions of Corollaries~\ref{cor:non_centred} and \ref{cor:er_beh} can be easily obtained if 
all entries of $W$ have a common positive mean with appropriate upper and lower bounds. 
Furthermore, for these results and Theorem~\ref{thm:eigenvalues_general_sparse_matrices} the 
boundedness assumption on the entries of $W$ can be considerably relaxed with some extra work.

\section{Notations}  \label{sec:notations} 

In this section we collect notations and tools used throughout this paper. The reader interested in the strategy of the proof can skip this section at first reading and proceed directly to Section~\ref{sec:ideas_proof}, returning to this section as needed for the precise notations.

We denote the positive integers by $\N = \{1, 2, 3, \ldots \}$ and define $\N_0 \defeq \N \cup \{0\}$. 
We set $[n] \defeq \{1, \ldots, n\}$
for any $n \in \N$, $[0] \defeq \emptyset$ and $\qq{r} \defeq \{ 0, \ldots, r\}$ for any $r \in \N_0$. We write $\abs{X}$ for the cardinality of the finite set $X$. 
We use $\ind{\Omega}$ as symbol for the indicator function of the event $\Omega$. 
Universal constants or estimates involving a universal constant are denoted by $C$ and $O(\,\cdot\,)$, respectively.

\paragraph{Notations related to vectors and matrices} 
Vectors in $\R^N$ are denoted by boldface lowercase Latin letters like $\f u$, $\f v$ and $\f w$ and their Euclidean norms by $\norm{\f u}$, $\norm{\f v}$ and $\norm{\f w}$, respectively. 
For a matrix $M \in \R^{N \times N}$, $\norm{M}$ is its operator norm induced by the Euclidean norm on $\R^N$. 

Let $M \in \R^{N \times N}$ be a matrix and $V \subset [N]$. We define the matrix $M_V \in \R^{\abs{V} \times \abs{V}}$ and the family $M_{(V)}$ through 
\[ M_V \defeq (M_{ij})_{i,j \in V}, \qquad \qquad M_{(V)} \defeq (M_{ij})_{i \in V \;\text{or}\; j \in V}. \]
If $V = \{ x \}$ for some $x \in [N]$ then we also write $M_{(x)}$ instead of $M_{(\{ x \})}$. 

The eigenvalues of a Hermitian matrix $M \in \R^{N\times N}$ are denoted by 
\[ \lambda_1(M) \geq \lambda_2(M) \geq \cdots \geq \lambda_N(M). \]  
Moreover, for Hermitian matrices $R,T \in \R^{N\times N}$ we write $R \geq T$ if 
\[ \scalar{\f w}{R \f w} \geq \scalar{\f w}{T \f w} \] 
for all $\f w \in \R^N$. We remark that this is equivalent to $\lambda_N(R - T) \geq 0$. 

For any $x \in [N]$, we define the standard basis vector $\f 1_x \defeq (\delta_{xy})_{y \in [N]} \in \R^N$.
To any subset $S \subset [N]$ we assign the vector $\f 1_S\in \R^N$ given by $\f 1_S \defeq \sum_{x \in S} \f 1_x$. 
Note that $\f 1_{\{ x\}} = \f 1_x$. 
We also introduce the normalized vector $\f e \defeq N^{-1/2} \f 1_{[N]}$. 
If $V \subset [N]$ and $\f w = (w_x)_{x \in [N]} \in \R^N$ then $\f w\vert_V$ denotes 
the vector in $\R^N$ with components $\scalar{ \f 1_y}{\f w\vert_V} \defeq \scalar{ \f 1_y}{\f w}$ for all $ y\in V$ and 
$\scalar{ \f 1_y}{\f w \vert_V} = 0$ for all $y \in [N] \setminus V$.

\paragraph{Notations related to graphs} 
In the entire paper, we consider finite graphs exclusively. 
Let $H$ and $G$ be two graphs. We write $H \subset G$ if $V(H) \subset V(G)$ and $E(H) \subset E(G)$. If $H \subset G$ then we denote by $G \setminus H$ the graph on $V(G)$ with edge set $E(G) \setminus E(H)$. 
To each graph $G = (V(G), E(G))$ we assign its adjacency matrix $\op{Adj}(G)$.
If $G$ is a graph on $[N]$ then, for any $V \subset [N]$, we denote by $G \vert_V$ the subgraph induced by $G$ on the vertex set $V$.
If $A$ is the adjacency matrix of $G$ then $A_V = \op{Adj}(G \vert_V)$ is the adjacency matrix of $G\vert_V$. 

For simplicity, we specialize to the vertex set $[N]$ in the following definitions. Let $H$ be a graph with vertex set $[N]$ and $M=\op{Adj}(H)$ be its adjacency matrix. 
Vertices in $[N]$ are usually labelled by $x,y,z$. The degree of the vertex $x$ is $D_x^H \deq \sum_{y \in [N]} M_{xy}$.  
With respect to $H$, the graph distance of two vertices $x,y \in [N]$ is denoted by 
\begin{equation*}
d^H(x,y) \deq \min \{k \in \N_0 : (M^k)_{xy} \neq 0\}.
\end{equation*}
For $i \in \N_0$, we introduce the $i$-sphere $S_i^H(x)$ and the $i$-ball $B_i^H(x)$ around $x$ defined through 
\[ S_i^H(x) = \{ y \in [N] \col d^H(x,y) = i \}, \qquad B_i^H(x) = \{ y \in [N] \col d^H(x,y) \leq i \}. \] 

For the remainder of this work, $G$ will be an Erd{\H o}s-R\'enyi graph with vertex set $[N]$ and edge probability $d/N$, where $N$ is a large parameter and $d\equiv d_N$ is a function of $N$. 
Moreover, $A=\op{Adj}(G) = (A_{xy})_{x,y \in [N]} \in \{0,1\}^{N\times N}$ will always denote the adjacency matrix of $G$. In this situation, we write $D_x$, $d(x,y)$, 
$S_i(x)$ and $B_i(x)$ instead of $D_x^G$, $d^G(x,y)$, $S_i^G(x)$ and $B_i^G(x)$, respectively. 
Note the relation $\alpha_x d = D_x$ between the normalized degree $\alpha_x$ defined in \eqref{eq:def_alpha_x} and the degree $D_x$.

\paragraph{Probabilistic notations and tools} 

We now introduce a notion of very high probability event as well as a notation for bounds which hold with very high probability. 
Both will be used extensively throughout the present work. 

\begin{definition}[Very high probability] \phantomsection \label{def:very_high_probability_def} 
\begin{enumerate}
\item
Let $\Xi \equiv \Xi_{N,\nu}$ be a family of events parametrized by $N \in \N$ and $\nu > 0$. We say that $\Xi$ \emph{holds with very high probability} if for every $\nu > 0$ there exists $\Cnu_\nu$ such that
\begin{equation*}
\P(\Xi_{N,\nu}) \geq 1 - \Cnu_\nu N^{-\nu}
\end{equation*}
for all $N \in \N$.
\item
For a $\sigma$-algebra $\cal F_N$ and an event $E_N \in \cal F_N$, we extend the definition (i) to \emph{$\Xi$ holds with very high probability on $E$ conditioned on $\cal F$} if for all $\nu > 0$ there exists $\Cnu_\nu$ such that
\begin{equation*}
\P(\Xi_{N, \nu} \cond \cal F_N) \geq 1 - \Cnu_\nu N^{-\nu}
\end{equation*}
almost surely on $E_N$, for all $N \in \N$.
\end{enumerate}
\end{definition}

We remark that the notion of very high probability survives a union bound involving $N^{O(1)}$ events.  We shall tacitly use this fact throughout the paper.

\begin{convention}[Estimates with very high probability] \label{conv:estimates_very_high_probability} 
In statements that hold with very high probability, we use the symbol $\cal C \equiv \cal C_\nu$ to denote a generic positive constant depending on $\nu$ such that the statement holds with probability at least $1 - c_\nu N^{-\nu}$ provided $\cal C_\nu$ and $c_\nu$ are chosen large enough. 
\end{convention} 

We now illustrate the previous convention by explaining in detail the meaning of $\abs{X} \leq \Cnu Y$ with very high probability. Such estimates often appear throughout the paper. 
The bound \emph{$\abs{X} \leq \Cnu Y$ with very high probability} means that, for each $\nu >0$, there are constants $\Cnu_\nu>0$ and $c_\nu >0$, depending on $\nu$, such that 
\[ \P \big( \abs{X} \leq \Cnu_\nu Y \big) \geq 1 - c_\nu N^{-\nu} \] 
for all $N \in \N$. Here, $X$ and $Y$ are allowed to depend on $N$. 

We also write $X = \cal O(Y)$ to mean $\abs{X} \leq \cal C Y$.

Throughout the following we use the function
\begin{equation} \label{eq:def_h} 
h(\alpha) \deq (1 + \alpha) \log (1 + \alpha) - \alpha
\end{equation}
for $\alpha \geq 0$.

To illustrate Definition~\ref{def:very_high_probability_def} and Convention~\ref{conv:estimates_very_high_probability}, we record the following lemma that we shall need throughout the paper.

\begin{lemma}[Upper bound on the degree]  \label{lem:upper_bound_degree} 
For any $x \in [N]$ we have with very high probability
\begin{equation*}
D_x \leq \Delta \leq \Cnu ( d+ \log N), 
\end{equation*} 
where $\Delta \equiv \Delta(d, N, \Cnu)$ is defined by 
\begin{equation} \label{eq:def_Delta}  
\Delta \defeq \begin{cases}
d + \cal C \sqrt{d \log N} & \text{if } d \geq \frac{1}{2} \log N
\\
\cal C \frac{\log N}{\log \log N - \log d} & \text{if } d \leq \frac{1}{2} \log N. 
\end{cases}
\end{equation}
\end{lemma} 

\begin{proof} 
From Bennett's inequality we obtain
\begin{equation*}
\P(D_x \geq d + \alpha d) \leq \ee^{-d h(\alpha)}.
\end{equation*}
The claim now follows from an elementary analysis of the right-hand side, by requiring that it be bounded by $N^{-\nu}$.
\end{proof}

\section{Main ideas of the proof} \label{sec:ideas_proof} 

In this section we explain the main ideas of the proof of Theorem~\ref{thm:correspondence_eigenvalue_large_degree}. 
Let $G$ be an Erd{\H o}s-R\'enyi graph with vertex set $[N]$ and edge probability $d/N$ and let $A$ be its adjacency matrix. 
In the actual proof, all arguments will be applied to $\underline{A}=A-\E A$. However, in this sketch, we explain certain ideas on the level of $A$ 
for the sake of clarity. In each case, a simple adjustment yields the argument for $\underline{A}$ instead of $A$. 

If $d \ll \log N$, then $A$ has many eigenvalues of modulus larger 
than $2 \sqrt{d}$ and they are related to vertices of large degree \cite{BBK2}. 
On the other hand, if $d \gg \log N$ then there are no eigenvalues whose modulus is larger than $2 \sqrt{d}$
\cite{BBK1} (with the exception of the trivial top eigenvalue of $A$).

In order to understand the relationship between large eigenvalues and vertices of large degree it is very insightful 
to analyse the structure of $G$ in the neighbourhood of a vertex $x \in [N]$ of large normalized degree $\alpha_x$. (In the following, we explain the arguments for large eigenvalues only. 
Dealing with small eigenvalues requires straightforward modifications.)  
If $\alpha_x$ is sufficiently large then there is $r_x \in \N$, depending on $\alpha_x$, such that 
$G$ has with very high probability the following properties. 
\begin{enumerate}[label=(\alph*)] 
\item \label{item:ideas_regular} For each $1 \leq i \leq r_x$, the ratio $\abs{S_{i+1}(x)}/\abs{S_i(x)}$ concentrates around $d$ (Lemma~\ref{lem:concentration_S_i} below). 
\item \label{item:ideas_tree} The subgraph $G|_{B_{r_x}(x)}$ is a tree up to a bounded number of edges (Lemma~\ref{lem:tree_approximation} below). 
\item \label{item:ideas_r_x_infty} The radius $r_x$ tends to infinity with $N$ (cf.~\eqref{eq:def_r_star} below).
\end{enumerate}

Owing to the properties \ref{item:ideas_regular} and \ref{item:ideas_tree} of the local geometry of $G$ around a vertex $x$ of large degree, it is natural to study the spectral properties of the adjacency matrix of the following idealized graph $\cal T$ on $[N]$. We suppose that in the ball $B_{r_x + 1}^{\cal T}(x)$ the graph $\cal T$ is a tree where the root vertex $x$ has $d \alpha_x$ children and the vertices in $B_{r_x}^{\cal T}(x) \setminus \{x\}$ have $d$ children. See Figure \ref{fig:regular_tree} 
for an illustration.

\begin{figure}[!ht]
\begin{center}
{\small 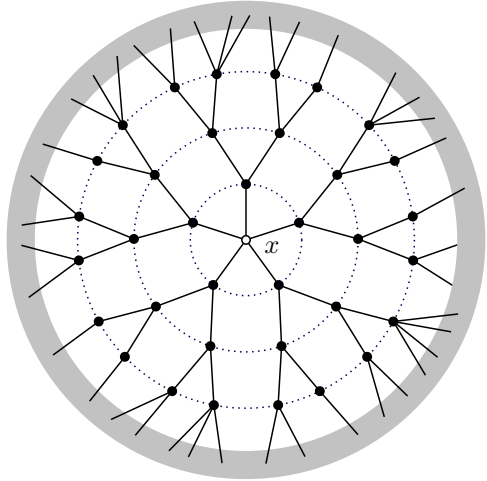}
\end{center}
\caption{The regular tree graph $\cal T$ with $r_x = 2$, $d = 2$, and $D_x = d \alpha_x = 5$. We only draw vertices in the ball $B_{r_x + 1}^{\cal T}(x)$, while the remaining vertices in $[N] \setminus B_{r_x + 1}^{\cal T}(x)$ are in the grey area. \label{fig:regular_tree}}
\end{figure}

The adjacency matrix associated with $\mathcal{T}$ is denoted by $A^\mathcal{T}$. 
The following standard construction \cite{Tro84} yields a convenient approach to the spectral analysis of $A^\mathcal{T}$. 
Let $\f s_0, \ldots, \f s_{r_x}$ be the Gram-Schmidt orthonormalization of $\f 1_x, (A^\mathcal{T}) \f 1_x, \ldots, (A^\mathcal{T})^{r_x} \f 1_x$.
Let $\f s_{{r_x}+1}, \ldots, \f s_{N-1}$ be any completion 
of $\f s_0, \ldots, \f s_{r_x}$ to an orthonormal basis of $\R^N$. We denote by $M^\mathcal{T}$ the matrix representation of $A^\mathcal{T}$ in this basis, i.e., 
\begin{equation} \label{eq:relation_M_A_cal_T} 
 M^\mathcal{T} = S^* A^\mathcal{T} S, \qquad \qquad S \defeq (\f s_0, \ldots, \f s_{N-1}) \in \R^{N\times N}. 
\end{equation}
Note that $A^\mathcal{T}$ and $M^\mathcal{T}$ have the same spectrum.  
The upper-left $(r_x+1) \times (r_x+1)$ block $(M^\mathcal{T})_{\qq{r_x}}$ of $M^\mathcal{T}$ has the tridiagonal form
\begin{equation}\label{eq:M_T_I_x_tridiagonal} 
(M^\mathcal{T})_{ \qq{r_x}} = \sqrt{d} 
\begin{pmatrix}
0 & \sqrt{\alpha_x} &&&&
\\
\sqrt{\alpha_x} & 0 & 1 &&&
\\
& 1 & 0 & 1 &&
\\&&1 & 0 & \ddots&
\\
&&&\ddots & \ddots & 1
\\
&&&& 1 & 0
\end{pmatrix}
\end{equation}
(see Lemma~\ref{lem:Gram_Schmidt} below). 
For $\alpha_x >1$ and $u_0 > 0$, we define the vector $\f u = (u_k)_{k=0}^{N-1}$ with components 
\[ u_1 \defeq \bigg(\frac{\alpha_x}{\alpha_x - 1}\bigg)^{1/2} u_0, \qquad u_i \defeq \bigg( \frac{1}{\alpha_x - 1}\bigg)^{(i-1)/2} u_1, \qquad u_j = 0\] 
for $i = 2, 3, \ldots, r_x$ and $j=r_x + 1, \ldots, N-1$.
If $\alpha_x >2$ then $u_i$ decays exponentially with $i$.  
Therefore, using the tridiagonal structure of $(M^\mathcal{T})_{ \qq{r_x}}$ from~\eqref{eq:M_T_I_x_tridiagonal}
and that $r_x$ is large, we see that 
$\f u$ is an approximate eigenvector of $M^\mathcal{T}$ corresponding to the 
approximate eigenvalue $\sqrt{d} \Lambda(\alpha_x)$, where $\Lambda(t)$ is defined as in \eqref{eq:def_Lambda} 
 (see Lemma~\ref{lem:transfermatrix} below). Therefore, owing to~\eqref{eq:relation_M_A_cal_T}, the vector 
\begin{equation} \label{eq:approximate_eigenvector_tree} 
  \sum_{i=0}^{r_x} u_i \f s_i 
\end{equation}
is an approximate eigenvector of $A^\mathcal{T}$ with approximate eigenvalue $\sqrt{d} \Lambda(\alpha_x)$. 

For $i=0, \ldots, r_x$, we have $\f s_i = \abs{S_i^\cal{T}(x)}^{-1/2} \f 1_{S_i^{\cal{T}}(x)}$ (see Lemma~\ref{lem:Gram_Schmidt} below). 
Hence, the construction in \eqref{eq:approximate_eigenvector_tree} naturally suggests to consider 
\begin{equation} \label{eq:def_v_pedagogical} 
 \f v = \sum_{i=0}^{r_x} u_i \abs{S_i(x)}^{-1/2} \f 1_{S_i(x)} 
\end{equation}
as approximate eigenvector of $A$, i.e., to replace $\f s_i$ in \eqref{eq:approximate_eigenvector_tree} by $\abs{S_i(x)}^{-1/2}\f 1_{S_i(x)}$.  
In Proposition~\ref{pro:approximate_eigenvector} below, we show that $\f v$ is an approximate eigenvector of $A$ 
with approximate eigenvalue $\sqrt{d} \Lambda(\alpha_x)$. The proof heavily relies on the properties \ref{item:ideas_regular}, \ref{item:ideas_tree} and \ref{item:ideas_r_x_infty} listed above 
and justified in Section~\ref{sec:large_degree_induces_eigenvalues}. 

The proof of Theorem~\ref{thm:correspondence_eigenvalue_large_degree} requires two additional key steps. Namely, 
\begin{enumerate}[label=(\roman*)]
\item \label{item:ideas_i} two different vertices of large degree induce two different eigenvalues,
\item \label{item:ideas_ii} all eigenvalues of modulus larger than $2 \sqrt{d}$ arise from vertices of large degree.  
\end{enumerate} 
We remark that \ref{item:ideas_i} is equivalent to a lower bound on the $l$-th largest eigenvalue in terms of the $l$-th largest degree of $G$ 
while \ref{item:ideas_ii} is equivalent to a corresponding upper bound.   

For \ref{item:ideas_i}, we construct the \emph{pruned graph} $G_2$. It is a subgraph of $G$ such that $A$ is well approximated by the adjacency matrix $A_2$ of 
$G_2$ and $B_{r_x}^{G_2}(x)$ and $B_{r_y}^{G_2}(y)$ are disjoint if $x,y \in [N]$, $x \neq y$ and $\alpha_x, \alpha_y \geq 2$ (see Lemma~\ref{lem:subgraph_separating_large_degrees} below).
Hence, the construction in \eqref{eq:def_v_pedagogical} yields two orthogonal approximate eigenvectors of $A$ which thus induce two 
different eigenvalues (or the same eigenvalue with multiplicity at least two). 
This completes \ref{item:ideas_i} (cf. Proposition~\ref{pro:lower_bound_number_outliers}).

Thanks to \ref{item:ideas_i}, we now know that $\lambda_1(\underline{A})\geq \ldots \geq \lambda_L(\underline{A}) \geq (2 + o(1))\sqrt{d}$ if $L \defeq N-\abs{V}$ and $V \defeq \{ x \in [N] \col \alpha_x \leq 2\}$. 
Hence, \ref{item:ideas_ii} will follow if we can show that $\lambda_{L+1}(\underline{A}) \leq (2 + o(1))\sqrt{d}$. 
By the min-max principle, we have 
\[ \max_{\f w \in \mathbb{S}(U)} \scalar{\f w}{\underline{A} \f w} \geq \lambda_{L+1}(\underline{A}), \] 
where $\mathbb{S}(U)$ is the unit sphere in the linear subspace $U \defeq \op{span}\{ \f 1_x \col x \in V \}\subset \R^N$. 
Thus, it suffices to establish an upper bound on the largest eigenvalue $\mu$ of $\underline{A}_V$. 
This will be deduced from  
the matrix inequality
\begin{equation} \label{eq:ihara_bass_inequality_pedagogical} 
 \id + D + o(1)\geq d^{-1/2} \underline{A} 
\end{equation}
which holds with very high probability (see Proposition~\ref{pro:upper_bound_on_adjacency_matrix} below). Here, $D = (\alpha_x \delta_{xy})_{x,y \in [N]}$ is the diagonal matrix of normalized degrees. 
The inequality \eqref{eq:ihara_bass_inequality_pedagogical} is a consequence of an estimate on the nonbacktracking matrix associated with $\underline{A}$ and an Ihara-Bass type formula 
from \cite{BBK1}. 

We now explain how to prove that $\mu$ is at most $(2 + o(1)) \sqrt{d}$. 
Let $\tilde{\f w}=(\tilde{w}_x)_{x \in V}$ be a normalized eigenvector of $\underline{A}_V$ corresponding to $\mu$. 
We define a normalized vector $\f w =(w_x)_{x \in [N]} \in \R^N$ through $w_x = \tilde{w}_x$ for $x \in V$ and $w_x = 0$ for $x \in [N] \setminus V$. 
Since $\scalar{\tilde{\f w}}{\underline{A}_V \tilde{\f w}} = \scalar{\f w}{\underline{A} \f w}$ we can evaluate the inequality in \eqref{eq:ihara_bass_inequality_pedagogical} at $\f w$.  
This yields
\begin{equation} \label{eq:ihara_bass_consequence_pedagogical} 
 \frac{\mu}{\sqrt{d}} - o(1) \leq \scalar{\f w}{(\id + D) \f w} = 1 + \sum_{x \col \alpha_x > 2} \alpha_x w_x^2 + 
\sum_{x \col 2 \geq \alpha_x > \tau} \alpha_x w_x^2 + \sum_{x \col \alpha_x \leq \tau} 
\alpha_x w_x^2 
\end{equation} 
for any $\tau \in (1,2)$, where we used that $\f w$ is normalized. The contribution for $\alpha_x>2$ vanishes as $w_x=0$ for such $x$. Since $\f w$ is normalized the contribution for $\alpha_x \leq \tau$ is at most $\tau$. 
We choose $\tau = 1 + o(1)$. 

What remains is estimating the sum in the regime $2 \geq \alpha_x > \tau$. In the following paragraph, we shall sketch the proof of the bound 
\begin{equation} \label{eq:bound_eigenvector_pedagogical} 
w_x^2 \leq \eps \norm{\f w \vert_{B_{r_x}^{G_\tau}(x)}}^2 
\end{equation}
which holds for some $\eps = o(1)$ uniformly for all $x \in [N]$ satisfying $\tau < \alpha_x \leq 2$. 
Here, $G_\tau$ is the \emph{pruned graph}, a subgraph of $G$ such that $A_\tau=\op{Adj}(G_\tau)$, the adjacency matrix of $G_\tau$, and $A$ are close and $B_{r_x}^{G_\tau}(x)$ and $B_{r_y}^{G_\tau}(y)$ are disjoint for all vertices $x,y \in [N]$ satisfying $x \neq y$ and $\alpha_x, \alpha_y>\tau$ (compare Lemma~\ref{lem:subgraph_separating_large_degrees} below). 
Given \eqref{eq:bound_eigenvector_pedagogical}, we conclude 
\[ \sum_{x \col 2 \geq \alpha_x > \tau} \alpha_x w_x^2 \leq 2 \sum_{x \col 2 \geq \alpha_x > \tau} w_x^2 \leq 2 \eps \norm{\f w}^2, \] 
where we employed in the last step that $(\f 1_{B_{r_x}^{G_\tau}(x)})_{x \col \alpha_x > \tau}$ is a family of orthogonal vectors. 
Since $\norm{\f w} = 1$, $\eps = o(1)$ and $\tau = 1 + o(1)$, we obtain from \eqref{eq:ihara_bass_consequence_pedagogical} that $\mu \leq ( 1 + o(1) + \tau + 2\eps )\sqrt{d} = ( 2 + o(1))\sqrt{d}$. 
Therefore, $\lambda_{L+1}(\underline{A})\leq \mu \leq (2 + o(1))\sqrt{d}$.

We now sketch the proof of \eqref{eq:bound_eigenvector_pedagogical}. For the graph $\cal{T}$ described above, the delocalization estimate in \eqref{eq:bound_eigenvector_pedagogical} can be 
obtained by analysing the tridiagonal matrix $M^\mathcal{T}$ introduced in \eqref{eq:relation_M_A_cal_T} via 
a transfer matrix argument. As $G_\tau$ is close to $\mathcal{T}$ locally around a vertex $x$ satisfying $\alpha_x > \tau$ 
the tridiagonal matrix $\wh{M}$ constructed from $A_\tau$ around $x$ is well approximated by $M^\mathcal{T}$. 
Hence, the transfer matrices associated with $\wh{M}$ and $M^\mathcal{T}$ are also close and a version of the argument for $M^\mathcal{T}$ can be used to deduce \eqref{eq:bound_eigenvector_pedagogical}. 
This completes the sketch of the proof of \ref{item:ideas_ii} and thus the sketch of the proof of Theorem~\ref{thm:correspondence_eigenvalue_large_degree}.

\section{Large eigenvalues induced by vertices of large degree} \label{sec:large_degree_induces_eigenvalues} 

Let $G$ be an Erd{\H os}-R\'enyi graph on the vertex set $[N]$ with edge probability $d/N$.  
Let $A = \op{Adj}(G)$ be the adjacency matrix of $G$ and $\underline{A} \defeq A - \E A$. 
Proposition~\ref{pro:approximate_eigenvector} below, the main result of this section, shows that each vertex of sufficiently large degree induces 
two approximate eigenvectors of $\underline{A}$.  
As explained after the statement of Proposition~\ref{pro:approximate_eigenvector}, this locates a positive and a negative eigenvalue of $\underline{A}$ of large modulus. 

We now introduce the notation necessary to define the approximate eigenvectors.
To lighten notation, we fix the vertex $x$ throughout and omit all arguments $(x)$ from our notation. In particular, we just write $S_i$ and $B_i$ instead of $S_i(x)$ and $B_i(x)$.
Define
\begin{equation}\label{eq:def_r_star} 
{r_x} \defeq \bigg\lfloor \frac{\log N}{3\log D_x} \bigg\rfloor,
\end{equation}
and let $r \leq r_x$.
Let $u_0 > 0$ and define the coefficients
\begin{equation} \label{eq:def_v_x_coeff}
u_1 \defeq \frac{\sqrt{D_x}}{\sqrt{D_x-d}} u_0, \qquad u_i \defeq \frac{d^{(i-1)/2}}{(D_x-d)^{(i-1)/2}} u_1 \quad (i=2, 3, \ldots, r+1).
\end{equation}
Here and in the following, we exclusively consider the event $\{ D_x >d \}$ such that $u_1, \ldots, u_{r+1}$ 
are always well-defined.
Then, on the event $S_i \neq \emptyset$ for $i=1,\ldots, r$, we define the approximate eigenvectors $\f v \equiv \f v(x, r)$ and $\f v_- \equiv \f v_-(x,r)$ through
\begin{equation} \label{eq:def_v_x} 
\f v \defeq \sum_{i=0}^{r} u_i \f s_i, \quad \qquad \f v_- \defeq \sum_{i=0}^r (-1)^i u_i \f s_i, \quad \qquad \f s_i \defeq \abs{S_i}^{-1/2} \f 1_{S_i}.
\end{equation}
Finally, we choose $u_0$ so that the normalization $\norm{\f v}^2 = \norm{\f v_-}^2 = \sum_{i = 0}^{r} u_i^2 = 1$ holds.

For the following proposition, we recall the definition $\alpha_x = D_x / d$. Also, throughout this section we use $\cal K \geq 1$ to denote a constant that is chosen large enough depending on $\nu$ in the definition of very high probability. 

\begin{proposition}[Eigenvectors induced by vertex of large degree] \label{pro:approximate_eigenvector}
Let $x \in [N]$ be a fixed vertex. Suppose that $\cal K \sqrt{\log N} \leq d \leq N^{1/4}$
and $\log d \leq r \leq r_x$. Then
\[ \norm{(\underline{A} - \sqrt{d} \Lambda(\alpha_x)) \f v} + \norm{(\underline{A} + \sqrt{d} \Lambda(\alpha_x)) \f v_-} \leq
\cal C \bigg( \log d + \frac{\log N}{d}\bigg)^{1/2}\bigg( 1+ \frac{\log N}{D_x} \bigg)^{1/2}
\] 
with very high probability on
\begin{equation}\label{eq:approximate_eigenvector_event} 
\hbb{\pbb{2 + 2 \frac{\log d}{r}} d \leq D_x \leq \sqrt{N} (2d)^{-r}}
\end{equation}
conditioned on $D_x$. 
\end{proposition} 

We remark that if $M$ is a Hermitian matrix and $\f v$ a normalized vector such that $\norm{M \f v} \leq \epsilon$ then $M$ has an eigenvalue in $[-\epsilon, \epsilon]$. 
Therefore, Proposition~\ref{pro:approximate_eigenvector} implies that $\underline{A}$ possesses with very high 
probability two eigenvalues $\lambda_\pm$ in the vicinity of $\pm \sqrt{d}\Lambda(\alpha_x)$ if $\alpha_x$
is sufficiently large.

We shall show in Lemma~\ref{lem:concentration_S_i} below that $S_i \neq \emptyset$ for $i = 1, \ldots, {r}$ with very high probability on the event $\h{d\leq D_x \leq \sqrt{N}(2d)^{-r}}$.

To prove Proposition~\ref{pro:approximate_eigenvector}, we only consider the term $\norm{(\ul{A} - \sqrt{d} \Lambda(\alpha_x)) \f v}$. The other term is treated in the same way. 
We shall decompose $(\underline{A} - \sqrt{d} \Lambda(\alpha_x))\f v$ into a sum $\f w_0 + \cdots + \f w_4$ of vectors, which are all proved to have a small norm.  (See Lemma~\ref{lem:decomposition_app_eigenvector} below and the estimates in Lemma~\ref{lem:estimates_w_k} below.) Each of the vectors $\f w_i$ will turn out to be small for a different reason, which is why we treat them individually.

In order to define the vectors $\f w_i$, we introduce the notations
\begin{equation} \label{eq:def_e_N_i} 
 \f e \defeq N^{-1/2} \f 1_{[N]}, \qquad N_i(y) \defeq \scalar{\f 1_y}{A \f 1_{S_i}} = \abs{S_i \cap S_1(y)} 
\end{equation}
for all $i=0,\ldots, {r}$ and $y \in [N]$. Thus, $N_i(y)$ is the number of edges starting in $S_i$ and ending in $y$. Note that if the graph $G|_{B_{i+1}}$ is a tree then it is easy to see that
\begin{equation} \label{eq:Ny_tree}
N_i(y) = \ind{y \in S_{i - 1}} (D_y - \ind{i \geq 2}) + \ind{y \in S_{i+1}}\,
\end{equation}
with the convention that $S_{-1} \defeq \emptyset$. See Figure~\ref{fig:Ni} for an illustration of $N_i(y)$. 

\begin{figure}[!ht]
\begin{center}
{\small 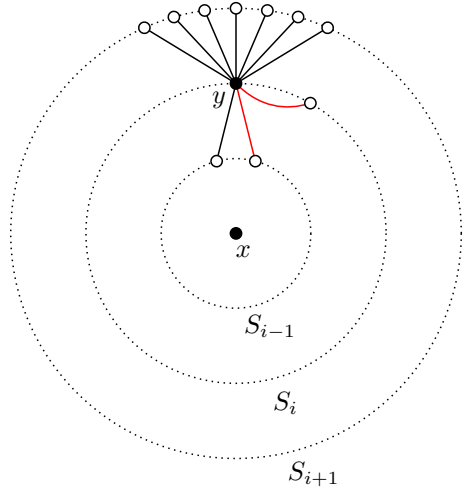}
\end{center}
\caption{An illustration of the definition of $N_i(y)$ from \eqref{eq:def_e_N_i}, where $y \in S_i(x)$. The red edges are forbidden in a tree. For a tree, $N_{i+1}(y) = D_y - 1$, $N_i(y) = 0$, and $N_{i - 1}(y) = 1$.\label{fig:Ni}}
\end{figure}

Define
\begin{equation} \label{eq:def_w_k} 
\begin{aligned}
 \f w_0 & \defeq \frac{d}{N} \f v - d \scalar{\f e}{\f v} \f e, 
 \\
 \f w_1 & \defeq \sum^{{r}}_{i= 0} \frac{u_i}{\sqrt{\abs{S_i}}} \left( \sum_{y \in S_{i+1}} \big(N_i(y) -1\big) \f 1_y 
+\sum_{y \in S_i} N_i(y) \f 1_y\right), \\ 
 \f w_2 & \defeq \sum_{i=1}^{{r}} \frac{u_i}{\sqrt{\abs{S_i}}} \sum_{y \in S_{i-1}} \left( N_i(y) - \frac{\abs{S_i}}{\abs{S_{i-1}}} \right) \f 1_y,  \\ 
\f w_3 & \defeq u_2 \left( \frac{\sqrt{\abs{S_2}}}{\sqrt{\abs{S_1}}} - \sqrt{d} \right) \f s_1 + \sum_{i=2}^{{r} - 1} \left[ u_{i+1} \left(\frac{\sqrt{\abs{S_{i+1}}}}{\sqrt{\abs{S_i}}} -\sqrt{d} \right) 
 + u_{i-1} \left( \frac{\sqrt{\abs{S_i}}}{\sqrt{\abs{S_{i-1}}}} - \sqrt{d} \right)  \right] \f s_i , \\ 
 \f w_4 & \defeq \left( u_{{r}-1} \frac{\sqrt{\abs{S_{r}}}}{\sqrt{\abs{S_{{r}-1}}}}  - u_{{r}-1} \sqrt{d}  - u_{{r}+1} \sqrt{d} \right) \f s_{r}
 + u_{r} \frac{\sqrt{\abs{S_{{r}+1}}}}{\sqrt{\abs{S_{r}}}} \, \f s_{{r}+1}.
\end{aligned}
\end{equation}

\begin{lemma}[Decomposition of $(\underline{A}- \sqrt{d}\Lambda(\alpha_x))\f v$] \label{lem:decomposition_app_eigenvector} 
We have the decomposition 
\begin{equation} \label{eq:decomposition_approximate_eigenvector}
 (\underline{A} - \sqrt{d} \Lambda(\alpha_x)) \f v = \f w_0 + \f w_1 + \f w_2 + \f w_3 + \f w_4. 
\end{equation}
\end{lemma} 

Lemma \ref{lem:decomposition_app_eigenvector} will be shown in Subsection~\ref{subsec:proof_decomposition_app_eigenvector} below. 
We now explain the origin and interpretation of the different errors $\f w_0, \dots, \f w_4$.
\begin{itemize}
\item
The vector $\f w_0$ is equal to $- (\E A) \f v$, and hence takes care of the expectation $\E A$ in the definition of $\underline{A} = A - \E A$. It will turn out to be small because the vector $\f v$ is localized near the vertex $x$, and hence has a small overlap with $\f e$, which is completely delocalized.
\item
The vector $\f w_1$ quantifies the extent to which $G|_{B_{{r}+1}}$ deviates from a tree. Indeed, by \eqref{eq:Ny_tree} it vanishes if $G|_{B_{{r}+1}}$ is a tree. It will turn out to be small because the number of cycles in $G|_{B_{{r}+1}}$ is not too large.
\item
The vector $\f w_2$ quantifies the extent to which $G|_{B_{r}}$ deviates from a tree with the property that, for each $i \geq 2$, all vertices in $S_{i}$ have the same degree. Indeed, it is immediate that the term $i = 1$ is always zero, and the other terms vanish under the above condition, by \eqref{eq:Ny_tree}. It will turn out to be small because the number of cycles in $G|_{{B_{r+1}}}$ is not too large and because $N_i(y)$ will concentrate around $\frac{\abs{S_i}}{\abs{S_{i-1}}}$ for most vertices $y \in S_{i - 1}$, for any $i \geq 2$.
\item
The vector $\f w_3$ quantifies the extent to which $G|_{B_{r}}$ deviates from a graph with the property that $\abs{S_{i+1}} = d \abs{S_i}$ for all $i \geq 1$. 
The ratios $\abs{S_{i+1}}/\abs{S_i}$ will turn out to concentrate around $d$ with very high probability, 
thus ensuring the smallness of $\f w_3$. 
\item
Finally, the vector $\f w_4$ quantifies the error arising from edges connecting the ball $B_{r}$, where the tree approximation is valid, to the rest of the graph $[N] \setminus B_{r}$, where it is not. It will be small by the exponential decay of the coefficients $u_i$.
\end{itemize}

\begin{lemma}[Estimates on $\f w_0, \ldots, \f w_4$] \label{lem:estimates_w_k} 
Let $\cal K \sqrt{\log N}\leq d \leq N^{1/4}$. 
For any $r \leq r_x$, the estimates
\begin{subequations} 
\begin{align} 
\norm{\f w_0} & =  \ord(d N^{-1/4})  \label{eq:estimate_w_0}  \\ 
\norm{\f w_1} & = \ord( d^{-1/2} ) , \label{eq:estimate_w_1}  \\ 
\norm{\f w_2} & =\ord\bigg(\bigg( \log d + \frac{\log N}{d}\bigg)^{1/2} \bigg( 1+  \frac{\log N}{D_x}  \bigg)^{1/2} \bigg)  , \label{eq:estimate_w_2}  \\ 
\norm{\f w_3} & = \ord\bigg(\bigg(\frac{\log N}{D_x} \bigg)^{1/2}\bigg), \label{eq:estimate_w_3} \\ 
\norm{\f w_4} & =\ord\bigg(\bigg(\frac{d}{D_x-d} \bigg)^{({r}-2)/2}\bigg( \frac{d}{\sqrt{D_x -d}} + \bigg(\frac{\log N}{D_x} \bigg)^{1/2} \bigg)\bigg) \label{eq:estimate_w_4} 
\end{align} 
\end{subequations} 
hold with very high probability on  $\{d < D_x \leq \sqrt{N} (2d)^{-r}\}$  conditioned on $A_{(x)}$.
\end{lemma}

Lemma \ref{lem:estimates_w_k} is proved in Subsection~\ref{subsec:smallness_w_k} below.

\begin{proof}[Proof of Proposition \ref{pro:approximate_eigenvector}]
From Lemmas \ref{lem:decomposition_app_eigenvector} and \ref{lem:estimates_w_k} we get
\[ \norm{(\underline{A} - \sqrt{d} \Lambda(\alpha_x)) \f v} \leq \cal C
\bigg( \log d + \frac{\log N}{d}\bigg)^{1/2}\bigg( 1+ \frac{\log N}{D_x} \bigg)^{1/2}
+ \cal C\sqrt{d} \bigg(\frac{d}{D_x-d} \bigg)^{({r}-1)/2}
\] 
with very high probability on $\{ 2d < D_x \leq \sqrt{N} (2d)^{-r}\}$. Write $D_x = (2 + t) d$ for $t > 0$. To conclude the proof, it suffices to show that
\begin{equation} \label{t_cond}
\sqrt{d} \bigg(\frac{d}{D_x-d} \bigg)^{({r}-1)/2} \leq 1\qquad \text{i.e.} \qquad
\log (1+t) \geq \frac{\log d}{r - 1}.
\end{equation}
Since by assumption $\log d \leq r$, this condition is satisfied provided that
\begin{equation*}
2 \frac{\log d}{r} \leq t = \frac{D_x}{d} - 2.
\end{equation*}
This concludes the proof for $\f v$. For $\f v_-$ the bound follows in the same way from trivial modifications 
of Lemma~\ref{lem:decomposition_app_eigenvector} and Lemma~\ref{lem:estimates_w_k} obtained by replacing $u_i$ 
by $(-1)^i u_i$. We leave the details of these modifications to the reader. 
\end{proof}

\subsection[Proof of Lemma~\ref{lem:decomposition_app_eigenvector}]{Decomposition of {$(\underline{A} - \sqrt{d}\Lambda(\alpha_x))\mathbf{v}$} -- Proof of Lemma~\ref{lem:decomposition_app_eigenvector}} 
\label{subsec:proof_decomposition_app_eigenvector}

In this subsection, we prove Lemma~\ref{lem:decomposition_app_eigenvector}.
We recall that $\f w_0$, \ldots, $\f w_4$ were defined in \eqref{eq:def_w_k} and $N_i(y)$ in \eqref{eq:def_e_N_i}.

\begin{proof}[Proof of Lemma~\ref{lem:decomposition_app_eigenvector}]
Recalling the definition $\underline{A} = A - \E A$, we have
\begin{equation*}
 \underline{A} \f v  = A \f v + \frac{d}{N} \f v- d \scalar{\f e}{\f v} \f e 
   = \f w_0 + \sum_{i=0}^{r} \frac{u_i}{\sqrt{\abs{S_i}}} A \f 1_{S_i}. 
\end{equation*}
By the definition \eqref{eq:def_e_N_i} of $N_i(y)$ and the triangle inequality for the graph distance, we have
\begin{equation*}
A \f 1_{S_i} = \ind{i \geq 1} \sum_{y \in S_{i-1}} N_i(y) \f 1_y + \sum_{y \in S_i} N_i(y) \f 1_y + \sum_{y \in S_{i+1}} N_i(y) \f 1_y,
\end{equation*}
so that
\begin{align*}
 \underline{A} \f v  & = \f w_0 + \sum_{i=1}^{r} \frac{u_i}{\sqrt{\abs{S_i}}} \sum_{y \in S_{i-1}} N_i(y) \f 1_y + \sum_{i=0}^{r} \frac{u_i}{\sqrt{\abs{S_i}}} \bigg[ \sum_{y \in S_i} N_i(y) \f 1_y + \sum_{y \in S_{i+1}} N_i(y) \f 1_y \bigg] \\ 
 & = \f w_0 + \f w_1 + \sum_{i=1}^{r} \frac{u_i}{\sqrt{\abs{S_i}}} \bigg[\sum_{y \in S_{i-1}} N_i(y) \f 1_y + \f 1_{S_{i+1}} \bigg] + u_0 \f 1_{S_1} \\ 
 & = \f w_0 + \f w_1 + \f w_2 + \sum_{i=1}^{r} \frac{u_i}{\sqrt{\abs{S_i}}} \bigg[ \sum_{y \in S_{i-1}} \frac{\abs{S_i}}{\abs{S_{i-1}}} \f 1_y + \f 1_{S_{i+1}} \bigg] + u_0 \f 1_{S_1}.
\end{align*} 
Thus, we conclude
\begin{equation}\label{eq:under_A_v_minus_w0_w1_w2}
\begin{aligned} 
\underline{A}\f v -\sum_{k=0}^2 \f w_k
& =  u_0 \f 1_{S_1} + \sum_{i=1}^{r} \frac{u_i}{\sqrt{\abs{S_i}}} \bigg[ \frac{\abs{S_i}}{\abs{S_{i-1}}} \f 1_{S_{i-1}} + \f 1_{S_{i+1}} \bigg].
\\
& =  u_0 \sqrt{\abs{S_1}} \, \f s_1 + u_1 \sqrt{\abs{S_1}} \, \f s_0 + u_2 \frac{\sqrt{\abs{S_2}}}{\sqrt{\abs{S_1}}} \, \f s_1 + \sum_{i=2}^{{r}-1}  
\bigg( u_{i+1} \frac{\sqrt{\abs{S_{i+1}}}}{\sqrt{\abs{S_i}}}  + u_{i-1} \frac{\sqrt{\abs{S_i}}}{\sqrt{\abs{S_{i-1}}}}  \bigg) \f s_i \\ 
 & \phantom{=} + u_{r}\frac{\sqrt{\abs{S_{{r}+1}}}}{\sqrt{\abs{S_{r}}}} \, \f s_{{r}+1} + u_{{r}-1} \frac{\sqrt{\abs{S_{r}}}}{\sqrt{\abs{S_{{r}-1}}}} \, \f s_{r}. 
\end{aligned} 
\end{equation}

Since $\sqrt{d} \Lambda(\alpha_x) = \frac{D_x}{\sqrt{D_x - d}}$, from the definition of $u_i$ in \eqref{eq:def_v_x_coeff} we get 
\[ \sqrt{d} \Lambda(\alpha_x) u_0 = \sqrt{D_x} u_1, \qquad \sqrt{d} \Lambda(\alpha_x) u_1 = \sqrt{D_x} u_0 + \sqrt{d} u_2, \qquad \sqrt{d} \Lambda(\alpha_x) u_i = \sqrt{d} u_{i-1} + \sqrt{d}u_{i+1} \] 
for all $i=2,3,\ldots, {r}$.
This implies 
\[ \sqrt{d} \Lambda(\alpha_x) \f v = u_1 \sqrt{\abs{S_1}} \, \f s_0 + \pB{u_0 \sqrt{\abs{S_1}} + u_2 \sqrt{d}} \f s_1 + \sum_{i=2}^{r} \pB{ u_{i-1} \sqrt{d} + u_{i+1} \sqrt{d}} \f s_i.  \] 
Together with \eqref{eq:under_A_v_minus_w0_w1_w2}, this yields 
\[ \begin{aligned} 
& (\underline{A} - \sqrt{d} \Lambda(\alpha_x))\f v - \sum_{k=0}^2 \f w_k 
 \\
 &\quad\qquad = u_2 \bigg( \frac{\sqrt{\abs{S_2}}}{\sqrt{\abs{S_1}}} - \sqrt{d} \bigg) \f s_1 
+ \sum_{i=2}^{{r} - 1}  \left[ u_{i+1} \left(\frac{\sqrt{\abs{S_{i+1}}}}{\sqrt{\abs{S_i}}} -\sqrt{d} \right)  
+ u_{i-1} \left( \frac{\sqrt{\abs{S_i}}}{\sqrt{\abs{S_{i-1}}}} - \sqrt{d} \right)  \right] \f s_i 
\\
&\quad\qquad \phantom{=}
+  \left( u_{{r}-1} \frac{\sqrt{\abs{S_{r}}}}{\sqrt{\abs{S_{{r}-1}}}}  - u_{{r}-1} \sqrt{d}  - u_{{r}+1} \sqrt{d} \right) \f s_{r}
 + u_{r} \frac{\sqrt{\abs{S_{{r}+1}}}}{\sqrt{\abs{S_{r}}}} \, \f s_{{r} + 1} \\
&\quad\qquad = \f w_3 +  \f w_4,
\end{aligned} \]
which concludes the proof.
\end{proof}

\subsection[Proof of Lemma~\ref{lem:estimates_w_k}]{Smallness of $\mathbf{w}_0$, \ldots, $\mathbf{w}_4$ -- Proof of Lemma~\ref{lem:estimates_w_k}} 
\label{subsec:smallness_w_k}

This subsection is devoted to the proof of Lemma~\ref{lem:estimates_w_k}.  

In order to estimate $\f w_0, \ldots,\f w_4$, we shall make frequent use of the following lemma. 
We recall from Section~\ref{sec:notations} that $A_{(x)} = (A_{ij})_{i = x \text{ or } j = x}$. 

\begin{lemma}[Concentration of $\abs{S_{i+1}}/\abs{S_i}$]\label{lem:concentration_S_i}
Let $1 \leq d \leq N^{1/4}$. 
\begin{enumerate}[label=(\roman*)]
\item For $r,i \in \N$ satisfying $1 \leq i \leq r$ we have 
\begin{subequations} 
\begin{equation}
\absa{\frac{\abs{S_{i+1}}}{d\abs{S_i}} - 1 }  = \ord \bigg(\bigg( \frac{\log N}{d \abs{S_i}} \bigg)^{1/2}\bigg) \label{eq:proof_w_3_aux_estimate}
\end{equation}
and
\begin{equation} \label{eq:concentration_S_i}
\abs{S_i} = D_x d^{i-1}\bigg(1 + \ord\bigg( \bigg( \frac{\log N}{dD_x}\bigg)^{1/2} \bigg) \bigg) 
\end{equation}
\end{subequations} 
with very high probability on
\begin{equation} \label{D_assumption}
\hbb{\cal K \frac{\log N}{d}\leq D_x \leq \sqrt{N}(2d)^{-r}}
\end{equation}
conditioned on $A_{(x)}$.
\item Moreover, for all $r, i \in \N$ satisfying $1 \leq i \leq r$, the bound  
\begin{equation} \label{eq:upper_bound_S_i_without_lower_bound_D_x} 
 \abs{S_{i+1}} \leq d \abs{S_i} + \Cnu (d \abs{S_i} \log N)^{1/2} 
\end{equation} 
holds with very high probability on $\{ D_x \leq \sqrt{N} (2d)^{-r}\}$ conditioned on $A_{(x)}$. 
\end{enumerate} 
\end{lemma} 

In the applications of Lemma \ref{lem:concentration_S_i} below, we shall always work under the assumptions $d \geq \cal K \sqrt{\log N}$ and $D_x > d$, which imply that the lower bound in \eqref{D_assumption} is always satisfied.
Before proving Lemma~\ref{lem:concentration_S_i}, we first conclude \eqref{eq:estimate_w_0}, \eqref{eq:estimate_w_3} and \eqref{eq:estimate_w_4} from it. 

\begin{proof}[Proof of \eqref{eq:estimate_w_0}, \eqref{eq:estimate_w_3} and \eqref{eq:estimate_w_4}] 
In the whole proof, we exclusively work on the event $\{d < D_x \leq \sqrt{N} (2d)^{-r} \} $. 
For the proof of \eqref{eq:estimate_w_0}, we start by using the Cauchy-Schwarz inequality to obtain
\begin{equation} \label{eq:proof_w_0_aux} 
 \norm{\f w_0}  \leq \frac{d}{N} + \frac{d}{\sqrt N} \sum_{y \in B_{r}} \abs{\f v(y)} 
\leq \frac{d}{N} + \frac{d}{\sqrt{N}} \sqrt{\abs{B_{r}}}. 
\end{equation} 
From Lemma~\ref{lem:concentration_S_i}, we conclude with very high probability
\begin{equation*}
\abs{B_{r}} \leq 1 + \sum_{i = 0}^{r - 1} D_x (2d)^{i} \leq 2 D_x (2 d)^{r - 1} \leq (2 D_x)^{r} \leq (2 D_x)^{r_x} \leq \sqrt{N}
\end{equation*}
by definition of $r_x$. Hence \eqref{eq:estimate_w_0} follows.

We now turn to the proof of \eqref{eq:estimate_w_3}. Estimating the definition of $\f w_3$ yields 
\[\begin{aligned} 
  \norm{\f w_3}^2  & \leq  d \left[ \left( \frac{\sqrt{\abs{S_2}}}{\sqrt{d \abs{S_1}}} -1 \right)^2 u_2^2 + 2 \sum_{i=2}^{{r}-1} \left( \left(\frac{\sqrt{\abs{S_{i+1}}}}{\sqrt{d \abs{S_i}}} - 1\right)^2 u_{i+1}^2 + \left( \frac{\sqrt{\abs{S_i}}}{
\sqrt{d \abs{S_{i-1}}}} - 1 \right)^2 u_{i-1}^2 \right) \right] \\ 
 & \leq \Cnu \frac{\log N}{D_x} \bigg[ u_2^2 + 2 \sum_{i=2}^{{r}-1} \big(u_{i+1}^2 + u_{i-1}^2\big) \bigg] 
\end{aligned} \] 
with very high probability conditioned on $A_{(x)}$, 
where we used \eqref{eq:proof_w_3_aux_estimate}  and \eqref{eq:concentration_S_i} in the last step. 
This completes the proof of \eqref{eq:estimate_w_3}.  

For the proof of \eqref{eq:estimate_w_4}, we apply the triangle inequality to $\f w_4$ to obtain 
\begin{equation} \label{eq:triangle_w_4}  
\norm{\f w_4} \leq \abs{u_{{r}-1}} \absa{ \frac{\sqrt{\abs{S_{r}}}}{\sqrt{\abs{S_{{r}-1}}}} - \sqrt{d}} + \abs{u_{{r}+1}} \sqrt{d} + \abs{u_{r}} \frac{\sqrt{\abs{S_{{r}+1}}}}{\sqrt{\abs{S_{r}}}}.
\end{equation} 
In case $D_x \leq 2d$, this implies  
\eqref{eq:estimate_w_4} directly by $\abs{u_i} \leq 1$, Lemma~\ref{lem:concentration_S_i} and $\sqrt{d} \leq d/\sqrt{D_x- d}$.  
If $D_x \geq 2d$ then using Lemma~\ref{lem:concentration_S_i} we conclude from \eqref{eq:triangle_w_4} that  
\[ \norm{\f w_4} 
 \leq \frac{d^{({r}-2)/2}}{(D_x-d)^{({r}-2)/2}}\bigg( \frac{d}{\sqrt{D_x -d}} +  \Cnu \bigg(\frac{\log N}{D_x} \bigg)^{1/2} \bigg), 
\] 
where we used the definition of $u_{{r}-1}$ and $u_{r}$ as well as $\abs{u_1} \leq 1$ in the last step. This shows the bound \eqref{eq:estimate_w_4}.
\end{proof}

\begin{proof}[Proof of Lemma~\ref{lem:concentration_S_i}]
Define
\begin{equation*}
\cal E_i \deq \frac{d \abs{S_i}}{N} + \frac{1}{\sqrt{N}}.
\end{equation*}
We shall prove below that there are constants $C,c > 0$ such that if there exists $i \geq 1$ 
satisfying $\abs{B_i} \leq \sqrt{N}$ then, for all $\eps \in [0,1]$, we have  
\begin{equation} \label{eq:concentration_S_i_plus_1_general_eps} 
 \P \Big( (1-\eps - C \cal E_i) d \abs{S_i} \leq \abs{S_{i+1}} \leq (1 + \eps + C \cal E_i) d \abs{S_i} \Bigm\vert A_{(B_{i-1})} \Big) \geq 1 - 2 \exp\Big( - c d\abs{S_i} \eps^2\Big).  
\end{equation} 
From \eqref{eq:concentration_S_i_plus_1_general_eps}, 
we now conclude \eqref{eq:proof_w_3_aux_estimate} and
\begin{equation}
D_x \bigg(\frac{d}{2}\bigg)^{i-1} \leq \abs{S_i} \leq D_x (2d)^{i - 1}, \label{eq:lower_bound_S_i}
\end{equation}
for $1 \leq i \leq r$ simultaneously by induction. 
For $i=1$ we choose $\eps^2 = \Cnu \log N/(dD_x)$, which is bounded by $1$ for $\cal K$ in \eqref{D_assumption} large enough. Since $d \geq 1$ the upper bound on $D_x$ from \eqref{D_assumption} implies $\cal E_1 \leq \epsilon$. 
Thus, we obtain
 \eqref{eq:proof_w_3_aux_estimate} directly from \eqref{eq:concentration_S_i_plus_1_general_eps}.  
The estimate \eqref{eq:lower_bound_S_i} is trivial for $i = 1$.

For the induction step, we assume that $\cal K$ in \eqref{D_assumption} is large enough that the right-hand side of \eqref{eq:proof_w_3_aux_estimate} for $i = 1$ is less than $1/2$. Suppose first that with very high probability \eqref{eq:proof_w_3_aux_estimate} and \eqref{eq:lower_bound_S_i} hold up to $i$. Since $\abs{S_i} \geq D_x$ by \eqref{eq:lower_bound_S_i}$_i$, we conclude from \eqref{eq:proof_w_3_aux_estimate}$_{i}$ that \eqref{eq:lower_bound_S_i}$_{i+1}$ holds.
Next, suppose that with very high probability \eqref{eq:proof_w_3_aux_estimate} holds up to $i$ and  \eqref{eq:lower_bound_S_i} up to $i + 1$. By \eqref{eq:lower_bound_S_i}$_{i+1}$, we deduce that for $i+1 \leq r$ we have $\abs{B_{i+1}} \leq D_x (2d)^r \leq \sqrt{N}$, where the last inequality follows by assumption on $D_x$. Hence, we may apply \eqref{eq:concentration_S_i_plus_1_general_eps}$_{i+1}$ to estimate $\abs{S_{i+2}}$, with the choice $\eps^2 = \Cnu \log N/(d\abs{S_{i + 1}})$ with the same $\cal C$ as in the first induction step. 
From $d \geq 1$, the upper bound on $D_x$ in \eqref{D_assumption}, $i + 1 \leq r$ and 
\eqref{eq:lower_bound_S_i}$_{i+1}$, we obtain 
$\cal E_{i + 1} \leq \epsilon$, and hence we conclude \eqref{eq:proof_w_3_aux_estimate}$_{i+1}$ 
after taking the conditional expectation with respect to $A_{(x)}$.
Note that the necessary union bounds are affordably since the right-hand side of \eqref{eq:concentration_S_i_plus_1_general_eps} is always at least $1 - 2 N^{-c \cal C}$. 
 
The expansion in \eqref{eq:concentration_S_i} is a direct consequence of 
\begin{equation}\label{eq:proof_concentration_S_i_aux1} 
 D_x d^{i-1}(1- \eps_i) \leq \abs{S_{i}} \leq D_x d^{i-1}(1 +\eps_i), \qquad \qquad \eps_i \defeq 2 \Cnu \bigg(\frac{\log N}{dD_x}\bigg)^{1/2} \sum_{j=0}^{i-2} d^{-j/2}, 
\end{equation}
for $1 \leq i \leq r$ with very high probability on $\{ \mathcal{K} \log N/d \leq D_x \leq \sqrt{N}(2d)^{-r}\}$ conditioned on $A_{(x)}$ as well as the fact that the geometric sum 
in the definition of $\eps_i$ is bounded by $2$ uniformly in $i$. 

The estimates in \eqref{eq:proof_w_3_aux_estimate} and \eqref{eq:lower_bound_S_i} imply \eqref{eq:proof_concentration_S_i_aux1} by induction as follows. 
The case $i=1$ is trivial. For the induction step, we conclude from \eqref{eq:proof_w_3_aux_estimate} that 
\[ \abs{S_{i+1}} \geq d \abs{S_i}\bigg(1 - \Cnu \bigg(\frac{\log N}{d \abs{S_i}}\bigg)^{1/2}\bigg) \geq D_x d^{i} \bigg( 1 - \eps_i - \Cnu \bigg(\frac{\log N}{d^i D_x (1 - \eps_i)}\bigg)^{1/2} \bigg). 
\] 
Here, we used $\abs{S_i} \neq 0$ by \eqref{eq:lower_bound_S_i} in the first step and the induction hypothesis in the second step. 
As $\eps_i \leq 3/4$ for sufficiently large $\mathcal{K}$ due to $D_x \geq \mathcal{K} \log N/d$ we obtain the lower bound in \eqref{eq:proof_concentration_S_i_aux1}. 
The upper bound is proved completely analogously. This completes the proof of \eqref{eq:concentration_S_i}.

Next, we also conclude \eqref{eq:upper_bound_S_i_without_lower_bound_D_x} from \eqref{eq:concentration_S_i_plus_1_general_eps} by an induction argument. 
For $i = 1$, we can assume that $D_x = \abs{S_1} \neq 0$ as there is nothing to show otherwise. If $\abs{S_1} \neq 0$ then we choose $\eps^2 = \Cnu \log N /(d \abs{S_1})$ in \eqref{eq:concentration_S_i_plus_1_general_eps}. As $\cal E_1 \leq \eps$ this directly implies \eqref{eq:upper_bound_S_i_without_lower_bound_D_x} for $i=1$. In the induction step we assume $\abs{S_{i+1}} \neq 0$ as \eqref{eq:upper_bound_S_i_without_lower_bound_D_x}$_{i+1}$ is trivial for $\abs{S_{i+1}} = 0$. The induction hypothesis and $D_x \leq \sqrt{N} (2d)^{-r}$ imply $\abs{B_{i+1}} \leq \sqrt{N}$. 
Hence, $\abs{S_{i+1}} \neq 0$ allows the choice $\eps^2 = \Cnu \log N/(d \abs{S_{i + 1}})$ in \eqref{eq:concentration_S_i_plus_1_general_eps} in order to bound $\abs{S_{i+2}}$. 
As $\mathcal E_{i+1} \leq \eps$ we obtain \eqref{eq:upper_bound_S_i_without_lower_bound_D_x}$_{i+1}$ by taking the conditional expectation with respect to $A_{(x)}$.

What remains therefore is the proof of \eqref{eq:concentration_S_i_plus_1_general_eps}. We condition on $A_{(B_{i - 1})}$ and suppose that $\abs{B_i} \leq \sqrt{N}$. (Note that $B_i$ is measurable with respect to $A_{(B_{i - 1})}$.) Let us compute the law of $\abs{S_{i+1}}$ conditioned on $A_{(B_{i-1})}$. For $y \in B_i^c$ denote by $F_y = \ind{\text{$y$ adjacent to $S_i$}}$. Then, conditioned on $A_{(B_{i-1})}$, $(F_y)_{y \in B_i^c}$ are i.i.d.\ Bernoulli random variables with expectation $1 - (1 - p)^{\abs{S_i}}$, where $p = d/N$. Thus,
\begin{equation*}
\abs{S_{i + 1}} = \sum_{y \in B_i^c} F_y \eqdist \op{Binom}\Big(1-(1-p)^{\abs{S_i}}, N-\abs{B_i}\Big)
\end{equation*}
conditioned on $A_{(B_{i-1})}$. Thus,
\begin{equation} \label{exp_Si}
E_{i+1} \deq \E \qb{\abs{S_{i+1}} \big\vert A_{(B_{i-1})}} = \pb{1 - (1 - p)^{\abs{S_i}}} (N - \abs{B_i}) = d \abs{S_i} \pb{1 + O (\cal E_i)}, 
\end{equation}
where, in the last step, we used that $1 -p= \ee^{-p +  O(p^2)}$ as well as the assumptions $\abs{B_i} \leq \sqrt{N}$ and $d \leq N^{1/4}$. 
We recall the definition of $h$ from \eqref{eq:def_h}.  
Applying Bennett’s inequality yields 
\begin{align*}
\P \Big(\absb{\abs{S_{i+1}} - E_{i+1}} > \eps \, E_{i+1} \Bigm\vert A_{(B_{i-1})} \Big) 
&\leq 2 \exp \Big( - E_{i+1} \min \{ h(1+\eps), h(1-\eps)\} \Big)
\\
&\leq 2 \exp \Big( - \frac{1}{2} d \abs{S_i} \min \{ h(1+\eps), h(1-\eps)\} \Big).
\end{align*}
where we used \eqref{exp_Si} and that $\cal E_i = O(N^{-1/4})$ as follows from the assumptions $\abs{B_i} \leq \sqrt{N}$ and $d \leq N^{1/4}$. Now the claim \eqref{eq:concentration_S_i_plus_1_general_eps} follows from \eqref{exp_Si} and the observation that there exists a $c > 0$ such that $\min\{ h(1+ \eps), h(1-\eps)\} \geq c \eps^2$ for all $\eps \in [0,1]$.
\end{proof} 

\begin{lemma}[Few cycles in small balls] \label{lem:tree_approximation}
For $k, r \in \N$ we have 
\begin{equation} \label{eq:tree_approx}
\P \Big(\abs{E(G\vert_{B_r})} - \abs{B_{r}} +1 \geq k \Bigm\vert S_1 \Big) \leq
\frac{1}{N^k} \pb{C (d + \abs{S_1})}^{2kr + k} (2kr)^{2k}.
\end{equation}
\end{lemma} 

\begin{corollary} \label{cor:tree_approximation}
If $d \geq C$ for some universal constant $C>0$ then the number of cycles in $G|_{B_r}$ is bounded,  
\[ \abs{\{ \text{\normalfont{cycles in }} G|_{B_r} \}} = \ord(1), \] 
with very high probability on $\{d < D_x \leq N^{1/4}\} \cap \{r \leq r_x\}$ conditioned on $S_1$.
\end{corollary}

\begin{proof} 
Given $\nu >0$, it is easy to conclude from \eqref{eq:tree_approx} that there is $k \in \N$ 
such that the number of cycles in $G|_{B_r}$ is bounded by $k$ with very high probability
on the event $\{ d < D_x \leq N^{1/4}\} \cap \{ r \leq r_x\}$ 
by the lower bound on $d$ and the definition of $r_x$ in \eqref{eq:def_r_star}. 
\end{proof} 

\begin{corollary}\label{cor:corollary_tree_approximation} 
With very high probability on $\{D_x > d\}$ conditioned on $S_1$, we have for all $i \leq r_x + 1$
\begin{equation} 
 \abs{S_i}  = \sum_{y \in S_i} N_{i-1} (y) + \ord(1), \qquad
 \abs{S_i}  = \sum_{y \in S_{i-1}} N_{i} (y) + \ord(1), \label{eq:rel_S_i_sum_N_i} 
\end{equation}
as well as
\begin{equation} \label{eq:Nloops}
\sum_{y \in S_i} N_i(y) = \ord(1). 
\end{equation}
\end{corollary}

\begin{proof}
By choosing a spanning tree of $G\vert_{B_{r_x}}$, we conclude that, with very high probability on $\{ D_x > d\}$  conditioned on $S_1$, we can find $\ord(1)$ edges of $G\vert_{B_{r_x}}$ such that removing them yields a tree on the vertex set $B_{r_x}$. Now \eqref{eq:rel_S_i_sum_N_i} follows easily by noting that
\begin{equation*}
\scalar{\f 1_{S_{i - 1}}}{A \f 1_{S_i}} = \sum_{y \in S_{i - 1}} N_i(y) = \sum_{y \in S_i} N_{i - 1}(y),
\end{equation*}
and that the left-hand side equals $\abs{S_i}$ if $A$ is the adjacency matrix of a tree. Finally, \eqref{eq:Nloops} follows by noting that its left-hand side vanishes if $A$ is the adjacency matrix of a tree.
\end{proof}

Before the proof of Lemma~\ref{lem:tree_approximation}, we show how 
\eqref{eq:rel_S_i_sum_N_i} and \eqref{eq:Nloops} are used to bound $\norm{\f w_1}$ and establish \eqref{eq:estimate_w_1}. 
 
\begin{proof}[Proof of \eqref{eq:estimate_w_1}]
Since $S_0 = \{ x\}$ we have that $N_0(y) = 1$ for all $y \in S_1$. Moreover, $N_0(x) = 0$ since $G$ has no loops. Hence,
\begin{equation} \label{eq:w_1_computed} 
 \f w_1 = \sum^{{r}}_{i= 1} \frac{u_i}{\sqrt{\abs{S_i}}} \bigg( \sum_{y \in S_{i+1}} \big(N_i(y) -1\big) \f 1_y +\sum_{y \in S_i} N_i(y) \f 1_y\bigg). 
\end{equation}
By the triangle inequality, $\abs{u_i} \leq 1$, and the fact that $N_i(y) \geq 1$ for all $y \in S_{i+1}$, we find
\begin{equation*}
\norm{\f w_1} \leq \sum^{{r}}_{i= 1} \frac{1}{\sqrt{\abs{S_i}}} \bigg( \sum_{y \in S_{i+1}} \big(N_i(y) -1\big) +\sum_{y \in S_i} N_i(y) \bigg) = \sum^{{r}}_{i= 1} \frac{1}{\sqrt{\abs{S_i}}} \ord(1),
\end{equation*}
where in the last step we used \eqref{eq:rel_S_i_sum_N_i} and \eqref{eq:Nloops}.
The claim follows using \eqref{eq:concentration_S_i}. 
\end{proof}

\begin{proof}[Proof of Lemma~\ref{lem:tree_approximation}]
Throughout the proof we condition on $S_1$.
Let $r \leq N$ and $k \in \N$, and without loss of generality $r,k \geq 1$. Define the set $\cal H_k$ as the set of connected graphs $H$ satisfying $x \in V(H) \subset [N]$, $S_{1}^H \subset S_1$, $\abs{E(H)} = \abs{V(H)} - 1 + k$, and $\abs{V(H)} \leq 2kr + 1$.
Let $H \in \cal H_k$. Then
\begin{equation*}
\P (E(H) \subset E(G) \cond S_1) = \pbb{\frac{d}{N}}^{\abs{E(H)} - \abs{S_{1}^H}} = \pbb{\frac{d}{N}}^{\abs{V(H)} - 1 + k - \abs{S_{1}^H}}.
\end{equation*}
Hence, by a union bound,
\begin{equation*}
\P \pb{\exists H \in \cal H_k ,\, E(H) \subset E(G) \cond S_1} \leq \sum_{H \in \cal H_k} \P (E(H) \subset E(G) \cond S_1).
\end{equation*}
In the sum over $H \in \cal H_k$, we shall sum first over the set of vertices $S_1^H$, the set of vertices $V(H) \setminus (S_1^H \cup \{x\})$, and then over all graphs $H$ on the vertex set $V(H)$. 
Writing $q_1 \deq \abs{S_1^H}$ and $q_2 \deq \abs{V(H) \setminus (S_1^H \cup \{x\})}$, so that $\abs{V(H)} = q_1 + q_2 + 1$, 
 we find
\begin{equation*}
\P \pb{\exists H \in \cal H_k ,\, E(H) \subset E(G) \cond S_1} \leq \sum_{0 \leq q_1 + q_2 \leq 2kr} \binom{\abs{S_1}}{q_1} \binom{N - \abs{S_1} - 1}{q_2} C_{q_1 + q_2 + 1, k} \pbb{\frac{d}{N}}^{q_2 + k},
\end{equation*}
where $C_{q,k}$ is the number of connected graphs on $q$ vertices with $q -1 + k$ edges. To estimate $C_{q,k}$, we note that each such graph can be written as a union of a tree on $q$ vertices and $k$ additional edges. By Cayley's theorem on the number of trees, we therefore conclude that
\begin{equation*}
C_{q,k} \leq q^{q - 2} q^{2k} = q^{q + 2 k - 2}.
\end{equation*}
Putting everything together, we conclude that
\begin{align*}
\P \pb{\exists H \in \cal H_k ,\, E(H) \subset E(G) \cond S_1}
&\leq \sum_{0 \leq q_1 + q_2 \leq 2kr} \frac{\abs{S_1}^{q_1}}{q_1 !} \frac{N^{q_2}}{q_2 !} (q_1 + q_2 + 1)^{q_1 + q_2 + 2k - 1} \pbb{\frac{d}{N}}^{q_2 + k}
\\
& = \frac{d^k}{N^k}\sum_{0 \leq i \leq 2kr} \frac{1}{i!}(|S_1|+d)^{i} (i+1)^{i+2k-1} \\
&\leq \frac{1}{N^k} \pb{C (d + \abs{S_1})}^{2kr + k} (2kr)^{2k},
\end{align*}
where in the second step we used the binomial theorem and in the last step Stirling's approximation. 

In order to conclude the proof, it suffices to show that
\begin{equation} \label{sp_tree_est}
\hb{\abs{E(G\vert_{B_r})} - \abs{B_r} + 1 \geq k} \subset \hb{\exists H \in \cal H_k ,\, E(H) \subset E(G)}.
\end{equation}
To show \eqref{sp_tree_est}, we suppose that $G$ is in the event on the left-hand side of \eqref{sp_tree_est}. Let $T$ be a spanning tree of $B_r$ such that $d_T(x,y) = d(x,y)$ for all $y \in B_r$, where $d_T$ is the graph distance on $T$. Since $\abs{E(G\vert_{B_r})} - \abs{B_r} + 1 \geq k$, we can find $k$ edges of $G\vert_{B_r}$ that are not edges of $T$; denote these edges by $E_1$. Let $U_1$ denote the vertices incident to the edges of $E_1$. Let $U_2$ denote the vertices in all the (unique) paths of $T$ connecting the vertices of $U_1$ to $x$, and $E_2$ the edges of these paths. Consider now the graph $H$ with vertex set $U_1 \cup U_2 \cup \{x\}$ and edge set $E_1 \cup E_2$.
See Figure~\ref{fig:tree_approx} for an illustration.

\begin{figure}[!ht]
\begin{center}
{\small 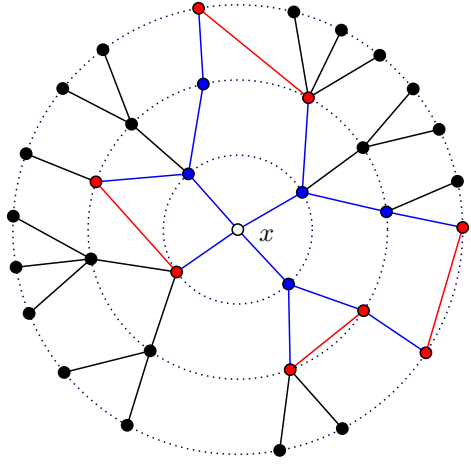}
\end{center}
\caption{Graphical representation of the proof of \eqref{sp_tree_est}. We draw the ball $B_r$ for $r = 3$. The spanning tree $T$ is drawn using black and blue edges, and the spheres of radii $1,2,3$ are drawn using dots. The red edges are $E_1$, the blue edges $E_2$, the red vertices $U_1$, and the blue vertices $U_2 \setminus U_1$. \label{fig:tree_approx}}
\end{figure}

We claim that $H \in \cal H_k$, which will conclude the proof of \eqref{sp_tree_est}. The only non-obvious property to verify is that $\abs{V(H)} = \abs{U_1 \cup U_2 \cup \{x\}} \leq 2kr + 1$. This follows easily from the observation that $\abs{U_1} \leq 2k$ and that each of the above paths has at most $r - 1$ vertices in $[N] \setminus (U_1 \cup \{x\})$, and there are at most $2k$ paths. This concludes the proof.
\end{proof}

\begin{proof}[Proof of \eqref{eq:estimate_w_2}] 
We define $\Xi_i \defeq \{d <  D_x \leq \sqrt{N}(2d)^{-i} \}$ for $i \in \N$. 
We start by noting that 
 \[ \f w_2 = \sum_{i=2}^{{r}} \frac{u_i}{\sqrt{\abs{S_i}}} \sum_{y \in S_{i-1}} \left( N_i(y) - \frac{\abs{S_i}}{\abs{S_{i-1}}} \right) \f 1_y,  \] 
since $S_0 = \{ x\}$ and $N_1(x) = \abs{S_1} = \abs{S_1}/\abs{S_0}$. 

We apply the Pythagorean theorem, use \eqref{eq:rel_S_i_sum_N_i}, $\abs{S_i} \geq D_x d/2$ uniformly for $i \in \{2, \ldots, {r}\}$ 
with very high probability  on $\Xi_r$ conditioned on $A_{(x)}$ by \eqref{eq:concentration_S_i} in Lemma~\ref{lem:concentration_S_i} 
as well as $\sum_{i=1}^{r - 1} u_{i+1}^2 \leq 1$ and obtain
\[ \begin{aligned} 
 \norm{\f w_2}^2  & = \sum_{i=1}^{r - 1} \frac{u_{i+1}^2}{\abs{S_{i+1}}} \sum_{y \in S_{i}} \left(N_{i+1}(y) - \frac{\abs{S_{i+1}}}{\abs{S_{i}}} \right)^2 \\ 
 & \leq  2 \sum_{i=1}^{r - 1} \frac{u_{i+1}^2}{\abs{S_{i+1}}} \sum_{y \in S_{i}} \bigg( N_{i+1}(y) - \frac{1}{\abs{S_{i}}} \sum_{y\in S_{i}} N_{i+1}(y) \bigg)^2 + \frac{\Cnu}{D_x d} \\ 
 & \leq  \frac{4}{d} \max_{1 \leq i \leq r - 1} (Z_i + Y_i^2)+ \frac{\Cnu}{D_x d},
\end{aligned} \] 
where we defined
\begin{equation*}
Z_i \deq \frac{1}{\abs{S_{i}}} \sum_{y \in S_{i}} \pB{ N_{i+1}(y) - \E \qb{N_{i+1}(y) \condb A_{(B_{i-1})}}}^2, \qquad
Y_i \deq \frac{1}{\abs{S_{i}}} \sum_{y \in S_{i}} \pB{N_{i+1}(y) - \E \qb{N_{i+1}(y) \condb A_{(B_{i-1})}}}.
\end{equation*}
Here we used Young's inequality and the fact that $\E \qb{N_{i+1}(y) \condb A_{(B_{i-1})}}$ does not depend on $y \in S_{i}$. In fact, conditioned on $A_{(B_{i - 1})}$, the random variables $(N_{i+1}(y))_{y \in S_{i}}$ are i.i.d.\ with law $\op{Binom}(N - \abs{B_{i}}, d/N)$.

The term $Y_i$ can be easily estimated by Bennett's inequality for $\op{Binom}\pb{\abs{S_{i}} (N - \abs{B_{i}}), d/N}$, which yields $Y_i^2 \leq \cal C \frac{d \log N}{D_x}$
with very high probability on  $\Xi_r$ conditioned on $A_{(x)}$. Here we used that, by Lemma \ref{lem:concentration_S_i}, $N - \abs{B_{i}} \geq N/2$ and $\abs{S_{i}} \geq D_x$, and, by definition of $\Xi_r$, $\frac{\log N}{d D_x} \leq 1$.

What remains, therefore, is the estimate of $Z_i$. We shall prove that for all $1 \leq i \leq r - 1$
\begin{equation} \label{eq:Z_claim}
Z_i \leq \Cnu d \pbb{1 + \frac{\log N}{\abs{S_i}}} \bigg( \log d + \frac{\log N}{d} \bigg)
\end{equation}
with very high probability on $\Xi_i$ conditioned on $A_{(x)}$, which will conclude the proof of the lemma.

The estimate \eqref{eq:Z_claim} can be regarded as a concentration result for the degrees of the vertices in $S_i$; indeed, by Lemma \ref{lem:tree_approximation} for any $y \in S_i$ we have $D_y = N_{i+1}(y) + \cal O(1)$ with very high probability. For any vertex $y \in S_i$ we have the variance estimate $\sqrt{\E (D_y - d)^2} \asymp C \sqrt{d}$. On the other hand, in the relevant regime $d \leq C \log N$, the estimate with very high probability (following from Bennett's inequality) $\abs{D_y - d} \leq \cal C \log N$ is much worse. Essentially, we need an estimate with very high probability of the average $\frac{1}{\abs{S_i}} \sum_{y \in S_i} \big(D_y - d \big)^{2}$, and the trivial bound $\cal C(\log N)^2$ obtained by applying the above estimate is much too large. Instead, we need to use that the typical term of $Z_i$ is much smaller than $(\log N)^2$. We do this using a dyadic classification of the degrees of the vertices in $S_i$.

For the proof of \eqref{eq:Z_claim}, we always condition on $A_{(x)}$ and work on the event $\Xi_i$.  
We abbreviate
\begin{equation*}
E_y \deq N_{i+1}(y) - \E \qb{N_{i+1}(y) \condb A_{(B_{i-1})}}
\end{equation*}
for $y \in S_i$, and introduce the level set sizes
\begin{equation} \label{eq:def_L_s_i}
 L_s^i \deq \absb{\big\{ y \in S_i \col E_y^2 > s^2 d^2 \big\}} 
\end{equation}
for any $s >0$. We have the probabilistic tail bound on $L_s^i$
\begin{equation} \label{eq:L_est}
\P \pb{ L_s^i \geq \ell \condb A_{(B_{i-1})} } \leq \begin{pmatrix} \abs{S_i} \\ \ell \end{pmatrix} \exp (-cd \ell (s \wedge s^2))
\end{equation}
for all $s > 0$ and $\ell \in \N$, with very high probability. Here $c > 0$ is a universal constant. To prove \eqref{eq:L_est}, we use a union bound to get
\begin{equation} \label{eq:proof_L_s_aux0}
\begin{aligned} 
 \P ( L_s^i \geq \ell\mid A_{(B_{i-1})} ) & \leq \sum_{T \subset S_i, ~\abs{T} = {\ell}} \P \Big( E_y^2 > (s d)^2 \text{ for all } y \in T \Bigm\vert A_{(B_{i-1})} \Big) \\ 
 & \leq \begin{pmatrix} \abs{S_i} \\ {\ell} \end{pmatrix} \max_{T \subset S_i, ~\abs{T} = \ell} 
\P \Big( E_y^2 > (s d)^2 
\text{ for all } y \in T \Bigm\vert A_{(B_{i-1})} \Big)
\\
& = 
\begin{pmatrix} \abs{S_i} \\ {\ell} \end{pmatrix}
\P \Big( E_y^2 > (s d)^2 \Bigm\vert A_{(B_{i-1})} \Big)^\ell
\end{aligned} 
\end{equation}
for some $y \in S_i$, since $(N_{i+1}(y))_{y \in S_{i}}$ are i.i.d.\ conditioned on $A_{(B_{i - 1})}$. By Bennett's inequality, we obtain
\begin{equation*}
\P \Big( E_y^2 > (s d)^2 \Bigm\vert A_{(B_{i-1})} \Big) \leq \exp (-cd (s \wedge s^2))
\end{equation*}
since $N - \abs{B_i} \geq N/2$ (by Lemma \ref{lem:concentration_S_i}), and hence \eqref{eq:L_est} follows.

Next we conclude the argument by establishing \eqref{eq:Z_claim}.
We decompose
\begin{equation} \label{eq:proof_concentration_degrees_aux1} 
\abs{S_i} Z_i =  \sum_{y \in S_i} E_y^2  \leq d  \abs{S_i} + \sum_{y \in S_i} E_y^2 \ind{E_y^2 \geq d}. 
\end{equation}
In order to estimate the second summand in \eqref{eq:proof_concentration_degrees_aux1}, we now establish the dyadic 
decomposition 
\begin{equation} \label{eq:proof_concentration_degrees_aux2} 
 \sum_{y \in S_i} E_y^2 \ind{E_y^2 \geq d} \leq \sum_{k=k_{\min}}^{k_{\max}} \sum_{y \in \mathcal{N}_k^i} E_y^2\leq d^2 \sum_{k=k_{\min}}^{k_{\max}} \ee^{k+1} \abs{\mathcal{N}_k^i}   
\end{equation}
with very high probability conditioned on $A_{(B_{i-1})}$, where we introduced 
\[ \begin{aligned} 
 \mathcal{N}_k^i & \defeq \Big\{ y \in S_i \col d^2 \ee^k < E_y^2 \leq \ee^{k+1} d^2 \Big \}, \qquad k \in \Z,\\ 
k_{\min} & \defeq -\ceil{\log(d)}, \\
 k_{\max} & \defeq \bigg\lceil \Cnu + \log\bigg(\frac{\log N}{d}\bigg) \vee \bigg(2 \log \bigg(\frac{\log N}{d}\bigg)\bigg)\bigg\rceil 
\end{aligned} \] 
with some possibly $\nu$-dependent constant $\Cnu>0$.

We now prove \eqref{eq:proof_concentration_degrees_aux2} by showing that, for all $y \in S_i$, $E_y^2 \leq d^2 \ee^{k_{\max}+1}$ with very high probability conditioned on $A_{(B_{i-1})}$.
To that end, we note that
\begin{equation*}
\P \pB{ \exists y \in S_i , E_y^2 > d^2 \ee^{k_{\max} + 1} \condB A_{(B_{i-1})}} = \P \pb{L_{s}^i \geq 1 \condb A_{(B_{i-1})} }, \qquad s = \ee^{(k_{\max} + 1)/2}.
\end{equation*}
Moreover, from \eqref{eq:L_est} with $\ell = 1$ and $\abs{S_i} \leq N$ we obtain
\begin{equation*}
\P \pb{L_{s}^i \geq 1 \condb A_{(B_{i-1})} } \leq N \exp(-cd (s \wedge s^2)) \leq N^{-\nu},
\end{equation*}
where the last inequality follows for $s = \ee^{(k_{\max} + 1)/2}$ and $k_{\max}$ defined above. This establishes~\eqref{eq:proof_concentration_degrees_aux2}.

Next, we estimate $\abs{\mathcal{N}_k^i}$. This will allow us to conclude the statement of the lemma from \eqref{eq:proof_concentration_degrees_aux1} 
and \eqref{eq:proof_concentration_degrees_aux2}. In fact, we have $\P ( \abs{\mathcal{N}_k^i} \geq \ell \mid A_{(B_{i-1})} ) \leq \P ( L_{\ee^{k/2}}^i \geq \ell\mid A_{(B_{i-1})}) $. 
We choose $\ell = \ell_k$, where 
\[ \ell_k \defeq \frac{\Cnu}{d}\big( \abs{S_i} + \log N \big)\begin{cases} \ee^{-k/2} , & \text{ if } k \geq 0, \\ 
\ee^{-k}, & \text{ if }k < 0. \end{cases} \]
Using \eqref{eq:L_est} and estimating $\binom{\abs{S_i}}{\ell} \leq \ee^{\abs{S_i}}$, we deduce that $\abs{\mathcal{N}_k^i} \leq \ell_k$ with very high probability on $\Xi_i$
conditioned on $A_{(x)}$.

With this information, we now estimate the right-hand side of \eqref{eq:proof_concentration_degrees_aux2}. We conclude from \eqref{eq:proof_concentration_degrees_aux1} 
and \eqref{eq:proof_concentration_degrees_aux2} that 
\[\begin{aligned} 
\abs{S_i} Z_i & \leq d\abs{S_i} + d^2 \sum_{k=k_{\min}}^{k_{\max}} \ee^{k+1} \ell_k \\ 
 & \leq d\abs{S_i} + \Cnu d \big( \abs{S_i} + \log N\big) \bigg( \sum_{k=k_{\min}}^0 \ee^{k+1} \ee^{-k} + \ind{k_{\max} \geq 0}\sum_{k=0}^{k_{\max}} \ee^{k+1} \ee^{-k/2} \bigg) \\ 
& \leq d\abs{S_i} + \Cnu d \big(\abs{S_i} + \log N\big)\bigg( \abs{k_{\min}} + \ee^{k_{\max}/2} \bigg) \\ 
 & \leq \Cnu d \big( \abs{S_i} + \log N\big) \bigg( \log (d) + \bigg( \frac{\log N}{d} \bigg)^{1/2} \vee \frac{\log N}{d} \bigg),
\end{aligned} \] 
which concludes the proof of \eqref{eq:Z_claim} and hence also of \eqref{eq:estimate_w_2}.
\end{proof}

\section{Quadratic form estimates on centred adjacency matrix} \label{sec:proof_upper_bound_adjacency} 

The main result of the present section is a bound on $\underline{A} = A - \E A$ in Proposition~\ref{pro:upper_bound_on_adjacency_matrix} below. 
In the following, we write $S \geq T$ for two Hermitian $N\times N$ matrices $S, T$ if $S - T$ is positive semidefinite, i.e.,  
$\scalar{\f w}{S \f w} \geq \scalar{\f w}{T \f w}$ for all $ \f w \in \R^N$. We recall the choice of $\sigma$ from \eqref{eq:def_sigma_permutation}.

\begin{proposition}[Upper bound on $d^{-1/2} \underline{A}$] \label{pro:upper_bound_on_adjacency_matrix} 
If $4 \leq d \leq N^{2/13}$ then, with very high probability, we have 
\[ \id + D + E \geq d^{-1/2} \abs{\underline{A}}, \] 
where $\abs{\underline{A}} = \sqrt{\underline{A}^*\underline{A}}$, $D$ is the diagonal matrix defined through $D \defeq \diag(D_x/d)_{x \in [N]}$ and the error matrix $E$ satisfies 
\[ \norm{E} \leq \frac{\Cnu (d + D_{\sigma(1)})}{d^{3/2}} \leq \Cnu d^{-1/2} \begin{cases} 1 + d^{-1/2} \sqrt{\log N}, & \text{ if } d \geq \frac{1}{2} \log N, \\ \frac{\log N}{d( \log \log N - \log d)}, & \text{ if } d \leq \frac{1}{2} \log N \end{cases} \] 
with very high probability.  
\end{proposition} 

We postpone the proof of Proposition~\ref{pro:upper_bound_on_adjacency_matrix} to Section~\ref{subsec:proof_upper_bound_on_adjacency_matrix} below. First we state and prove the following corollary of Proposition~\ref{pro:upper_bound_on_adjacency_matrix}. 

\begin{corollary}[Norm bound on $\underline{A}$] \label{cor:norm_bound_underline_A} 
If $4 \leq d \leq N^{2/13}$ then we have 
\[ \norm{\underline{A}} \leq \sqrt{d} + \frac{D_{\sigma(1)}}{\sqrt{d}} + \Cnu \bigg(1 + \frac{D_{\sigma(1)}}{d} \bigg) \leq \begin{cases} 2\sqrt{d} + \Cnu \sqrt{\log N}, & \text{ if } d \geq \frac{1}{2} \log N, \\ 
 \frac{\Cnu \log N}{d^{1/2}( \log\log N - \log d)}, & \text{ if } d \leq \frac{1}{2} \log N \end{cases} \] 
with very high probability. 
\end{corollary}

Corollary~1.3 in \cite{BBK2} and Corollary~3.3 in \cite{BBK1} 
provide similar statements to Corollary~\ref{cor:norm_bound_underline_A}.

\begin{proof}[Proof of Corollary~\ref{cor:norm_bound_underline_A}]
Owing to Lemma~\ref{lem:upper_bound_degree}, we have $\norm{D} \leq D_{\sigma(1)}/d \leq \Delta / d$
with very high probability. Therefore, Corollary~\ref{cor:norm_bound_underline_A} follows immediately from Proposition~\ref{pro:upper_bound_on_adjacency_matrix}. 
\end{proof} 

\subsection{Proof of Proposition~\ref{pro:upper_bound_on_adjacency_matrix}} \label{subsec:proof_upper_bound_on_adjacency_matrix} 

Let $B$ be the \emph{nonbacktracking matrix} associated with $d^{-1/2} \underline{A} = d^{-1/2}(A - \E A)$, i.e.\ the $N^2 \times N^2$ matrix with entries  $B_{ef} \defeq d^{-1/2} \underline{A}_{uv} \ind{y=u} \ind{x \neq v} $ for $e = (x,y) \in [N]^2$ and $f = (u,v)\in [N]^2$. 
The next proposition provides a high probability bound on the spectral radius of the nonbacktracking matrix. It is proved in \cite[Theorem 2.5]{BBK1}. 

\begin{proposition}[Bound on the nonbacktracking matrix of $d^{-1/2}\underline{A}$] \label{pro:bound_nonbacktracking_matrix} 
There are universal constants $C>0$ and $c >0$ such that, for all $4 \leq d \leq N^{2/13}$ and $\eps \geq 0$, we have 
\[ \P \big( \rho(B) \geq 1 +\eps \big) \leq C N^{3 - c \sqrt{d} \log ( 1 + \eps)}. \] 
\end{proposition} 

The Ihara-Bass-type formula in the following lemma relates the spectra of $B$ and $\underline{A}$. 
Its formulation is identical to the one of Lemma 4.1 in \cite{BBK1}.  
Therefore, we shall not present its proof in this paper. 

\begin{lemma}[Ihara-Bass-type formula] \label{lem:ihara_bass_formula} 
Let $\ul A$ be an $N \times N$ matrix and let $B$ be the nonbacktracking matrix associated with $d^{-1/2}\underline{A}$. Let $t \in \C$ satisfy $t^2 \neq d^{-1} \underline{A}_{xy} \underline{A}_{yx}$ for all $x,y 
\in [N]$. 
We define the matrices $\underline{A}(t)=(\underline{A}_{xy}(t))_{x,y \in [N]}$ and $M(t)= (m_x(t)\delta_{xy})_{x,y \in [N]}$ through 
\[ \underline{A}_{xy}(t) \defeq \frac{t\sqrt{d}\underline{A}_{xy}}{t^2d -\underline{A}_{xy} \underline{A}_{yx}}, \qquad m_x(t) \defeq 1 + \sum_y \frac{\underline{A}_{xy}\underline{A}_{yx}}{t^2 d - \underline{A}_{xy}\underline{A}_{yx}}. \] 
Then $t \in \spec(B)$ if and only if $\det(M(t) - \underline{A}(t)) = 0$. 
\end{lemma} 

An argument similar to the following proof of Proposition~\ref{pro:upper_bound_on_adjacency_matrix} has been used to show Proposition 4.2 in \cite{BBK1}. 

\begin{proof}[Proof of Proposition~\ref{pro:upper_bound_on_adjacency_matrix}]
We only show that $d^{-1/2} \underline{A} \leq \id + D + E$. The same proof implies that $-d^{-1/2} \underline{A}$ 
satisfies the same bound. 
In this proof, we use the matrices $\underline{A}(t)$ and $M(t)$ defined in Lemma~\ref{lem:ihara_bass_formula} exclusively for $t \in \R$. Note that $\underline{A}(t)$ and $M(t)$ are Hermitian for all $t \in \R$. 
If $t \in \R$ converges to $+ \infty$ then we have that $M(t) - \underline{A}(t) \to \id$. 
Therefore, $M(t) - \underline{A}(t)$ is strictly positive definite for all sufficiently large $t >0$. Let $t_*$ be the infimum of all $t >0$ such that $M(t) - \underline{A}(t)$ is strictly positive definite. 
Hence, by continuity, the smallest eigenvalue of $M(t_*) - \underline{A}(t_*)$ is zero while all eigenvalues of $M(t)- \underline{A}(t)$ are strictly positive for $t >t_*$. Therefore, Lemma~\ref{lem:ihara_bass_formula} implies that 
$t_* \in \spec(B)$ and $M(t) - \underline{A}(t)$ is strictly positive definite for all $t \in (\rho(B), \infty)$. 
Hence, Proposition~\ref{pro:bound_nonbacktracking_matrix} yields that 
\begin{equation} \label{eq:proof_upper_bound_adjacency_aux1} 
\P\big( M(1 + \eps) - \underline{A}( 1 + \eps) \geq 0 \big) \geq 1- C N^{3 - c \sqrt{d} \log ( 1 + \eps)}. 
\end{equation} 

We shall establish below the existence of a constant $C>0$ such that
\begin{equation} \label{eq:proof_upper_bound_adjacency_aux2} 
\norm{\underline{A}(t) - t^{-1} d^{-1/2}  \underline{A}} \leq \frac{C( 1+ D_{\sigma(1)})}{d^{3/2}}, \qquad \norm{M(t) - \id - t^{-2} D} \leq \frac{C(1 + D_{\sigma(1)})}{d^2} 
\end{equation} 
for each $t \in [1,2]$. 
Since $M(t) - \underline{A}(t) \geq 0$ and $D \geq 0$ imply 
\[ d^{-1/2} \underline{A} \leq \id + D + (t-1)\id + t\big(\norm{\underline{A}(t) - t^{-1} d^{-1/2}  \underline{A}} + \norm{M(t) - \id - t^{-2} D}\big)\id, \] 
choosing $t = 1 + \eps$ with $\eps = \Cnu d^{-1/2}$ and using \eqref{eq:proof_upper_bound_adjacency_aux1}, \eqref{eq:proof_upper_bound_adjacency_aux2} as well as 
$\log(1 + \eps) \geq c \eps$ for some $c>0$ establish Proposition~\ref{pro:upper_bound_on_adjacency_matrix} up to showing \eqref{eq:proof_upper_bound_adjacency_aux2}. 

We now prove \eqref{eq:proof_upper_bound_adjacency_aux2}. In order to estimate $\underline{A}(t) - t^{-1} d^{-1/2} \underline{A}$, we use the Schur test to conclude
\[ \norm{\underline{A}(t) - t^{-1} d^{-1/2} \underline{A}} \leq \max_{x} \sum_{y} \abs{\underline{A}_{xy}(t) - t^{-1} d^{-1/2} \underline{A}_{xy}}.  \] 
A short computation shows that 
\[ \max_x\sum_{y} \abs{\underline{A}_{xy}(t) - t^{-1} d^{-1/2} \underline{A}_{xy}}  \leq \max_x \sum_y \frac{\abs{\underline{A}_{xy}}^3}{t\sqrt{d}(t^2d - \underline{A}_{xy}^2)}
\leq \frac{2}{t^3 d^{3/2}} \big( D_{\sigma(1)} + 1\big),   \] 
where we used $\abs{A_{xy}} \leq 1$, $t^2d/2 \geq \underline{A}_{xy}^2$ and $\sum_y \underline{A}_{xy}^2 \leq D_x + 1\leq D_{\sigma(1)} + 1$ in the last step (recalling that $\ul A_{xy} = A_{xy} - d/N$). 
Thus, the first bound in \eqref{eq:proof_upper_bound_adjacency_aux2} follows due to $t \geq 1$. 

As $M(t)$ and $D$ are diagonal matrices by definition, we obtain 
\[ \norm{M(t) - \id - t^{-2} D} = \max_x \absbb{\sum_{y} \bigg(\frac{\underline{A}_{xy}^2}{t^2d - \underline{A}_{xy}^2} - \frac{1}{t^2d} A_{xy}^2 \bigg)} \leq \max_x \sum_y \bigg(\frac{\underline{A}_{xy}^4}
{t^2d(t^2d - \underline{A}_{xy}^2)} + \frac{2 \underline{A}_{xy}^2}{t^2N} + \frac{d}{t^2 N^2} \bigg) 
. \] 
Arguing similarly as in the proof of the first bound in \eqref{eq:proof_upper_bound_adjacency_aux2} implies the second bound in \eqref{eq:proof_upper_bound_adjacency_aux2}. 
This completes the proof of \eqref{eq:proof_upper_bound_adjacency_aux2} and, thus, the one of Proposition~\ref{pro:upper_bound_on_adjacency_matrix}. 
\end{proof}

\section{Lower bounds on large eigenvalues}  \label{sec:lower_bound_eigenvalues} 

The main result of this section is the following proposition. It states that the $l$-th largest eigenvalue of $\underline{A}$, ${\lambda}_l(\underline{A})$, is bounded from below 
by $\sqrt{d}\Lambda(\alpha_{\sigma(l)})$ , up to a small error term, as long as $\alpha_{\sigma(l)}$ is sufficiently large. 
We recall that $\alpha_x \defeq D_x/d$ for any $x \in [N]$ and the permutation $\sigma$ of $[N]$ is chosen such that $(\alpha_{\sigma(l)})_{l=1}^N$ is nonincreasing (cf.~\eqref{eq:def_sigma_permutation}). 
Similarly, up to a small error term, $-\sqrt{d} \Lambda(\alpha_{\sigma(l)})$ bounds the $l$-th smallest eigenvalue, $\lambda_{N+1-l}(\underline{A})$, of $\underline{A}$ from above.

\begin{proposition} \label{pro:lower_bound_number_outliers} 
Let $\cal K \sqrt{\log N} \leq d \leq \exp\pb{\sqrt{\log N} / 4}$.
There is a universal constant $C>0$ such that if 
the random index $L_\geq$ is defined through
\begin{equation} \label{eq:def_L_geq_tau_star} 
 L_\geq \defeq \max\bigg\{  l \in [N] \col \alpha_{\sigma(l)} \geq \tau_* \bigg\}, \qquad
\tau_* \deq 2 + C \frac{(\log d)^2}{d\wedge \log N}\, 
\end{equation} 
then, for any $l \in [L_\geq]$, the bound 
\[ \begin{aligned} 
\min\big\{\lambda_l(\underline{A}),- \lambda_{N+1-l}(\underline{A})\big\} 
& \geq \sqrt{d} \Lambda\big(\alpha_{\sigma(l)} \big) + \ord\bigg( \bigg( \log d + \frac{\log N}{d} \bigg)^{1/2} 
\bigg( 1 + \frac{\log N}{D_{\sigma(l)}} \bigg)^{1/2} \\ 
 & \quad \qquad \qquad \quad \qquad \qquad \qquad \qquad  +
 \bigg(1 +  \frac{\Delta}{d} \bigg)\bigg( \frac{ \Delta}{D_{\sigma(l)}} \frac{ d + \log N}{d} \bigg)^{1/2}\bigg) 
\end{aligned} \] 
holds with very high probability. Here, $\Delta$ is defined as in \eqref{eq:def_Delta}. 
In the definition of $L_\geq$, we use the convention that $L_\geq=0$ if $\alpha_{\sigma(1)} < \tau_*$.
\end{proposition} 

The following lemma will be a key ingredient in the proof of the previous proposition. For its formulation, we introduce the 
set $\cal V_\tau$ of vertices of large degree given by 
 \[ \cal V_\tau \defeq \big\{ x \in [N] \col D_x \geq \tau d\big\},\] 
where $\tau >1$. 
We recall the definition of $h$ from \eqref{eq:def_h}. 
The following lemma provides a subgraph $G_\tau$ of $G$ such that, as $N$ goes to infinity, the length of the shortest path in $G_\tau$ of two vertices in $\cal{V}_\tau$ tends to 
infinity with a lower bound given in terms of $r(\tau)$ defined through  
\begin{equation} \label{eq:def_r_tau} 
r(\tau) \defeq \frac{d}{2 \log d}  h\bigg( \frac{\tau - 1}{2} \bigg)  -2. 
\end{equation} 
The next lemma establishes the existence of the prunder graph $G_\tau$ and lists its properties. 

\begin{lemma}[Existence of pruned graph] 
\label{lem:subgraph_separating_large_degrees} 
Let $\tau > 1$ and $r(\tau)$ be defined as in \eqref{eq:def_r_tau}. For all $x \in \cal V_\tau$, we set $r_{x,\tau} \defeq (\frac{1}{4} r_x)\wedge ( \frac{1}{2} r(\tau)) $ with $r_x$ from \eqref{eq:def_r_star}. 
Then there exists a subgraph $G_\tau$ of $G$ with the following properties. 
\begin{enumerate}[label=(\roman*)]
\item \label{item:subgraph_paths} If a path $p$ in $G_\tau$ connects two vertices $x, y \in \cal V_\tau$, $x \neq y$, then $p$ has length at least $r_{x,\tau} + r_{y,\tau} +1$. In particular, the balls $B_{r_{x,\tau}}^{G_\tau}(x)$
for $x \in \cal V_\tau$ are disjoint.
\item \label{item:subgraph_tree} The induced subgraph $G_\tau|_{B_{r_{x,\tau}}^{G_\tau}(x)}$ is a tree for each $x \in \cal V_\tau$. 
\item \label{item:subgraph_cut_only_in_S_1} For each edge in $G\setminus G_\tau$, there is at least one vertex in $\cal{V}_\tau$ incident to it. 
\item \label{item:subgraph_inclusion_S_i} For each $x \in \cal V_\tau$ and each $i \in \N$ satisfying $1 \leq i \leq r_{x,\tau}$ we have $S_i^{G_\tau}(x) \subset S_i^G(x)$. 
\item \label{item:subgraph_N_s_agree} For each $x \in \cal V_\tau$, we have  
\[S_1^{G_\tau}(y) \cap S_i^{G_\tau}(x) = S_1^{G}(y) \cap S_i^{G}(x)\] 
 for all $y \in B_{r_{x,\tau}}^{G_\tau}(x)\setminus \{ x\}$ and $1 \leq i \leq r_{x,\tau}$. 
\item \label{item:subgraph_degrees} Let $ C \leq d \leq N^{1/4}$ and $\tau \geq 1 + \mathcal{K} \big(\frac{\log N}{d^2} \big)^{1/3}$. 
The degrees induced on $[N]$ by $G\setminus G_\tau$ are bounded according to
\[\max_{x \in [N]} D_x^{G \setminus G_\tau} = \ord\bigg(1 + \frac{\log N}{h( (\tau-1)/2)d}\bigg)  \] 
with very high probability. 
\item \label{item:subgraph_S_i} 
Let $\mathcal{K}\log \log N\leq d \leq \mathcal{K}^{-1} N^{1/4}$. 
For each $x \in \cal V_\tau$ and all $2 \leq i \leq \log N/(4 \log d)$, the bound 
\begin{equation}  
 \abs{S_i^G(x) \setminus S_i^{G_\tau}(x)} \leq  D_x^{G \setminus G_\tau} d^{i-2} \Delta \bigg[ 1+ \Cnu \sqrt{\frac{\log N}{d \Delta}}\bigg], 
\label{eq:bound_S_i_setminus_S_i_tau}
\end{equation}
holds with very high probability. Here, $\Delta$ is defined as in \eqref{eq:def_Delta}.  
\end{enumerate} 
\end{lemma} 

We postpone the proof to the following subsection. First we now conclude Proposition~\ref{pro:lower_bound_number_outliers} from Proposition~\ref{pro:approximate_eigenvector}, Corollary~\ref{cor:norm_bound_underline_A} and Lemma~\ref{lem:subgraph_separating_large_degrees}.

\begin{proof}[Proof of Proposition~\ref{pro:lower_bound_number_outliers}]
We always assume that $L_\geq >0$. Otherwise there is nothing to prove. 
We shall only prove the statement about $\lambda_l(\underline{A})$ and leave the necessary modifications for the analogous statement about ${\lambda}_{N+1-l}(\underline{A})$ 
to the reader (see the proof of Proposition~\ref{pro:approximate_eigenvector}). 
Let $G_2$ be a subgraph of $G$ possessing the properties described in Lemma~\ref{lem:subgraph_separating_large_degrees} for $\tau =2$. 

We fix $l \in [L_\geq]$ and set $x \defeq \sigma(l)$. 
Let $\f v^{(x)}$ be the associated approximate eigenvector of $\underline{A}$ around $x$ constructed in~\eqref{eq:def_v_x} with $r=r_{x,2} - 2$, where $r_{x,2} \defeq  (\frac{1}{4}r_x) \wedge (\frac{1}{2} r(2))$ 
for $r(2)$ defined in \eqref{eq:def_r_tau}. We now apply Proposition~\ref{pro:approximate_eigenvector}. The condition $\log d \leq r$ is satisfied provided that $\log d \leq \frac14 r_x$, which, by Lemma \ref{lem:upper_bound_degree}, holds with very high probability under our assumption $\log d \leq \sqrt{\log N} / 4$. The upper bound on $D_x$ in \eqref{eq:approximate_eigenvector_event} holds with very high probability due to Lemma~\ref{lem:upper_bound_degree}. Finally, the lower bound on $D_x$ in \eqref{eq:approximate_eigenvector_event} follows from $\alpha_x \geq \tau_*$ by definition of $L_\geq$ (see \eqref{eq:def_L_geq_tau_star}). Thus, from Proposition~\ref{pro:approximate_eigenvector}, we conclude for $\lambda_x = \sqrt{d} \Lambda(\alpha_x)$ that 
\begin{equation} \label{eq:bound_A_lambda_x_v_restricted} 
 \norm{(\underline{A} - \lambda_x)\f v^{(x)}} = \ord\bigg( \bigg(\log d + \frac{\log N}{d} \bigg)^{1/2} \bigg( 1 + \frac{\log N}{D_x}\bigg)^{1/2} \bigg) 
\end{equation}
with very high probability.

We define $\tilde {\f v}^{(x)}\defeq (\tilde v^{(x)}(y))_{y \in [N]}$ through
\begin{equation*}
\tilde v^{(x)}(y) \deq v^{(x)}(y) \ind{y \in B^{G_{2}}_{r_{x,2} - 2}(x)}. 
\end{equation*}
We note that the vector $\tilde{\f v}^{(x)}$ is not normalized. 
By the explicit definition of $\f v^{(x)}$ in \eqref{eq:def_v_x}, we therefore conclude that
\begin{equation*}
\norm{\f v^{(x)} - \tilde{\f v}^{(x)}}^2 = \sum_{i = 1}^{r_{x,2}-2} u_i^2 \frac{\abs{S_i^G(x)\setminus S_i^{G_{2}} (x)}}{\abs{S_i^G(x)}} = \ord\bigg(\frac{\Delta}{D_x}\frac{D_x^{G\setminus G_{2}}}{d} \bigg) 
\end{equation*}
with very high probability due to Lemma~\ref{lem:subgraph_separating_large_degrees} \ref{item:subgraph_S_i}, \eqref{eq:concentration_S_i} combined with $\cal K \log N \leq D_x d$ (see the remark below Lemma \ref{lem:concentration_S_i}) by \eqref{eq:def_L_geq_tau_star}  
and $\sum_{i=1}^{r_{x,2}} u_i^2 \leq 1$. Here, when applying Lemma~\ref{lem:subgraph_separating_large_degrees} \ref{item:subgraph_S_i}, we also employed that $r_{x,2} \leq r_x/4 \leq \log N/(4 \log d)$ as $D_x \geq 2d$. 
Hence, we have that 
\begin{equation} \label{eq:bound_v_x_minus_tilde_v_x} 
 \norm{\f v^{(x)} - \tilde{\f v}^{(x)}} = \ord\bigg(\bigg(\frac{\Delta}{D_x}\frac{D_x^{G\setminus G_{2}}}{d} \bigg)^{1/2} \bigg) 
\end{equation}
with very high probability. 
Therefore, 
from \eqref{eq:bound_v_x_minus_tilde_v_x}, Corollary~\ref{cor:norm_bound_underline_A}, and \eqref{eq:bound_A_lambda_x_v_restricted}, 
we deduce for $\lambda_x = \sqrt{d}\Lambda(\alpha_x)$ that
\begin{equation}\label{eq:proof_lower_bound_approximate_eigenvector} 
\begin{aligned} 
(\underline{A} - \lambda_x) \frac{\tilde{\f v}^{(x)}}{\norm{\tilde{\f v}^{(x)}}} & =
(\underline{A} - \lambda_x) \frac{\f v^{(x)}}{\norm{\tilde{\f v}^{(x)}}} + (\underline{A}-\lambda_x) \frac{\tilde{\f v}^{(x)}- \f v^{(x)}}{\norm{\tilde{\f v}^{(x)}}} \\ 
& = \ord\bigg(\bigg(\log d + \frac{\log N}{d} \bigg)^{1/2} \bigg( 1 + \frac{\log N}{D_x}\bigg)^{1/2} + \frac{d + \Delta}{\sqrt{d}} \bigg(\frac{\Delta}{D_x} 
\frac{D_x^{G \setminus G_{2}}}{d} \bigg)^{1/2} \bigg) \\ 
& = \ord\bigg( \bigg( \log d + \frac{\log N}{d} \bigg)^{1/2}\bigg( 1 + \frac{\log N}{D_x} \bigg)^{1/2} + \bigg(1 +  \frac{\Delta}{d} \bigg)\bigg( \frac{\Delta}{D_x} 
 \frac{d + \log N}{d} \bigg)^{1/2} \bigg) 
\end{aligned} 
\end{equation}
with very high probability. Here, we used Lemma~\ref{lem:subgraph_separating_large_degrees} \ref{item:subgraph_degrees} in the last step. 
Hence, $(\tilde{\f v}^{(\sigma(l))})_{l=1}^{L_\geq}$ defines a family of orthogonal approximate eigenvectors of $\underline{A}$ as their supports are disjoint by Lemma~\ref{lem:subgraph_separating_large_degrees} \ref{item:subgraph_paths}. 

We set $\wt{W}_l \defeq \op{span}\{ \wt{\f v}^{(\sigma(k))} \col k \in [l] \}$. 
In the following, we write $\mathbb{S}(W)$ for the unit sphere with respect to the Euclidean norm in any linear subspace $W \subset \R^N$. 
Let $A_2=\op{Adj}(G_2)$ be the adjacency matrix of $G_2$. 
The max-min principle for $\lambda_l(\underline{A})$ yields 
\begin{equation} \label{eq:proof_lower_bound_estimate_lambda_l} 
\begin{aligned} 
 \lambda_l(\underline{A}) &  = \max_{\dim W = l} \min_{\f w \in \mathbb{S}(W)} \scalar{\f w}{\underline{A} \f w } \\ 
& \geq \min_{\f w \in \mathbb{S}(\wt{W}_l)} \scalar{\f w }{\underline{A} \f w } \\ 
& \geq \min_{\f w \in \mathbb{S}(\wt{W}_l)} \scalar{\f w }{A_{2} \f w} -\norm{A-A_{2}} - \norm{(\E A)|_{\wt{W}_l}}\\ 
& \geq \min_{k \in [l]} \lambda_{\sigma(k)} - 2\norm{A-A_{2}} -2   \norm{(\E A)|_{\wt{W}_l}} - \max_{k \in [l]} \norm{(\underline{A} - \lambda_{\sigma(k)})\tilde{\f v}^{(\sigma(k))}}\norm{\tilde{\f v}^{(\sigma(k))}}^{-1}. 
\end{aligned} 
\end{equation} 
Here, we added and subtracted $A_{2}$ in the third step and denote by $(\E A)|_{\wt{W}_l}$ the restriction of the matrix $\E A$ to the subspace $\wt{W}_l$. The last step follows from the definition of $\wt{W}_l$, the orthogonality of $\tilde{\f v}^{(\sigma(k))}$ and $A_{2}\tilde{\f v}^{(\sigma(k'))}$ for $k \neq k'$ and 
\[ \frac{\scalarb{\tilde{\f v}^{(\sigma(k))}}{A_{2} \tilde{\f v}^{(\sigma(k))}}}{\norm{\tilde{\f v}^{(\sigma(k))}}^2} \geq \lambda_{\sigma(k)}  - \norm{\tilde{\f v}^{(\sigma(k))}}^{-1} \norm{(\underline{A}- \lambda_{\sigma(k))}) \tilde{\f v}^{(\sigma(k))}} - \norm{A - A_{2}} - \norm{(\E A)|_{\wt{W}_l}}. \]

Now, we estimate the terms on the right-hand side of \eqref{eq:proof_lower_bound_estimate_lambda_l} to obtain the lower bound on $\lambda_l(\underline{A})$ in the proposition. 
Since $t \mapsto \Lambda(t)$ is monotonically increasing for $t \geq 2$, the first term is bounded from below by $\sqrt{d} \Lambda(\alpha_{\sigma(l)})$. 
For the second term, we use $\norm{A-A_{2}} \leq \max_{x \in [N]}D_x^{G\setminus G_{2}} \leq \Cnu ( 1 +  \log N/d)$ with very high probability by Lemma~\ref{lem:subgraph_separating_large_degrees} 
\ref{item:subgraph_degrees}. If $\f w \in \mathbb{S}(\wt{W}_l)$ then 
$\supp \f w \subset Z_l$, where $Z_l \deq \bigcup_{k \in [l]} B^{G_2}_{r_{\sigma(k),2}-2}(\sigma(k))$. 
Therefore, by the definition of $(\E A)_{Z_l}$, we obtain 
\[ \norm{(\E A)|_{\wt{W}_l}}^2 \leq \norm{(\E A)_{Z_l}}^2 \leq \frac{d^2}{N^2} \abs{Z_l}^2 \leq  \frac{4}{d^2} . \] 
Here, we estimated the operator norm of $(\E A)_{Z_l}$ by its Hilbert-Schmidt norm 
and used that $\E A_{ij} = \frac{d}{N}$ for $i \neq j$ as well as 
$\abs{Z_l} \leq \frac{2}{d^2} \sum_{k \in [l]} \abs{B_{r_{\sigma(k),2}}^{G_2}(\sigma(k))} \leq \frac{2N}{d^2}$ by Lemma~\ref{lem:concentration_S_i} and Lemma~\ref{lem:subgraph_separating_large_degrees} ~\ref{item:subgraph_paths}. 
For the fourth term in \eqref{eq:proof_lower_bound_estimate_lambda_l}, we use \eqref{eq:proof_lower_bound_approximate_eigenvector}.  
This completes the proof of Proposition~\ref{pro:lower_bound_number_outliers}.  
\end{proof}

\subsection{Proof of Lemma~\ref{lem:subgraph_separating_large_degrees}} 

For the proof of Lemma~\ref{lem:subgraph_separating_large_degrees} we need the next lemma. 
For any $x \in \cal V_\tau$, it provides a bound on the number of other vertices in $\cal V_\tau$ whose distance from $x$ is sufficiently small. 

\begin{lemma} \label{lem:number_large_degree_distance_r} 
There is a universal constant $C>0$ and a ($\nu$-dependent) $\cal K >0$ such that if 
$ C \leq d \leq N^{1/4}$ and $\tau \geq 1 + \mathcal{K} \big(\frac{\log N}{d^2} \big)^{1/3}$ 
then the following holds. 

For any $r\in \N$ satisfying $r\leq r(\tau)$ with $r(\tau)$ from \eqref{eq:def_r_tau} 
and for any $x \in\cal V_\tau$, we have
\begin{equation} \label{eq:V_alpha_B_r_bound}
 \abs{\cal V_\tau \cap B_r(x)} = \ord\bigg( \frac{\log N}{h((\tau-1)/2) d} \bigg) 
\end{equation}
with very high probability.  
\end{lemma}

The following lemma controls the growth of $\abs{S_i(z)}$ in terms of $d$ and $\Delta$. In contrast to \eqref{eq:concentration_S_i} in Lemma~\ref{lem:concentration_S_i}, 
no lower bound on $D_z$ is required and no lower bound on $\abs{S_i(z)}$ is provided. 

\begin{lemma} \label{lem:upper_bound_S_i_without_lower_bound_D_z} 
Let $\mathcal{K} \log \log N \leq d \leq \mathcal{K}^{-1} N^{1/4}$ and let $z \in [N]$. 
For any $i \leq \log N/(4 \log d)$, the bound 
\[ \abs{S_i(z)} \leq \Delta d^{i-1} \bigg[ 1 + \Cnu \bigg( \frac{\log N}{d \Delta} \bigg)^{1/2} \bigg] \] 
holds with very high probability. Here, $\Delta$ is defined as in \eqref{eq:def_Delta}. 
\end{lemma} 
The proofs of the previous two lemmas are postponed until the end of this section.  

\begin{proof}[Proof of Lemma~\ref{lem:subgraph_separating_large_degrees}] 
In the entire proof, we write $\cal V$ instead of $\cal V_\tau$. 
We shall construct a subgraph $H_\tau$ of $G$ in two steps such that $G_\tau = G\setminus H_\tau$ satisfies the properties stated in the lemma. 
For a graphical depiction of the following argument, we refer to Figure~\ref{fig:pruning}. 

\begin{figure}[!ht]
\begin{center}
\includegraphics{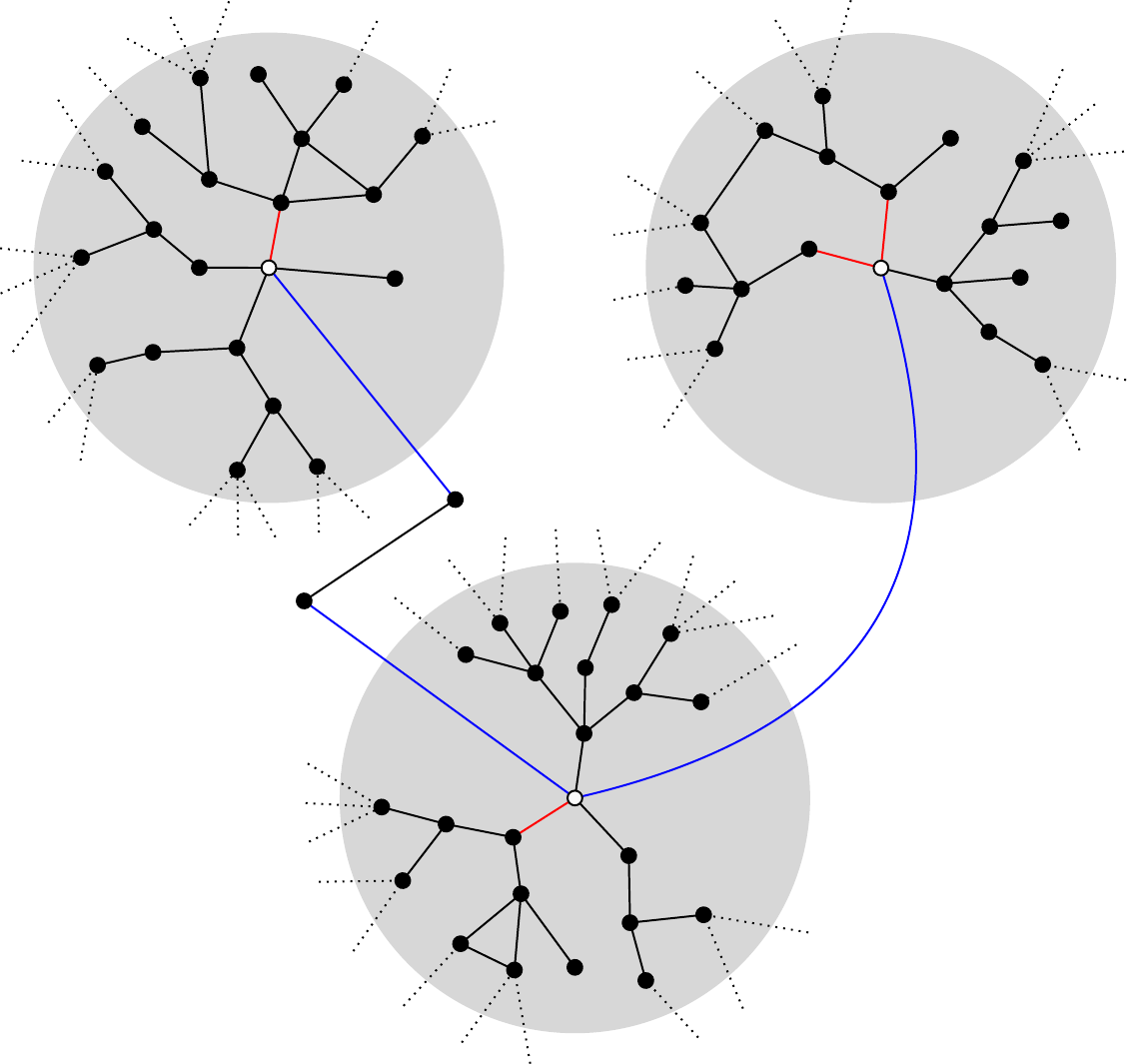}
\end{center}
\caption{A schematic illustration of the algorithm in the proof of Lemma \ref{lem:subgraph_separating_large_degrees}. The vertices of $\cal V$ are white and the other vertices black. The balls $B_{r_{x,\tau}}^{G_\tau}(x)$, for each $x \in \cal V$, are indicated using grey balls, and they are disjoint by construction. The edges of the subgraph $H^{(1)}$ are drawn in red. The edges of the subgraph $H^{(2)}$ are drawn in blue.
\label{fig:pruning}}
\end{figure}

In the following construction of $H_\tau$, we shall identify those edges indicent to a 
vertex $x \in \cal V_\tau$ that lead to a loop (i.e.\ prevent the graph from being a tree in the 
vicinity of $\cal V_\tau$) or a connection to another vertex in $\cal V_\tau$. 
We shall exclusively cut these edges, thus removing the whole corresponding branches in $B_{r_x}(x)$, 
while leaving the other branches of $B_{r_x}(x)$ unchanged.

First, we construct a subgraph $H^{(1)}\subset G$ such that $B_{r_x}^{G \setminus H^{(1)}}(x)$ is a tree for each $x \in \cal V$. 
Indeed, for any $x \in \cal V$ we apply the following algorithm. 
For each $y \in S^G_1(x)$, let $T_y$ be the set of those vertices that are connected to $y$ by a path of length at most $r_x$ not traversing the edge connecting $x$ and $y$. 
If $G|_{T_y}$ is not a tree, i.e., $\abs{T_y} < \abs{E(G|_{T_y})}+1$, or $x \in T_y$, then we include the edge between $x$ and $y$ into $H^{(1)}$. 
We now show that 
\begin{equation} \label{eq:max_degree_H_1} 
 \max_{x \in \cal V} D_x^{H^{(1)}} \leq \Cnu + q_1 
\end{equation}
with very high probability, 
where $q_i$ denotes the maximal number of vertices in $\cal V$ that is in the ball of radius $i$ around a vertex in $\cal V$, i.e., 
\begin{equation} \label{eq:def_q_i} 
 q_i \defeq \max_{x \in \cal V} \abs{\cal V \cap B_i^G(x) \setminus \{x \}}. 
\end{equation}
Let $x \in \cal V$. 
Indeed, owing to Corollary~\ref{cor:tree_approximation}, with very high probability, there are at most $\ord(1)$ edges in $B_{r_x}^G(x)$ that prevent it from being a tree. 
Moreover, $S_1^G(x)$ contains at most $q_1$ vertices in $\cal V$. This proves \eqref{eq:max_degree_H_1} 
and by construction $(G\setminus H^{(1)})|_{B_{{r_x/2}}^{G\setminus H^{(1)}}(x)}$ is a tree for any $x \in \cal V$. 

Second, the subgraph $H^{(2)} \subset G$ consists of edges incident to vertices $x \in \cal V$ that are traversed by paths in $G \setminus H^{(1)}$ of length at most $2 r_{x,\tau}$ connecting $x$ to another vertex in $\cal V$. 
More precisely, for $x \in \cal V$ we add the following edges to $H^{(2)}$. Since $(G\setminus H^{(1)})|_{B_{{r_x/2}}^{G\setminus H^{(1)}}(x)}$ is a tree,  for each $y \in (\cal V \cap B_{2 r_{x,\tau}}^{G\setminus H^{(1)}} (x)) \setminus \{ x\}$, there is a unique vertex $z \in S_1^{G \setminus H^{(1)}}(x)$ such that each path in $G\setminus H^{(1)}$ of length at most $2r_{x,\tau}$ connecting $x$ and $y$
traverses the edge between $x$ and $z$. All such edges between $x$ and $z$ are added to $H^{(2)}$. 
This algorithm yields that 
\begin{equation} \label{eq:max_degree_H_2} 
\max_{x \in \cal V} D_x^{H^{(2)}} \leq q_{r(\tau)}, 
\end{equation}
where $q_{r(\tau)}$ is defined in \eqref{eq:def_q_i} with $i ={r(\tau)}$. 

We set $H_\tau \defeq H^{(1)} \cup H^{(2)}$ and $G_\tau \defeq G \setminus H_\tau$. 
By construction, each path in $G_\tau$ between $x, y\in \cal V$ with $x \neq y$ has length at least $2(r_{x,\tau} \vee r_{y,\tau}) +1$. 
This establishes property \ref{item:subgraph_paths} of Lemma~\ref{lem:subgraph_separating_large_degrees}. 
Moreover, since $G_\tau \subset G\setminus H^{(1)}$ is a subgraph and the latter is a tree when restricted to $B_{r_x}^{G \setminus H^{(1)}}(x)$ 
we obtain \ref{item:subgraph_tree}. 

We also note that the construction of $H_\tau$ explained above yields 
\begin{equation}\label{eq:edges_H_tau} 
 E(H_\tau) \subset \bigcup_{x \in \cal V} \bigcup_{y \in S_1^G(x)} \{ x, y \}, 
\end{equation}
i.e., for each edge in $H_\tau$, there is at least one vertex $x \in \cal V$ incident to it. This shows \ref{item:subgraph_cut_only_in_S_1}. 

For the proof of property \ref{item:subgraph_inclusion_S_i} let $x \in \cal V$ be fixed. The construction of $H^{(1)}$ implies that $S_i^{G\setminus H^{(1)}}(x) \subset S_i^{G}(x)$ for all $1 \leq i \leq r_{x}/2$. 
As $(G\setminus H^{(1)})|_{B_{r_x/2}^{G\setminus H^{(1)}}(x)}$ is a tree, a vertex lies in $S_i^{G_\tau}(x)$ only if it was in $S_i^{G \setminus H^{(1)}}(x)$ 
due to the construction of $H^{(2)}$. Hence, $S_i^{G_\tau}(x) \subset S_i^{G \setminus H^{(1)}}(x) \subset S_i^{G}(x)$ and we deduce \ref{item:subgraph_inclusion_S_i}.  

Property~\ref{item:subgraph_N_s_agree} follows directly from the construction of $G_\tau$ as it 
left all branches in $B_{r_x}^{G_\tau}(x)$ for $x \in \cal V_\tau$ unchanged. 

For each $x \in [N]$, we now verify the bound on $D_x^{G\setminus G_\tau}=D_x^{H_\tau}$ in \ref{item:subgraph_degrees}. 
For any $x \in \cal V$, we have 
\[ D_x^{H_\tau} = D_x^{H^{(1)}} + D_x^{H^{(2)}} \leq \Cnu + q_1 + q_{r(\tau)} \leq \Cnu + 2 \max_{y \in \cal V} \abs{\cal V \cap B_{r(\tau)}(y)} \] 
due to \eqref{eq:max_degree_H_1}, \eqref{eq:max_degree_H_2} and  $q_1 \leq q_{r(\tau)} \leq \max_{y \in \cal V} \abs{\cal V \cap B_{r(\tau)}(y)}$. 
Thus, Lemma~\ref{lem:number_large_degree_distance_r} implies \ref{item:subgraph_degrees} for all $x \in \cal V$.  
Let $x \in [N] \setminus \cal V$. If $S_1^G(x)\cap \cal V = \emptyset$ 
then $D_x^{H_\tau} = 0$ due to \eqref{eq:edges_H_tau}. If $x \in S_1^G(y)$ for some $y \in \cal V$ then 
\[ D_x^{H_\tau} \leq\Cnu \abs{S_1^G(x) \cap \cal V} \leq \Cnu \abs{S_2^G(y) \cap \cal V} \leq \Cnu \frac{\log N}{h\big(\frac{\tau-1}{2}\big)d} \] 
by Lemma~\ref{lem:number_large_degree_distance_r}. This completes the proof of \ref{item:subgraph_degrees}.

What remains is the proof of \eqref{eq:bound_S_i_setminus_S_i_tau}. We fix $x \in \cal V$ and conclude from 
\eqref{eq:edges_H_tau} that 
\begin{equation}
\abs{S_i^G(x) \setminus S_i^{G_\tau}(x)} \leq \sum_{z \in S_1(x) \col \{x,z\} \in E(H_\tau)} \abs{S_{i-1}^G (z)}.
\end{equation}
Lemma~\ref{lem:upper_bound_S_i_without_lower_bound_D_z} provides a uniform bound on the summands in the previous sum. 
The number of elements in this sum is bounded by $D_x^{G \setminus G_\tau}$. 
Hence, we conclude \ref{item:subgraph_S_i}.   
This completes the proof of Lemma~\ref{lem:subgraph_separating_large_degrees}. 
\end{proof} 

\begin{proof}[Proof of Lemma~\ref{lem:number_large_degree_distance_r}] 
The lemma will follow from an estimate on the probability of the event $\Xi^{(k)}$ defined through 
\[ \Xi^{(k)} \defeq \{ \exists \, x \in [N] \col x \in \cal V_\tau,\, \abs{\cal V_\tau \cap B_r(x)} \geq k\} \] 
 for some $k \in \N$ to be chosen later. 
We decompose $\Xi^{(k)}$ according to 
\begin{equation} \label{eq:decomposition_Xi_k} 
\hspace*{-0.15cm} 
\begin{aligned} 
\Xi^{(k)} & = \bigcup_{x, \f y, \f z} \Xi^{(k)}_{x,\f y, \f z}, \\ 
\Xi_{x, \f y, \f z}^{(k)} & = \Big\{ x, y_j \in \cal V_\tau, \, \{x,z^{(j)}_1\}, \{ z^{(j)}_i, z^{(j)}_{i+1}\}, 
\{ z^{(j)}_{r_j}, y_j \} \in E(G) \text{ for all } i \in [r_j-1] \text{ and } j \in [k] \Big\}, 
\end{aligned} 
\end{equation}
where the union is taken over all $x \in [N]$, $k$-tuples $\f y =(y_1, \ldots,  y_k)$ of distinct elements of $[N] \setminus \{ x\}$ 
and $k$-tuples $\f z = (z^{(1)}, \ldots, z^{(k)})$ of paths $z^{(j)} = (x, z_1^{(j)}, \ldots, z_{r_j}^{(j)}, y_j)$  of length $r_j\in \{ 0, \ldots, r\}$ for $j=1, \ldots, k$. 

We fix such $x$, $\f y = (y_1, \ldots, y_k)$ and $\f z = (z^{(1)}, \ldots, z^{(k)})$. 
Then, recalling $p = d/N$, it is easy to see that 
\begin{equation} \label{eq:prob_Xi_x_y_z} 
 \P( \Xi_{x,\f y, \f z}^{{(k)}}) \leq P_{x, \f y} \prod_{j=1}^k p^{r_j+ 1}, 
\end{equation}
where $P_{x, \f y} \defeq \P ( D_x \geq \tau d - k, D_{y_1} \geq \tau d - 1, \ldots, D_{y_k} \geq \tau d-1)$.  
We now show that 
\begin{equation} \label{eq:prob_D_x_D_y_i} 
 P_{x, \f y} \leq C^{k+1} \exp\bigg( -d (k + 1) h\bigg(\frac{\tau-1}{2}\bigg)\bigg) + (k + 1) \bigg(\frac{ (k+1)d}{N} \bigg)^{d(\tau - 1)/2 - k }. 
\end{equation}
We start the proof of \eqref{eq:prob_D_x_D_y_i} by exploiting the fact that $D_x, D_{y_1}, \ldots, D_{y_k}$ are almost independent. Indeed, setting $y_0 \defeq x$, 
$a_0 \defeq (\tau- 1) d -k$, $a_1 \defeq \ldots \defeq a_k \defeq (\tau -1)d  - 1$, we obtain 
\begin{equation} \label{eq:prob_D_x_D_y_i_decoupling} 
 P_{x, \f y} = \E \big[ \P ( D_{y_0} -d \geq a_0, \ldots, D_{y_k} - d \geq a_k\mid A_{Y} ) \big] 
  = \E \bigg[ \prod_{i=0}^k\P (D_{y_i} - d \geq a_i \mid A_{Y} ) \bigg], 
\end{equation} 
where $Y \defeq \{ y_0, \ldots, y_k \}$ and in the last step we used that $D_{y_0}, \ldots D_{y_k}$ 
are independent conditionally on $A_Y$ as
\begin{equation} \label{eq:decomposition_D_y_i} 
 D_{y_i} -d = X_i - \E[X_i\mid A_Y] + \delta_i, \quad \qquad X_i \defeq  \sum_{z \in [N] \setminus Y} A_{z{y_i}}, \qquad \delta_i \defeq \sum_{z \in Y} \bigg( A_{z{y_i}} - \frac{d}{N} \bigg)
\end{equation}
and $X_0, \ldots, X_k$ are independent while the remainder is measurable with respect to $A_Y$.  
Hence, \eqref{eq:decomposition_D_y_i} and Bennett's inequality imply 
\[ \P( D_{y_i} - d \geq a_i \mid A_{Y} ) \leq C \exp \Big( -d h((a_i -\delta_i)/d)\Big)\leq C \exp\Big( -d h \Big(\min_{i=0,\ldots, k} a_i/d - \max_{i=0,\ldots, k} \delta_i/d\Big)\Big). \] 
Therefore, since $\min_i a_i = (\tau - 1) d - k$, \eqref{eq:prob_D_x_D_y_i_decoupling} yields 
\[ 
P_{x, \f y} \leq C^{k+1} \exp\Big( -d (k + 1) h(( \tau-1)/2)\Big) + \P\Big( \max_{i=0,\ldots, k} \delta_i > (\tau - 1)d /2 - k \Big). \] 
We choose $n = (\tau - 1)d /2 - k$ and use the bound 
\[ \P ( \delta_i > n ) \leq \P \bigg( \sum_{z \in Y} A_{z{y_i}} \geq n \bigg) \leq \begin{pmatrix} k+1 \\ n
\end{pmatrix} \bigg( \frac{d}{N}\bigg)^n\leq \bigg(\frac{ (k+1)d}{N} \bigg)^n  \] 
to conclude \eqref{eq:prob_D_x_D_y_i}. 

We now finish the proof by combining the previous estimates. In fact, from \eqref{eq:decomposition_Xi_k}, \eqref{eq:prob_Xi_x_y_z}, $p=d/N$ and \eqref{eq:prob_D_x_D_y_i},  we conclude 
\[ \begin{aligned} 
\P (\Xi^{(k)}) & \leq \sum_{x, \f y, \f z} \P (\Xi_{x,\f y, \f z}^{ {(k)}} ) \\ 
 & \leq 
\begin{pmatrix} N \\ k + 1\end{pmatrix} \sum_{r_1 = 0}^r \ldots \sum_{r_k=0}^r \begin{pmatrix} N - k - 1 \\ r_1 \end{pmatrix} 
\ldots \begin{pmatrix} N -k - 1 -\sum_{l=1}^{k-1} r_l \\ r_k \end{pmatrix} p^{k + \sum_{l=1}^k r_l} \max_{x,\f y} P_{x, \f y} \\ 
 & \leq \sum_{r_1,\ldots,r_k=0}^r N^{k + 1 + \sum_{l=1}^k r_l} p^{k +  \sum_{l=1}^k r_l} \max_{x,\f y} P_{x, \f y} \\ 
 & \leq N \bigg(\sum_{l=0}^r d^{l+1}\bigg)^k \max_{x,\f y} P_{x, \f y} \\ 
 & = N \bigg(\frac{d^{r+2} - 1}{d -1}\bigg)^k \bigg(  C^{k + 1}  \exp\bigg( -d (k + 1) h\bigg(\frac{\tau-1}{2}\bigg)\bigg) + (k + 1) \bigg(\frac{ (k+1)d}{N} \bigg)^{d(\tau - 1)/2 - k } \bigg) . 
\end{aligned} \] 
Therefore, in order to obtain \eqref{eq:V_alpha_B_r_bound}, 
we now show separately that each of the terms in this upper bound is dominated by $N^{-\nu}$. 
For the first term, 
the condition 
\[ (k + 1) \Big(\log C - d h((\tau - 1)/2) + (r + 2) \log d \Big) + \log N < - \nu \log N  \] 
has to be satisfied. This condition is met if $k = \Cnu \log N/( d h( (\tau - 1)/2))$, $r\leq r(\tau)$ and $d \geq C$. Here we used that $d h( (\tau-1)/2) \geq 4 \log d$ as we can assume $r(\tau) \geq 0$ without loss of generality (otherwise there is nothing to be proved). 

The upper bound on the second term follows from $k d \leq \cal C N^{1/3} (\log N)^{2/3}$, 
$d (\tau -1) \geq \cal C k$ and $d(\tau- 1) \geq \cal C$. 
These latter estimates are consequences of $d \leq N^{1/4}$, the lower bound on $\tau - 1$ 
and $h(\tau-1) \asymp (\tau - 1)^2 \wedge (\tau- 1)$. 
This completes the proof of Lemma~\ref{lem:number_large_degree_distance_r}.  
\end{proof}  

\begin{proof}[Proof of Lemma~\ref{lem:upper_bound_S_i_without_lower_bound_D_z}]
By Lemma~\ref{lem:upper_bound_degree}, $D_z \leq \Delta$ with very high probability. Defining $r \deq \log N/(4\log d)$ we thus obtain $D_z \leq \sqrt{N} (2d)^{-r}$ with very high probability
as $d \leq \mathcal{K}^{-1} N^{1/4}$.  
Hence, a simple induction argument 
starting from $\abs{S_{1}(z)} = D_z \leq \Delta$ and  
using \eqref{eq:upper_bound_S_i_without_lower_bound_D_x} as well as $i \leq r$ for the induction step yields 
\begin{equation} \label{eq:proof_subgraph_aux2} 
\abs{S_i(z)} \leq d^{i-1} \Delta \bigg[ 1 + 2 \Cnu \sqrt{\frac{\log N}{d \Delta}} \sum_{k=1}^{i-1}  d^{-(k-1)/2} \bigg] 
\end{equation}
with very high probability for all $1 \leq i \leq \log N/(4 \log d)$.  
Here, we used that $\sqrt{\log N/(d \Delta)}$ is small if $\mathcal{K}$ is large due to $d \geq \mathcal{K}\log \log N$, and for $i \geq 1$, 
\begin{equation} \label{eq:proof_subgraph_aux3} 
 \sum_{j=1}^{i-1} d^{-(j-1)/2} \leq (1-d^{-1/2})^{-1} \leq 2. 
\end{equation}
Combining \eqref{eq:proof_subgraph_aux2} and \eqref{eq:proof_subgraph_aux3} completes the proof of Lemma~\ref{lem:upper_bound_S_i_without_lower_bound_D_z}.  
\end{proof}

\section{Upper bounds on large eigenvalues} \label{sec:upper_bound_eigenvalues}

The following proposition provides the upper bound on the $l$-th largest eigenvalue matching the lower bound of Proposition~\ref{pro:lower_bound_number_outliers}.
We recall that the permutation $\sigma$ was chosen in \eqref{eq:def_sigma_permutation}. 

\begin{proposition} \label{pro:upper_bound_large_eigenvalues} 
Let $\kappa \in (0,1/2)$ be fixed
and suppose that $\theta \in (0,1/2]$.  
 Suppose that $(\log N)^{4/(5-2\theta)} \leq d \leq \mathcal{K}^{-1} N^{1/4}$. 
Define the random index $L_\leq$ through 
\begin{equation} \label{eq:def_L_leq_tau_star} 
L_\leq \defeq \max\{ l \geq 1 \col \alpha_{\sigma(l)} \geq 2 + (\log d)^{-\kappa} \} 
\end{equation}
with the convention that $L_\leq = 0$ if $\alpha_{\sigma(1)} < 2 + (\log d)^{-\kappa}$.  
There is a universal constant $c>0$ such that the following holds with very high probability. 
\begin{enumerate}[label=(\roman*)]
\item If $L_\leq > 0 $ then, for all $l \in [L_\leq]$, \label{item:lemma_upper_bound_l}
\[ \max\{ \lambda_l(\underline{A}), -\lambda_{N+1-l}(\underline{A}) \} 
\leq \sqrt{d} \Lambda(\alpha_{\sigma(l)} ) + \Cnu \sqrt{d} \, \pB{d^{-c(\Lambda(\alpha_{\sigma(l)}) - 2)}+ d^{-\theta/3}}.  \] 
\item \label{item:lemma_upper_bound_0}If $L_\leq = 0$ then 
\[ \max\{ \lambda_1(\underline{A}), - \lambda_N(\underline{A}) \} \leq 2 \sqrt{d} + \Cnu \sqrt{d} (\log d)^{-2\kappa }.\] 
\end{enumerate}
 (Here the constant $\Cnu$ depends on $\kappa$.) 
\end{proposition}

Let $L_\leq$ be defined as in \eqref{eq:def_L_leq_tau_star}. If $L_\leq >0$ then, for $l \in [L_\leq]$, we set 
\begin{equation} \label{eq:def_V_l} 
V_l \defeq \begin{cases} [N], & \text{ if } l = 1, \\ [N]\setminus \{ \sigma(1), \ldots, \sigma(l-1)\}, & \text{ if } l \geq 2.  \end{cases} 
\end{equation}

Let $G_\tau$ be the subgraph of $G$ introduced in Lemma~\ref{lem:subgraph_separating_large_degrees}. We denote by  $A_\tau=\op{Adj} (G_\tau)$ the 
adjacency matrix of $G_\tau$ and also define 
\begin{equation} \label{eq:def_underline_A_tau} 
 \underline{A}_\tau \defeq A_\tau - \Pi_\tau(\mathbb{E}A)\Pi_\tau, \qquad \cal{Z}_\tau \defeq  \bigcup_{x\in \cal{V}_\tau} B^{G_\tau}_{r_{x,\tau}-2}(x),  
\end{equation}
where $\Pi_\tau$ is the orthogonal projection onto $\op{Span}\{ \f 1_y \col y \in [N]\setminus \cal{Z}_\tau\}$. 
Moreover, we introduce the $N\times N$-matrix $\underline{A}_{\tau,l}$ with entries 
\[ (\underline{A}_{\tau,l})_{xy} \defeq (\underline{A}_{\tau})_{xy} \ind{x \in V_l} \ind{y \in V_l} \] 
for $x,y \in [N]$. 

For all $\tau, \zeta, \mu >0$, we define a subset $\cal W_{\tau, \mu, \zeta}$ of $\cal V_\tau$ through 
\begin{equation} \label{def_calW}
\cal W_{\tau, \mu, \zeta} \defeq \{ x \in [N] \col \alpha_x \geq \tau, \, \mu \geq \sqrt{d}(\Lambda(\alpha_x \vee 2) + \zeta) \}.
\end{equation} 

The formulation of the following proposition uses the function $\alpha \col [2,\infty) \to [2, \infty)$ defined through 
\begin{equation} \label{eq:def_alpha}  
 \alpha(\eta) \defeq \frac{\eta}{2} \big( \eta + \sqrt{\eta^2 - 4} \big). 
\end{equation}  
Note that $\alpha$ is monotonically increasing and $\Lambda(\alpha(\eta) ) = \eta$ for all $\eta \geq 2$.

\begin{proposition}[Delocalization estimate] \label{pro:delocalization_estimate} 
Let $(\log N)^{4/(5-2\theta)} \leq d \leq \mathcal{K}^{-1} N^{1/4}$ for some $\theta \in (0,1/2]$ and $1 + d^{-\theta/4} \leq \tau \leq 2$.  
Let $L_\leq$ be defined as in \eqref{eq:def_L_leq_tau_star}, $l \in [L_\leq]$ and $V_l$ be defined as in \eqref{eq:def_V_l}. 
Then there are a universal constant $c>0$ and a constant $\Cnu >0$ such that the following holds. 
If an eigenvalue $\mu> 2\sqrt{d}$ of $\underline{A}_{\tau,l}$ and some $\zeta>0$ satisfy 
\begin{subequations} \label{eq:condition_mu_zeta} 
\begin{align} 
\zeta & \geq \bigg(\Cnu \bigg(\frac{\mu}{\mu - 2\sqrt{d}}\bigg)^{1/2} d^{-\theta/2} \bigg)\vee d^{-c(\mu/\sqrt{d} - 2)}, 
\label{eq:lemma_bound_eigenvector_condition_zeta} \\ 
\mu & \geq \sqrt{d}(2 + \zeta)
\end{align}
\end{subequations}
then, for any normalized eigenvector $\f w$ of $\underline{A}_{\tau,l}$ associated with $\mu$, we have 
\[ \alpha\bigg(\frac{\mu}{\sqrt{d}} \bigg) \sum_{x \in \cal W_{\tau, \mu, \zeta}} \scalar{\f w}{\f 1_x}^2 
\leq \Cnu d^{-c((\mu/\sqrt{d} - 2)\wedge \log (\mu/\sqrt{d}))} \] 
with very high probability.  
\end{proposition} 

\begin{proof}[Proof of Proposition~\ref{pro:upper_bound_large_eigenvalues}]
We shall only prove the upper bound on the large eigenvalues. 
The corresponding lower bound on the small eigenvalues is shown similarly (cf.\ the proof of Proposition~\ref{pro:approximate_eigenvector}). 
We first prove \ref{item:lemma_upper_bound_l} assuming $L_\leq >0$.
We fix $l \in [L_\leq]$ and define the sphere $\bb S_l \deq \h{\f w \in \R^N \col \norm{\f w} = 1, \f w \vert_{V_l} = \f w}$.
The min-max principle implies that 
\begin{equation} \label{eq:proof_upper_bound_aux1} 
 \max_{\tilde{\f w} \in \mathbb{S}_l} \scalar{\tilde{\f w}}{\underline{A}\tilde{\f w}} \geq \lambda_l(\underline{A}). 
\end{equation}

Let $\mu$ be the largest eigenvalue of $\underline{A}_{\tau,l}$ and $\f w\in \bb S_l$ be a corresponding eigenvector.
Since $\scalar{\tilde{\f w}}{\underline{A}_{\tau, l} \tilde{\f w}} = \scalar{\tilde{\f w}}{\underline{A}_\tau \tilde{\f w}}$ for all $\tilde{\f w} \in \mathbb{S}_l$, we get  
\begin{equation} \label{eq:proof_upper_bound_aux7} 
 \mu =\max_{\tilde{\f w}\in \mathbb{S}_l} \scalar{\tilde{\f w}}{\underline{A}_{\tau, l} \tilde{\f w}} = \max_{\tilde{\f w} \in \mathbb{S}_l} \scalar{\tilde{\f w}}{\underline{A}_\tau \tilde{\f w}}. 
\end{equation}
Thus, we obtain the lower bound 
\begin{equation} \label{eq:proof_upper_bound_aux2}
 \mu \geq \max_{\tilde{\f w} \in \mathbb{S}_l} \scalar{\tilde{\f w}}{\underline{A}\tilde{\f w}} - \norm{\underline{A}_\tau - \underline{A}}. 
\end{equation}
On the other hand, \eqref{eq:proof_upper_bound_aux7} and Proposition~\ref{pro:upper_bound_on_adjacency_matrix} imply the upper bound  
\begin{equation} \label{eq:proof_upper_bound_aux3} 
\sqrt{d}\, \scalar{\f w}{(\id + D + E) \f w} \geq \scalar{\f w}{\ul A \f w} \geq \mu - \norm{\underline{A}_\tau - \underline{A}} \geq \lambda_l(\underline{A}) - 2 \norm{\underline{A}_\tau - \underline{A}}. 
\end{equation}
Here, we used \eqref{eq:proof_upper_bound_aux2} and \eqref{eq:proof_upper_bound_aux1} in the last step.

We choose $\tau = 1 + d^{-\theta/4}$ and apply Proposition~\ref{pro:delocalization_estimate}. 
To that end, let $\zeta >0$ be defined by the right-hand side of \eqref{eq:lemma_bound_eigenvector_condition_zeta}. 
For a proof by contradiction, we now assume that 
\begin{equation} \label{eq:proof_upper_bound_aux5} 
 \lambda_l(\underline{A}) > \sqrt{d}\Big( \Lambda(\alpha_{\sigma(l)}) + \zeta \Big) + \norm{\underline{A} - \underline{A}_\tau}. 
\end{equation}
From \eqref{eq:proof_upper_bound_aux5}, \eqref{eq:proof_upper_bound_aux1} and \eqref{eq:proof_upper_bound_aux2},
we deduce  
 $\mu \geq \sqrt{d} \big(\Lambda(\alpha_{\sigma(l)}) + \zeta \big)$. 
This implies that $V_l \cap \cal V_\tau \subset \cal{W}_{\tau, \mu, \zeta}$.  

Moreover, as $\mu \geq \sqrt{d}( \Lambda(\alpha_{\sigma(l)})+\zeta) \geq \sqrt{d}\Lambda(\alpha_{\sigma(l)})$ we conclude that 
\begin{equation} \label{eq:proof_upper_bound_aux8} 
\alpha_x \leq \alpha_{\sigma(l)} \leq \alpha\bigg(\frac{\mu}{\sqrt{d}}\bigg) 
\end{equation}
for all $x \in V_l$, where we used that $\alpha$ defined in \eqref{eq:def_alpha} is the inverse function of $\Lambda$. 
Since $\scalar{\f w}{\f 1_x} = 0$ for all $x \in [N]\setminus V_l$ and $[N] \setminus V_l \subset \cal V_\tau$, we have
\begin{equation} \label{eq:proof_upper_bound_aux4} 
\scalar{\f w}{D\f w} = \sum_{x \in [N] \setminus \cal V_\tau}
\scalar{\f w}{\f 1_{x}}^2 \alpha_x + \sum_{x \in \cal V_\tau \cap V_l}   \scalar{\f w}{\f 1_{x}}^2 \alpha_x 
 \leq \tau +   \Cnu d^{-c ((\Lambda(\alpha_{\sigma(l)}) - 2) \wedge 1)} 
\end{equation}
with very high probability due to Proposition~\ref{pro:delocalization_estimate}, $V_l \cap \cal{V}_\tau \subset \cal{W}_{\tau, \mu,\zeta}$ and \eqref{eq:proof_upper_bound_aux8}. 
Proposition~\ref{pro:delocalization_estimate} is applicable since $\zeta$ and $\mu$ satisfy \eqref{eq:condition_mu_zeta} due to  $\mu \geq \sqrt{d} \big(\Lambda(\alpha_{\sigma(l)}) + \zeta \big)$ and the lower bound on $\Lambda(\alpha_{\sigma(l)})$ implied by the definition of $L_\leq$. 

We use the assumption \eqref{eq:proof_upper_bound_aux5} in \eqref{eq:proof_upper_bound_aux3}, employ \eqref{eq:proof_upper_bound_aux4} and $\tau = 1 + d^{-\theta/4}$ 
and obtain 
\begin{equation} \label{eq:proof_upper_bound_aux6} 
 2 + d^{-\theta/4} + \Cnu d^{-c ((\Lambda(\alpha_{\sigma(l)}) - 2) \wedge 1)}
+ \norm{E} \geq \Lambda(\alpha_{\sigma(l)}) + \zeta -d^{-1/2} \norm{\underline{A}_\tau - \underline{A}}  
\end{equation}
with very high probability. Thus, the bound on $\norm{E}$ in Proposition~\ref{pro:upper_bound_on_adjacency_matrix} yields 
\begin{equation} \label{eq:proof_upper_bound_aux9} 
 d^{-\theta/4} + \Cnu d^{-c((\Lambda(\alpha_{\sigma(l)}) -2)\wedge 1)}  + {\Cnu d^{-1/4 - \theta/2}} - (\Lambda(\alpha_{\sigma(l)}) - 2) \geq \zeta 
\end{equation}
with very high probability. 
Here, we used that 
\begin{equation} \label{eq:proof_upper_bound_aux11} 
 \norm{\underline{A}_\tau - \underline{A}}  \leq \norm{A_\tau - A} + \norm{\mathbb{E}A-\Pi_\tau(\mathbb{E}A)\Pi_\tau}  
  \leq \max_{x \in [N]} D_x^{G\setminus G_\tau} + 2
= \ord\big(d ^{1/4 - \theta/2} \big)
\end{equation}
with very high probability. 
This bound follows from Lemma~\ref{lem:subgraph_separating_large_degrees} \ref{item:subgraph_degrees}, $h(\eps) \geq c \eps^2$, $d \geq (\log N)^{4/(5 - 4 \theta)}$, $\tau = 1 + d^{-\theta/4}$, $\theta \leq 1/2$, and 
\begin{equation} \label{eq:bound_norm_EA_minus_Pi_tau_EA}
\norm{\mathbb{E}A-\Pi_\tau(\mathbb{E}A)\Pi_\tau} \leq 2 
\end{equation}
with very high probability.  
For the proof of \eqref{eq:bound_norm_EA_minus_Pi_tau_EA}, we remark that, by construction of $G_\tau$, all balls $B^{G_\tau}_{r_{x,\tau}}(x)$ for $x \in \cal{V}_\tau$ are disjoint and we have \[
|\cal{Z}_\tau|=\sum_{x\in \cal{V}_\tau} \abs{B^{G_\tau}_{r_{x,\tau}-2}(x)}\leq \frac{2}{d^2}\sum_{x\in \cal{V}_\tau} \abs{B^{G_\tau}_{r_{x,\tau}}(x)}\leq \frac{2N}{d^2}
\] with very high probability, where we used Lemma \ref{lem:concentration_S_i} for the middle inequality. Thus, estimating the operator norm by the Hilbert-Schmidt norm yields 
\[ \norm{\mathbb{E}A-\Pi_\tau(\mathbb{E}A)\Pi_\tau}^2  \leq \frac{d^2}{N^2} \big( \abs{\cal{Z}_\tau}^2 + 2 \abs{\cal{Z}_\tau} (N - \abs{\cal{Z}_\tau}) \big)  \leq  \frac{d^2 2 \abs{\cal{Z}_\tau}N}{N^2} \leq \frac{2d^2 \abs{\cal{Z}_\tau}}{N}  \leq 4\] 
with very high probability. This proves \eqref{eq:bound_norm_EA_minus_Pi_tau_EA} and, hence, \eqref{eq:proof_upper_bound_aux9}. 

The definition of $L_\leq$ in \eqref{eq:def_L_leq_tau_star} and $l \in [L_\leq]$ imply $\alpha_{\sigma(l)} \geq 2 + (\log d)^{-\kappa}$ and, hence,  
 $\Lambda(\alpha_{\sigma(l)}) \geq 2 + C(\log d)^{-2\kappa}$  for some $C>0$. 
Therefore, we can bound the other error terms from \eqref{eq:proof_upper_bound_aux9} by 
$\frac{1}{2}(\Lambda(\alpha_{\sigma(l)})-2)$, multiply the result by 2 and obtain  
\begin{equation}\label{eq:proof_upper_bound_aux10} 
  \Cnu d^{-c((\Lambda(\alpha_{\sigma(l)})-2)\wedge 1)} - (\Lambda(\alpha_{\sigma(l)}) - 2) \geq 2\zeta 
\end{equation}
with very high probability. 
Using $\alpha_{\sigma(l)} \geq 2 + (\log d)^{-\kappa}$ and
 $\Lambda(\alpha_{\sigma(l)}) \geq 2 + C(\log d)^{-2\kappa}$, we see 
that \eqref{eq:proof_upper_bound_aux10}, however, contradicts \eqref{eq:proof_upper_bound_aux10} and $\zeta>0$. 
Therefore, \eqref{eq:proof_upper_bound_aux5} is wrong, which implies part \ref{item:lemma_upper_bound_l} of Proposition~\ref{pro:upper_bound_large_eigenvalues} 
due to \eqref{eq:proof_upper_bound_aux11}, $\theta \leq 1/2$, and $\Lambda(\alpha_{\sigma(l)}) \geq 2 + C(\log d)^{-2\kappa}$.

We now prove \ref{item:lemma_upper_bound_0} assuming $L_\leq = 0$. We follow the proof of \ref{item:lemma_upper_bound_l} with $l=1$ and $V_1 = [N]$ and assume for the proof by contradiction that 
\[\lambda_1(\underline{A}) > \sqrt{d}(2+ \Cnu \zeta) + \norm{\underline{A}_\tau - \underline{A}}, \] 
for some sufficiently large $\Cnu >0 $ and $\zeta \defeq (\log d)^{-2\kappa}$. 

Owing to \eqref{eq:proof_upper_bound_aux11}, we have $\mu\geq\lambda_{1}(\underline{A})-\|\underline{A}_{\tau}-\underline{A}\| \geq2\sqrt{d}+(\mathcal{C}-1)(\log d)^{-2\kappa}$. 
Hence, $\zeta$ obviously satisfies \eqref{eq:lemma_bound_eigenvector_condition_zeta}.  
Moreover, together with $\Lambda(\alpha_{\sigma(1)} \vee 2) \leq 2 + C(\log d)^{-2\kappa}$ for some universal constant $C>0$, we obtain 
\[ \frac{\mu}{\sqrt{d}} -\Lambda(\alpha_{\sigma(1)}\vee2)\geq2+(\mathcal{C}-1)(\log d)^{-2\kappa}-(2+C(\log d)^{-2\kappa})\geq(\mathcal{C}-1 - C)(\log d)^{-2\kappa}\geq\zeta, \] 
where we assumed that $\mathcal{C} \geq 2 + C$ in the last step.  
Hence, $\cal{V}_\tau \subset \cal{W}_{\tau, \mu, \zeta}$. 

Similarly to the arguments in part \ref{item:lemma_upper_bound_l} we obtain 
 \[ \Cnu d^{-c \zeta } - \Cnu \zeta \geq \zeta  \] 
as analogue of \eqref{eq:proof_upper_bound_aux10}. 
This is a contradiction since $\zeta = (\log d)^{-2\kappa}$ which completes the proof of Proposition~\ref{pro:upper_bound_large_eigenvalues}.  
\end{proof}

The condition $\mu \geq \sqrt{d}(\Lambda(\alpha_x \vee 2) + \zeta)$ in the definition of $\cal{W}_{\tau,\mu,\zeta}$ in \eqref{def_calW} is an upper bound on $\alpha_x$. In  
fact, we have $\alpha_x \leq \omega$ for all $x \in \cal{W}_{\tau,\mu,\zeta}$, where the parameter $\omega \equiv \omega(\mu,\zeta)$ is defined as the unique solution in $(2,\infty)$ of 
\begin{equation} \label{def_omega}
\mu = \sqrt{d}(\Lambda(\omega) + \zeta).
\end{equation}
For the following result, we need the definition
\begin{equation} \label{def_r_omega}
r_\tau^\omega \defeq \frac{\log N}{12 \log (\omega d) } \wedge \bigg( \frac{d}{4 \log d} h \bigg(\frac{\tau - 1}{2} \bigg)  - 1 \bigg) - 2.
\end{equation}
Note that $r_\tau^\omega \leq r_x/4$ if $x \in [N] \setminus \cal V_\omega$ and $r_\tau^\omega \leq  r(\tau)/2$ with $r(\tau)$ as in \eqref{eq:def_r_tau}.
Hence, owing to Lemma~\ref{lem:subgraph_separating_large_degrees} \ref{item:subgraph_paths}, the balls $(B_r^{G_\tau}(x))_{x \in \cal V_\tau\setminus \cal V_\omega}$ are disjoint for $r = r_\tau^\omega$. 

\begin{lemma} \label{lem:delocalization_estimate_at_x}
Let $(\log N)^{4/(5-2\theta)}\leq d \leq \mathcal{K}^{-1} N^{1/4}$ for some $\theta \in (0,1/2]$ and let $1 + d^{-\theta/4} \leq \tau \leq 2$. 
Let $\zeta>0$. Let $\mu> 2\sqrt{d}$ be an eigenvalue of $\underline{A}_{\tau,l}$. 
Suppose that $\zeta$ and $\mu$ satisfy \eqref{eq:condition_mu_zeta}.  
Let $\omega \geq 2$ be the unique solution of \eqref{def_omega}. 
Then there exist an $r \in [r_\tau^\omega]$, a universal constant $c>0$ and a constant $\Cnu>0$ such that
for any $x \in \cal W_{\tau, \mu,\zeta}$ and 
any eigenvector $\f w$ of $\underline{A}_{\tau,\ell}$ associated to $\mu$, 
we have 
\[ \alpha(\mu/\sqrt{d}) \scalar{\f w}{\f 1_x}^2 \leq \Cnu d^{-c((\mu/\sqrt{d} - 2)\wedge \log (\mu/\sqrt{d}))} \norm{\f w \vert_{B^{G_\tau}_r (x)}}^2 \] 
with very high probability.  
\end{lemma}

\begin{proof}[Proof of Proposition~\ref{pro:delocalization_estimate}]
Let $\eta \defeq \mu/\sqrt{d}$. Lemma~\ref{lem:delocalization_estimate_at_x} and $r_\tau^\omega \leq (\frac{1}{4} r_x) \wedge ( \frac{1}{2} r(\tau))$ for $\alpha_x \leq \omega$ imply
\[
\alpha(\eta) \sum_{x \in \cal W_{\tau, \mu,\zeta}} \scalar{\f w}{\f 1_x}^2
  \leq \Cnu d^{-c((\eta - 2)\wedge \log (\eta))} \sum_{x \in \cal W_{\tau, \mu,\zeta}} \norm{\f w \vert_{B^{G_\tau}_{r_\tau^\omega}(x)}}^2  \leq 
\Cnu d^{-c((\eta - 2)\wedge \log(\eta))} 
\norm{\f w}^2. 
\] 
Here, we used in the second step that $(B_{r_\tau^\omega}^{G_\tau}(x))_{x \in \cal W_{\tau, \mu, \zeta}}$ are disjoint sets by Lemma~\ref{lem:subgraph_separating_large_degrees} \ref{item:subgraph_paths}. 
As $\f w$ is normalized, Proposition~\ref{pro:delocalization_estimate} follows. 
\end{proof} 

The next subsection is devoted to the proof of Lemma~\ref{lem:delocalization_estimate_at_x}.

\subsection{Proof of Lemma~\ref{lem:delocalization_estimate_at_x}}

For the rest of this section we fix $x \in \cal W_{\tau, \mu,\zeta}$ and omit it from our notation. 
For the proof of Lemma \ref{lem:delocalization_estimate_at_x} we shall need some basic facts about the tridiagonalization of matrices, which are summarized in Appendix \ref{sec:tridiagonalization}, and which we refer to throughout this section. Throughout this section, we only work with indices $i$ of tridiagonal matrices satisfying $i \leq m$, where $m$ was defined in Appendix \ref{sec:tridiagonalization}. This is always a simple application of \eqref{eq:concentration_S_i} in Lemma~\ref{lem:concentration_S_i} and we shall not dwell on this issue further.

In the following result, we compare a tridiagonalization of $\underline{A}_{\tau,l}$ with 
the tridiagonal matrix $M(\alpha)$ for an appropriately chosen $\alpha>0$. 
For $r \in \N$ and $\alpha>0$, the matrix $M(\alpha)$ is defined through 
 \begin{equation} \label{eq:def_M_alpha} 
M(\alpha) =
\begin{pmatrix}
0 & \sqrt{\alpha} &&&&
\\
\sqrt{\alpha} & 0 & 1 &&&
\\
& 1 & 0 & 1 &&
\\&&1 & 0 & \ddots&
\\
&&&\ddots & \ddots & 1
\\
&&&& 1 & 0
\end{pmatrix}\in \R^{(r + 1) \times (r + 1)}.
\end{equation}

\begin{proposition} \label{pro:approx_trimatrix}
Let $(\log N)^{4/(5-2\theta)} \leq d \leq \mathcal{K}^{-1} N^{1/4}$ for some $\theta \in (0,5/2)$  and let 
 $1 + d^{-\theta/4} \leq \tau \leq 2$. Let $r \leq r_\tau^\omega$ and $x \in \cal V_\tau$. For $k \in \N$ we define the error parameter
\begin{equation} \label{def_calE}
\cal E_{\tau,k} \deq \frac{(3\sqrt{\tau}+2)^{k}}{\sqrt{d}}\qbb{\pbb{\log d+\frac{\log N}{d}}\pbb{1+\frac{\log N}{d \tau}}}^{\frac{1}{2}}\, . 
\end{equation}
Let $\wh M \deq M^{(\underline{A}_{\tau,l}, x)}$ be the tridiagonal matrix associated with $\underline{A}_{\tau,l}$ around $x$, and $(\f g_k)$ the associated orthogonal basis (see Appendix \ref{sec:tridiagonalization}). Then there exists a constant $\cal C$ such that if $\cal E_{\tau,r} \leq \frac{1}{2 \cal C}$ then we have for all $k \leq r$ with very high probability
\begin{equation} \label{thm:approx_trimatrix_itemg}
\frac{\|\f g_k- \f 1_{S^{G_\tau}_k}\|}{\norm{\f 1_{S^{G_\tau}_k}}} \leq \Cnu \cal E_{\tau,k},
\end{equation}
and
\begin{equation} \label{thm:approx_trimatrix_itemM}
\|\wh M_{\qq{r}}-\sqrt{d}M(\alpha_x)\| = \cal O \bigg(\sqrt{d}\cal E_{\tau,r} + \frac{\log N}{\tau d^{3/2} } \bigg( 1 + \frac{\log N}{(\tau-1)^2 d} \bigg) \bigg).
\end{equation}
\end{proposition}

We postpone the proof of Proposition~\ref{pro:approx_trimatrix} to the following subsection.

\begin{proof}[Proof of Lemma~\ref{lem:delocalization_estimate_at_x}] 
We denote the standard basis vectors of $\R^{r+1}$ by $\f e_0, \ldots, \f e_r$. Let $\wh M \deq M^{(\underline{A}_{\tau,{l}}, x)}$ be the tridiagonal matrix associated with $\underline{A}_{\tau,{l}}$ around $x$, and $(\f g_k)$ the associated orthogonal basis (see Appendix \ref{sec:tridiagonalization}).
Let $\f w$ be an eigenvector of $\underline{A}_{\tau,{l}}$ with eigenvalue $\mu$. 
We denote by $\f b= (b_k)_{k \in \qq{N-1}}$ the vector representing $\f w$ with respect to the orthonormal basis $(\f g_k/\norm{\f g_k})_k$. 
Then, by the tridiagonal property of $\wh M$, we have $(\wh{M}_{\qq{r}} - \mu I_{r+1}) \f b \in \op{Span}\{\f e_r\}$.
Hence, we can apply Proposition~\ref{pro:bound_eigenvector_perturbation_M_alpha} with $\wt M \deq \frac{1}{\sqrt{d}}\wh M_{\qq{r}}$ to estimate $b_0^2 /\big(\sum_{i=0}^{r} b_i^2 \big)$ once we have verified 
its condition \eqref{eq:bound_eigenvector_condition}.

Because $G_{\tau}\vert_{B^{G_\tau}_r}$ is a tree, $\wt M_{00}=\wt M_{11}=0$ by Lemma~\ref{lem:tridiag_norm_expression} and $\wt M_{01} =\sqrt{\alpha_x}$. 
From \eqref{thm:approx_trimatrix_itemM} with $r = 2$, $\tau \geq 1 + d^{-\theta/4}$, $\log N \leq d^{5/4-\theta/2}$ and $\theta \leq 1/2$, we conclude 
\begin{equation} \label{eq:wt_M_12_eps_2} 
|1- \wt M_{12}|\leq \epsilon_2 \deq \cal C d^{-\theta/2}. 
\end{equation}

Throughout this proof, we need some Lipschitz-type bounds on $\Lambda(\alpha)$ and $\alpha(\eta)$. We have 
\begin{equation} \label{eq:Lambda_alpha_bounds} 
 \eta - \Lambda(\alpha\vee 2) \leq C \bigg( \frac{\alpha(\eta) - 2}{\alpha(\eta)^{3/2}} \bigg)\big ( \alpha(\eta) - \alpha \vee 2\big) 
\end{equation}
for all $\alpha >0$ and $\eta > \Lambda(\alpha \vee 2)$. 
The bound in \eqref{eq:Lambda_alpha_bounds} is a consequence of 
\begin{align*}
 \Lambda(\alpha(\eta)) - \Lambda(\alpha \vee 2)  = \int_{\alpha \vee 2}^{\alpha(\eta)} \Lambda'(t) \dd t = \int_{\alpha \vee 2}^{\alpha(\eta)} \frac{ t -2 }{2(t-1)^{3/2}} \dd t  
\end{align*}
and distinguishing cases $\alpha > \alpha(\eta)/2$ and $\alpha \leq \alpha(\eta)/2$.

With the notation $\eta = \frac{\mu}{\sqrt{d}}$, we calculate $\delta$ introduced in \eqref{eq:def_function_delta} below. From \eqref{eq:wt_M_12_eps_2} we conclude
\begin{equation} \label{eq:bound_delta_delocalization_proof}  
 \abs{\delta(0, \sqrt{\alpha_x}, 0, \wt M_{12}, \eta)} = \frac{\abs{\gamma(\eta) \eta^2 - \gamma(\eta) \alpha_x - \eta\wt{M}_{12}}}{\abs{\alpha_x + \gamma(\eta)\eta \wt{M}_{12} - \eta^2}}
\leq \frac{C\eta}{\abs{\alpha_x - \alpha(\eta)}} 
\end{equation}
for some universal constant $C>0$. Here, 
we used that the numerator is bounded from above by $C \eta$ due to $\gamma(\eta) \eta \leq 2$, \eqref{eq:wt_M_12_eps_2} and 
$\sqrt{\alpha_x} \leq \Lambda(\alpha_x \vee 2) \leq \eta$ as $x \in \cal W_{\tau, \mu, \zeta}$. 
The denominator is bounded from below by $\abs{\alpha_x - \alpha(\eta)} - \gamma(\eta)\eta \abs{1- \wt{M}_{12}}$ which proves \eqref{eq:bound_delta_delocalization_proof} since 
$\alpha(\eta) - \alpha_x \geq \eps_2\geq 2 \gamma(\eta) \eta\abs{1-\wt{M}_{12}}$. 
This last bounds are consequences of 
\begin{equation} \label{eq:proof_delocalization_lemma_aux} 
 \alpha(\eta) - \alpha_x \vee 2 \geq C \frac{\alpha(\eta)^{3/2}}{\alpha(\eta) -2} (\eta -\Lambda(\alpha_x \vee 2)) \geq \cal C d^{-\theta/2} \sqrt{\frac{\eta}{\eta - 2}} \frac{\alpha(\eta)^{3/2}}{\alpha(\eta)-2} \geq \cal C d^{-\theta/2}, 
\end{equation}
as well as \eqref{eq:wt_M_12_eps_2} and $\gamma(\eta)\eta \leq 2$. 
For the proof of \eqref{eq:proof_delocalization_lemma_aux},  
we use \eqref{eq:Lambda_alpha_bounds}, $\eta - \Lambda(\alpha_x \vee 2) \geq \zeta$ and 
the first condition on $\zeta$ in \eqref{eq:lemma_bound_eigenvector_condition_zeta} 
and observe that $t \mapsto \frac{t}{t - 2}$ and $t \mapsto \frac{t^{3/2}}{t -2}$ have strictly 
positive lower bounds for $t > 2$.

Let $\eps$ be defined as in Proposition~\ref{pro:bound_eigenvector_perturbation_M_alpha} below. 
Owing to Proposition~\ref{pro:approx_trimatrix}, we know that $\eps \leq \eps_r$ with very high probability, where $\eps_r$ is defined through 
\begin{equation} \label{eq:def_eps_r} 
 \epsilon_r \defeq \cal C \bigg( \cal E_{\tau, r} + \frac{\log N}{\tau d^{2}}\bigg( 1 + \frac{\log N}{(\tau-1)^2 d} \bigg)\bigg).  
\end{equation}

We now verify that the choice 
\begin{equation}\label{eq:delocalization_choice_r} 
 r = \bigg( c \log \Big( \zeta d^{1/5 + \theta/2} \sqrt{(\eta - 2)/\eta} \Big) + \tilde{\Cnu} \bigg) \wedge \frac{\log (d ^{1/5 + \theta/2} ) - \log (4 \cal{\Cnu})}{\log 8} 
\end{equation}
for some sufficiently small universal constant $c >0$ and some $\tilde{\Cnu} >0$ implies
\begin{equation}\label{eq:condition_decay_1}
 \eta^2 \geq 4 + \frac{C(1+\eta)^2\eps_r^2}{(1 - \gamma(\eta)^2)^2} \big ( 1 + 1 \vee \delta^2\big), 
\qquad \eps_r \leq 1/2, 
\end{equation}
thus, justifying the conditions of Proposition~\ref{pro:bound_eigenvector_perturbation_M_alpha}. 
We remark that $r\geq 1$ due to the lower bound on $\zeta$ from \eqref{eq:lemma_bound_eigenvector_condition_zeta} 
and the lower bound on $d$. Clearly, $r \leq r_\tau^{\omega}$.  
From the definitions of $\eps_r$ and $\cal{E}_{\tau,r}$ in \eqref{eq:def_eps_r} and \eqref{def_calE}, 
respectively, as well as $1 + d^{-\theta/4}\leq \tau \leq 2$ and $\log N \leq d^{5/4 - \theta/2}$, we conclude 
\begin{equation} \label{eq:proof_delocalization_lemma_eps_r}  
\eps_r \leq \Cnu \Big( 8^r d^{-1/5 - \theta/2} + \frac{d^{-\theta}}{(\tau - 1)^2} \Big) \leq \Cnu \Big( 8^r d^{-1/5- \theta/2} + d^{-\theta/2} \Big). 
\end{equation}
Hence, the second bound on $r$ in \eqref{eq:delocalization_choice_r} yields $\eps_r \leq 1/2$ in \eqref{eq:condition_decay_1}. 

The first bound in \eqref{eq:condition_decay_1} is equivalent to 
\begin{equation} \label{eq:condition_decay_2} 
 \eps_r^2 \leq \frac{ (\eta^2- 4) (1-\gamma(\eta)^2)^2}{C\eta^2 (1 \vee \delta^2)}.  
\end{equation}
 Owing to $8^r \leq \ee^{r/c}$ for a sufficiently small universal $c>0$, 
the definition of $r$ in \eqref{eq:delocalization_choice_r} and the first condition on $\zeta$ in \eqref{eq:lemma_bound_eigenvector_condition_zeta}, 
we obtain from \eqref{eq:proof_delocalization_lemma_eps_r} that 
 \[ \eps_r  \leq \Cnu \sqrt{\frac{\eta -2}{\eta}}\, \zeta    
\leq \Cnu \sqrt{\frac{\eta -2}{\eta}} \frac{(\alpha(\eta) - 2) (\alpha(\eta) - \alpha_x)} {\eta^3}  
  \leq \Cnu \sqrt{\frac{\eta^2 - 4}{\eta^2}} \frac{1-\gamma(\eta)^2}{\abs{\delta}}.  \] 
Here, the second step is a consequence of $\zeta \leq \eta - \Lambda(\alpha_x \vee 2)$, \eqref{eq:Lambda_alpha_bounds} and $\alpha(\eta) \geq \eta^2/2$ for $\eta > 2$. 
The third step follows from 
$(\alpha(\eta) - 2)/\eta^2\leq 1 - \gamma(\eta)^2$ and \eqref{eq:bound_delta_delocalization_proof}. 
This proves \eqref{eq:condition_decay_2} and, thus, the remaining estimate in~\eqref{eq:condition_decay_1}. 
Hence, we have justified the conditions of Proposition~\ref{pro:bound_eigenvector_perturbation_M_alpha}. It implies 
\[ 
\alpha(\eta) \frac{b_0^2}{\sum_{i=0}^{r} b_i^2} \leq \frac{C \alpha_x\alpha(\eta)}{(\alpha_x - \alpha(\eta))^2} \big( \gamma_\geq\big)^{-2r}   
  \leq \frac{C}{\zeta^2} \bigg( 1 + \frac{1}{4} (\eta - 2)\bigg)^{-2(r-2)}. 
 \] 
Here, we employed $\abs{\alpha(\eta) -\alpha_x} \geq 2 \gamma(\eta) \eta \abs{1- \wt{M}_{12}}$ in the first step. 
In the second step, we used $\zeta \leq C \alpha(\eta)^{-1/2} ( \alpha(\eta) - \alpha_x)$ due to 
\eqref{eq:Lambda_alpha_bounds} and $x \in \cal{W}_{\tau,\mu,\zeta}$. 
We also used \eqref{eq:lower_bound_gamma_geq} as well as $\alpha_x \leq \alpha(\eta) \leq \eta^2$ due to $x \in \cal{W}_{\tau,\mu,\zeta}$. 

Thus, our choice of $r$ in \eqref{eq:delocalization_choice_r}, the first condition in \eqref{eq:lemma_bound_eigenvector_condition_zeta} and the definition of $\eta$ yield 
\begin{equation} \label{eq:proof_delocalization_last_step} 
\begin{aligned} 
   \alpha(\eta) \frac{b_0^2}{\sum_{i=0}^{r} b_i^2} & \leq \frac{C}{\zeta^2} \exp\bigg(-2(r-2) \log \bigg( 1  + \frac{1}{4} (\eta - 2) \bigg)\bigg) \\ 
 & \leq \frac{\Cnu}{\zeta^2} \exp \bigg(- \frac{c}{2} \log d \log \bigg( 1 + \frac{1}{4} ( \eta - 2 ) \bigg) \bigg) \\ 
& \leq \Cnu \exp\Big( - c \log d \big( (\eta - 2) \wedge \log \eta  \big) \Big). 
\end{aligned} 
\end{equation} 
Here, we used the second condition on $\zeta$ in \eqref{eq:lemma_bound_eigenvector_condition_zeta} for a sufficiently small $c$ in the third step. 

Finally, Lemma~\ref{lem:delocalization_estimate_at_x} follows from \eqref{eq:proof_delocalization_last_step} 
since $b_0 = \scalar{\f w}{\f 1_x}$ and  
\[
\| \f w \vert_{B^{G_\tau}_r}\|^{2} \geq \sum_{i=0}^r \scalarbb{\frac{\f g_i}{\|\f g_i\|}}{\f w}^2= \sum_{i=0}^r{b}_{i}^{2}. 
\]
Here, we used that $(\f g_i)_{i \in \qq{r}}$ is a family of orthogonal vectors in $\text{Span} \{\f 1_y: y\in B^{G_\tau}_r \} $ to obtain the previous inequality. 
\end{proof}

\subsection{Proof of Proposition~\ref{pro:approx_trimatrix}}

\begin{proof}[Proof of Proposition~\ref{pro:approx_trimatrix}]
We first remark that $A_\tau$ and $\underline{A}_{\tau,l}$ agree in the vicinity of $x \in \cal{V}_\tau \cap V_l$ in the sense that 
\[ (\underline{A}_{\tau,l})^i \f 1_x = (A_\tau)^i \f 1_x \] 
for all $i \in \qq{r_{x,\tau}-2}$.  
This follows from first verifying the same identity with $A_\tau$ replaced by $\underline{A}_\tau$ and then using that the shift in the definition of $\underline{A}_\tau$ in \eqref{eq:def_underline_A_tau} vanishes on $B_{r_{x,\tau}-2}^{G_\tau}(x)$. 

For the proof of \eqref{thm:approx_trimatrix_itemg}, we now introduce a second family $(\f f_k)_k$ of vectors that will turn out to be a good approximation of $(\f g_k)_k$. 
The vectors $\f f_k$ are defined through 
\[
\f f_{0}  =\f 1_{x}, \qquad \f f_{1}  =\f 1_{S_{1}^{G_\tau}}, \qquad \f f_{2}  =\f 1_{S_{2}^{G_\tau}},\qquad \f f_{k+2}  =Q_0\big(A_{\tau}\f f_{k+1}-d\f f_{k}\big)
\]
for all $k\geq1$. 
Here and in the following, $Q_i$ denotes the orthogonal projection on $\text{Span}\{(A_\tau)^j\f 1_x\col j \in \qq{i}\}^\perp $ 
as in Appendix \ref{sec:tridiagonalization}. 

The careful analysis of $\f f_k$ presented below will imply \eqref{thm:approx_trimatrix_itemg} due to the bound 
\begin{equation} \label{eq:bound_g_k_S_k_leq_q_k} 
 \norm{\f g_k - \f 1_{S_k^{G_\tau}}} \leq \norm{\f q_k},  
\end{equation}
where we introduced 
\[ \f q_k \defeq \f f_k - \f 1_{S_k^{G_\tau}}. \] 

Before estimating $\f q_k$, we now establish \eqref{eq:bound_g_k_S_k_leq_q_k}.  
It is easy to check that there exists a monic polynomial $P_k$ of degree $k$ such that $\f f_k = P_k(A_\tau)\f 1_x$ and then  
\begin{equation} \label{eq:g_k_equals_Q_k_minus_1_of_f_k} 
 \f g_{k}=Q_{k-1}(A_\tau^k\f 1_x) =Q_{k-1}(P_k(A_\tau) \f 1_x) = Q_{k-1} \f f_{k}.
\end{equation}
Hence, $\norm{\f g_{k}} \leq \norm{\f f_{k}}$ and, thus, we have
\[\norm{\f g_k - \f 1_{S_k^{G_\tau}}}^2 = \norm{\f g_{k}}^2 - \norm{\f 1_{S_{k}^{G_\tau}}}^2 \leq \norm{\f f_{k}}^2 - \norm{\f 1_{S_{k}^{G_\tau}}}^2=\norm{\f q_{{k}}}^2. \]
Here, we used in the first step that $\f g_k - \f 1_{S_k^{G_\tau}}$ is orthogonal to $\f 1_{S_k^{G_\tau}}$. This is a consequence of $\supp(\f g_k - \f 1_{S_k^{G_\tau}}) \subset B^{G_\tau}_{k-2}$ 
by Lemma~\ref{lem:g_i_for_tree}. 
In the last step, we used $\f q_k \perp \f 1_{S_k^{G_\tau}}$ which 
follows from $\supp \f q_k \subset B_{k-1}^{G_\tau}$, a consequence of $\scalar{\f f_k}{\f 1_y} = \scalar{\f g_k}{\f 1_y} = 1$ for all $y \in S_k^{G_\tau}$ by \eqref{eq:g_k_equals_Q_k_minus_1_of_f_k}. 
This shows~\eqref{eq:bound_g_k_S_k_leq_q_k}. 

Owing to \eqref{eq:bound_g_k_S_k_leq_q_k}, the bound in \eqref{thm:approx_trimatrix_itemg} follows directly from 
\begin{equation} \label{eq:bound_q_k} 
 \norm{\f q_k} \leq \tilde{c}_k \norm{\f 1_{S_k^{G_\tau}}} 
\end{equation} 
which holds with very high probability for all $k \leq r_{x,\tau}-2$ as we shall show below. 
Here, $\tilde{c}_k$ is the unique solution of 
\[ 
\tilde{c}_{k+2}= 2\tilde{c}_k+3\sqrt{\tau}\tilde{c}_{k+1}+\frac{\Cnu}{\sqrt{d}}\bigg(\bigg(\log d+\frac{\log N}{d}\bigg)^{1/2}\bigg(1+\frac{\log N}{D_{x}}\bigg)^{1/2}\bigg)
\] 
with the initial choices $\tilde{c}_1 = 0$ and $\tilde{c}_2 =0$.

We now prove \eqref{eq:bound_q_k} by induction on $k$. The induction basis for $k=1$ and $k=2$ is trivial.  
For the induction step, we decompose
\begin{align*}
A_{\tau}\f f_{k+1}-d\f f_{k} & =A_{\tau}(\f 1_{S_{k+1}^{G_\tau}}+ \f q_{k+1})-d(\f 1_{S_{k}^{G_\tau}}+ \f q_{k})\\
 & =\f 1_{S_{k+2}^{G_\tau}}+\sum_{y\in S_{k}^{G_\tau}}(N_{k+1}(y)-d)\f 1_{y}+A_{\tau} \f q_{k+1}-d \f q_{k}\\
 & =\f 1_{S_{k+2}^{G_\tau}}+\sum_{y\in S_{k}^{G_\tau}}\bigg(N_{k+1}(y)-\frac{|S_{k+1}|}{|S_{k}|}\bigg)\f 1_{y}+\bigg(\frac{|S_{k+1}|}{|S_{k}|}-d\bigg)\f 1_{S_{k}^{G_\tau}}+A_{\tau} \f q_{k+1}-d \f q_{k}. 
\end{align*}
Here, we used in the second step that $\scalar{\f 1_y}{A_\tau \f 1_{S_{k+1}^{G_\tau}}}
= \abs{S_1^{G_\tau}(y)\cap S_{k+1}^{G_\tau}} = N_{k+1}(y)$ for all $y \in S_k^{G_\tau}$ by Lemma~\ref{lem:subgraph_separating_large_degrees} \ref{item:subgraph_N_s_agree} as $k >1$.

Therefore, following the proof of \eqref{eq:estimate_w_2} yields
\begin{equation}\label{eq:proof_approx_ideal_tridiagonal_aux1} 
\begin{aligned}
\frac{1}{\sqrt{|S_{k+1}|}}\normbb{\sum_{y\in S_{k}^{G_\tau}}\bigg(N_{k+1}(y)-\frac{|S_{k+1}|}{|S_{k}|}\bigg)\f 1_{y}}& \leq \frac{1}{\sqrt{|S_{k+1}|}} \normbb{\sum_{y\in S_{k}}\bigg(N_{k+1}(y)-\frac{|S_{k+1}|}{|S_{k}|}\bigg)\f 1_{y}} \\& = \ord\bigg(\bigg(\log d+\frac{\log N}{d}\bigg)^{1/2}\bigg(1+\frac{\log N}{D_{x}}\bigg)^{1/2}\bigg)
\end{aligned}
\end{equation}
with very high probability. Moreover using Lemma \ref{lem:concentration_S_i} and Lemma~\ref{lem:subgraph_separating_large_degrees} \ref{item:subgraph_inclusion_S_i} we get 
\begin{align*}
 & \frac{1}{\sqrt{|S_{k+1}|}}\normbb{\bigg(\frac{\abs{S_{k+1}}}{\abs{S_{k}}}-d\bigg)\f 1_{S_{k}^{G_\tau}}} \leq \frac{\sqrt{\abs{S_k}}}{\sqrt{|S_{k+1}|}}\absbb{\frac{\abs{S_{k+1}}}{\abs{S_{k}}}-d}\leq \mathcal{O}\bigg(\bigg[\frac{\log N}{dD_{x}}\bigg]^{\frac{1}{2}}\bigg)
\end{align*}
for $k \geq 2$, 
which is smaller than the right-hand side of \eqref{eq:proof_approx_ideal_tridiagonal_aux1}. 
Since $A_{\tau}$ has degree at most $\tau d$ on $B^{G_\tau}_{r} \setminus \{x\}$
then $\|A_{\tau}\f v\|\leq2\sqrt{\tau d}\|\f v\|$ for all $\f v$ with $\supp(\f v) \subset B^{G_\tau}_r$
and $\scalar{\f v}{\f 1_x}=0$ \cite[Chap.~11, Ex.~14]{Lov93}. 
Therefore we have 
\[
\|A_{\tau}\f q_{k+1}\|= \|A_{\tau}Q_0\f q_{{k+1}}\|\leq2\sqrt{\tau d}\|\f q_{{k+1}}\|.
\]
We put everything together and get
\[ \begin{aligned} 
\|\f q_{{k+2}}\|& =\norm{Q_0 (\f f_{k+2} - \f 1_{S_{k+2}^{G_\tau}})}\\ 
 & \leq d\|\f q_{{k}}\|+2\sqrt{\tau d}\|\f q_{{k+1}}\|+\sqrt{|S_{k+1}|}\mathcal{O}\bigg(\bigg[\bigg(\log d+\frac{\log N}{d}\bigg)\bigg(1+\frac{\log N}{D_{x}}\bigg)\bigg]^{\frac{1}{2}}\bigg)
\end{aligned} \]
with very high probability. 
We set $c_k = \frac{\|\f q_{{k}}\|}{\sqrt{|S_k|}}$. Thus, the previous estimate implies 
\[
c_{k+2}\leq \frac{d\sqrt{|S_{k}|}}{{\sqrt{|S_{k+2}|}}}c_k+\frac{2\sqrt{\tau d|S_{k+1}|}}{\sqrt{|S_{k+2}|}}c_{k+1}+\frac{\sqrt{|S_{k+1}|}}{\sqrt{|S_{k+2}|}}\mathcal{O}\bigg(\bigg[\bigg(\log d+\frac{\log N}{d}\bigg)\bigg(1+\frac{\log N}{D_{x}}\bigg)\bigg]^{\frac{1}{2}}\bigg).
\]
We use Lemma \ref{lem:concentration_S_i} and obtain 
\[
c_{k+2}\leq (1+o(1))c_k+2\sqrt{\tau}(1+o(1))c_{k+1}+\frac{1}{\sqrt{d}}\mathcal{O}\bigg(\bigg[\bigg(\log d+\frac{\log N}{d}\bigg)\bigg(1+\frac{\log N}{D_{x}}\bigg)\bigg]^{\frac{1}{2}}\bigg). 
\]
This completes the induction step. Thus, we have proved \eqref{eq:bound_q_k} and, hence, \eqref{thm:approx_trimatrix_itemg} as well.

We now verify \eqref{thm:approx_trimatrix_itemM}. We have \[
\|\f g_i\|^2=|S_i|\bigg(1+\bigg(\frac{\|\f g_i\|^2}{|S_i^{G_\tau}|}-1\bigg)\bigg)\bigg(1+\bigg(\frac{|S_i^{G_\tau}|}{|S_i|}-1\bigg)\bigg)=|S_i|\bigg(1+\mathcal{O}\bigg(\cal E_{\tau,r}^{2} + \frac{\log N}{\tau d^2 } \bigg( 1 + \frac{\log N}{(\tau-1)^2 d} \bigg) \bigg)\bigg). 
\]
Here, we used \eqref{thm:approx_trimatrix_itemg} and the orthogonality of $\f g_i$ and $\f g_i - \f 1_{S_i^{G_\tau}}$ to see that the second factor equals $1 + \ord(\cal E_{\tau,r}^2)$. 
To estimate the third factor, we used 
\eqref{eq:bound_S_i_setminus_S_i_tau}, Lemma~\ref{lem:subgraph_separating_large_degrees} \ref{item:subgraph_inclusion_S_i}, 
\eqref{eq:concentration_S_i} 
and $D_x \geq \tau d$. 
Using the assumption on $\cal E_{\tau,r}$, the lower bound $\tau \geq 1 + d^{-\theta/4}$ 
and the lower bound on $d$, we see that error term in this expansion of $\norm{\f g_i}^2$ 
is smaller than 1. 
Therefore, denoting the entries of $\wh{M}$ by $\wh{M}_{ij}$ and 
using Lemma~\ref{lem:tridiag_norm_expression} yield  
\[ \wh M_{i\text{ }i+1} = \frac{\norm{\f g_{i+1}}}{\norm{\f g_{i}}} = \sqrt{\frac{|S_{i+1}|}{|S_i|}}\Big(1+\mathcal{O}\bigg(\cal E_{\tau,r}^{2} + \frac{\log N}{\tau d^2} \bigg( 1 + \frac{\log N}{(\tau-1)^2 d} \bigg) \bigg)\bigg). 
 \]
Therefore, \eqref{eq:proof_w_3_aux_estimate} 
and $\wh{M}_{ii} =0$ by Lemma~\ref{lem:tridiagonal_bipartite_graph} \ref{item:tridiagonal_bipartite_vanishing_diagonal} 
complete the proof of \eqref{thm:approx_trimatrix_itemM} and, thus, the one of Proposition~\ref{pro:approx_trimatrix}. 
\end{proof}

\section{Proofs of the results in Section \ref{sec:results}} \label{sec:proofs_finish}

In this short section we state how to conclude the results of Section \ref{sec:results}. For $d \leq \exp\pb{\sqrt{\log N} / 4}$, Theorem \ref{thm:correspondence_eigenvalue_large_degree} follows from Propositions \ref{pro:lower_bound_number_outliers} and \ref{pro:upper_bound_large_eigenvalues}, noting that $L = L_\leq \leq L_\geq$. For $\exp\pb{\sqrt{\log N} / 4} \leq d \leq N/2$, Theorem \ref{thm:correspondence_eigenvalue_large_degree} follows immediately from \cite[(2.4) and Theorem 2.6]{BBK1}. Corollary \ref{cor:non_centred} follows from Theorem \ref{thm:correspondence_eigenvalue_large_degree} by eigenvalue interlacing, $\lambda_l(\ul A) \geq \lambda_{l + 1}(A) \geq \lambda_{l+1}(\ul A)$ for $1 \leq l \leq N - 1$, as well as the mean value theorem. Finally, the proof of Theorem \ref{thm:eigenvalues_general_sparse_matrices} is very similar to that of Theorem \ref{thm:correspondence_eigenvalue_large_degree}, and we explain the needed minor modifications in the next section.

\section{Modifications for sparse Wigner matrices} \label{sec:general_sparse_random_matrices} 

In this section, we explain how the arguments in the previous sections can be adapted to yield the proof of Theorem~\ref{thm:eigenvalues_general_sparse_matrices}. 
We consider a sparse Wigner matrix $X$ with entries $X_{xy} = W_{xy} A_{xy}$. Here, $A$ is the adjacency matrix of an Erd{\H o}s-R\'enyi graph $G$ on $[N]$
with edge probability $d/N$ and $W=(W_{xy})_{x,y\in [N]}$ is an independent Wigner matrix with bounded entries. That is, $W$ is Hermitian and the random variables $(W_{xy} \col x \leq y)$ are independent and 
\begin{equation}\label{eq:assumptions_on_W} 
 \E W_{xy} = 0, \qquad \E \abs{W_{xy}}^2 = 1, \qquad \abs{W_{xy}} \leq K 
\end{equation}
for all $x,y \in [N]$ and some constant $K>0$.

The assumptions imply that $X$ is symmetric and we consider $X$ as the adjacency matrix of an undirected weighted graph with edge weights $W_{xy}$. 
According to this philosophy, we define 
\[ S_i(x) \defeq \{ y \in [N] \col \min\{ j \geq 0 \col (X^j)_{xy} \neq 0 \} = i \}, \qquad B_i(x) \defeq \bigcup_{j \in \qq{i}} S_j(x)  \] 
for all $x \in [N]$.

In the remainder of this section, we explain the necessary adjustments in order to conclude Theorem~\ref{thm:eigenvalues_general_sparse_matrices} along the proof of Theorem~\ref{thm:correspondence_eigenvalue_large_degree} 
with the definition of $\alpha_x$ from \eqref{def_alpha_general} and $D_x = \alpha_x d$.
Throughout this section, the constant $\Cnu$ as well as the implicit constant in $\ord$ are also allowed to depend on $K$, the uniform bound on $W_{xy}$ in \eqref{eq:assumptions_on_W}. 
With this convention, the arguments in Section~\ref{sec:proof_upper_bound_adjacency} to Section~\ref{sec:upper_bound_eigenvalues} do not require any changes. They only have to be understood with respect to the new 
definition of $\Cnu$ and $\ord$. 
The necessary modifications of Section~\ref{sec:large_degree_induces_eigenvalues} are explained in the following subsection. 
Once they are taken into account Theorem~\ref{thm:eigenvalues_general_sparse_matrices} follows from Propositions~\ref{pro:lower_bound_number_outliers} and \ref{pro:upper_bound_large_eigenvalues}. 

\subsection{Modifications in Section \ref{sec:large_degree_induces_eigenvalues}} 

In this subsection, we fix $x \in [N]$ and explain the modifications required in Section~\ref{sec:large_degree_induces_eigenvalues} to obtain the corresponding results in the setup described above. 
\paragraph{Definition of the approximate eigenvector, decomposition of the error terms} 
We now introduce the analogue of the approximate eigenvector $\f v$ from \eqref{eq:def_v_x} in the present setup. 
We define $\f g_0 \defeq \f 1_x$. For $i \geq 1$, we define 
\[ \f g_i \defeq (X \f g_{i-1})\vert_{S_i(x)}. \] 
Note that $\f g_0, \ldots, \f g_i$ are orthogonal. As a heuristic for the following argument, suppose that for some $r \geq 1$ the graph $G$ restricted to $B_r(x)$ is a tree and that for all $i \in [r]$ we have $\sum_{z \in S_{i+1}(x)} \scalar{\f 1_z}{X \f 1_y}^2 = d$ for all $y \in S_i(x)$; then the upper-left $(r+1) \times (r+1)$ block of the tridiagonal matrix associated with $X$ around $x$ (see Appendix \ref{sec:tridiagonal_tree}) is equal to the right-hand side of \eqref{eq:block_structure_M} with $D_x = d \alpha_x$ and $\alpha_x$ given by \eqref{def_alpha_general}. This motivates the construction of $\f v$ in the following paragraph.

With the choices of $u_i$ from \eqref{eq:def_v_x_coeff}, we set 
\[ \f v = \sum_{i=0}^r u_i \hat{\f g}_i,\qquad \qquad \hat{\f g}_i \defeq \frac{\f g_i}{\norm{\f g_i}}.  \] 
Similarly to the proof of Lemma~\ref{lem:decomposition_app_eigenvector}, we obtain
\[ (X - \sqrt{d} \Lambda(\alpha_x)) \f v = \f w_1 + \ldots  + \f w_4, \] 
where the error terms $\f w_1, \ldots, \f w_4$ are defined through 
\[\begin{aligned} 
 \f w_1 & \defeq \sum_{i=0}^r \frac{u_i}{\norm{\f g_i}} \bigg( \sum_{y \in S_{i+1}} \Big(N_i(y) - \scalar{\f 1_y}{\f g_{i+1}} \Big) \f 1_y + \sum_{y \in S_i} N_i(y) \f 1_y \bigg),\\ 
 \f w_2 & \defeq \sum_{i=1}^r \frac{u_i}{\norm{\f g_i}} \bigg(\sum_{y \in S_{i-1}}  N_i(y)\f 1_y - \frac{\norm{\f g_i}^2}{\norm{\f g_{i-1}}^2} {\f g}_{i-1} \bigg), \\ 
\f w_3 & \defeq  
 u_2 \left( \frac{\norm{\f g_2}}{\norm{\f g_1}} - \sqrt{d} \right) \hat{\f g}_1 + \sum_{i=2}^{{r} - 1} \left[ u_{i+1} \left(\frac{\norm{\f g_{i+1}}}{\norm{\f g_i}} -\sqrt{d} \right) 
 + u_{i-1} \left( \frac{\norm{\f g_i}}{\norm{\f g_{i-1}}} - \sqrt{d} \right)  \right] \hat{\f g}_i , \\ 
 \f w_4 & \defeq \left( u_{{r}-1} \frac{\norm{\f g_{r}}}{\norm{\f g_{r-1}}}  - u_{{r}-1} \sqrt{d}  - u_{{r}+1} \sqrt{d} \right) \hat{\f g}_{r}
 + u_{r} \frac{\norm{\f g_{{r}+1}}}{\norm{\f g_{r}}} \, \hat{\f g}_{{r}+1}.
\end{aligned} \] 
Here, $N_i(y) \defeq \scalar{\f 1_y}{X \f g_i}$ for all $y\in [N]$. 
We remark that the analogue of $\f w_0$ vanishes as the entries $W_{yz}$ are centred for all $y,z \in [N]$. 

\paragraph{Concentration of $\norm{\f g_{i+1}}/\norm{\f g_i}$}  
In order to establish that the ratio $\frac{\norm{\f g_{i+1}}}{\norm{\f g_i}}$ concentrates around $\sqrt{d}$ if $i \geq 1$ 
we follow the proof of Lemma~\ref{lem:concentration_S_i}. 
It suffices to verify \eqref{eq:concentration_S_i_plus_1_general_eps} in the new setup. 
 We first prove by induction that if $K$ is the uniform upper bound on the entries of $W$ then 
\begin{equation} \label{eq:sup_norm_bound_g_i}
 \norm{\f g_i}_\infty \leq (\Cnu K)^i 
\end{equation}
for $i \geq 0$ with very high probability. 
The case $i = 0$ holds trivially. 
Since $\supp \f g_{i-1} \subset S_{i-1}$ by definition the definition of $\f g_i$ implies for $y \in S_i$ that 
\[ \absbb{\sum_{z \in [N]} X_{yz} \scalar{\f 1_z}{\f g_{i-1}}} = \absbb{\sum_{z \in S_{i-1}} W_{yz} A_{yz} \scalar{\f 1_z}{\f g_{i-1}}} 
\leq \norm{\f g_{i-1}}_\infty K  N_{i-1}(y) \leq \norm{\f g_{i-1}}_\infty K \cal C,   \] 
where we used Corollary~\ref{cor:tree_approximation} in the last step. 
As $\supp \f g_i \subset S_i$ by definition, this proves \eqref{eq:sup_norm_bound_g_i}. 

Using Bennett's inequality, it is easy to see that 
\begin{equation} \label{eq:sparse_wigner_concentration_S_i_general_p} 
 \absbb{\frac{\norm{\f g_{i+1}}^2}{d \norm{\f g_i}^2} - 1} \leq \Cnu(1 + K^2) \sqrt{\frac{ \log N}{d}} \frac{\norm{\f g_i}_{4}^2}{\norm{\f g_i}^2} 
\end{equation}
with very high probability. 
Using $\norm{\f g_i}_4^2 \leq (\Cnu K)^i \norm{\f g_i}_2$ by \eqref{eq:sup_norm_bound_g_i} 
in \eqref{eq:sparse_wigner_concentration_S_i_general_p} 
yields the desired analogue of \eqref{eq:concentration_S_i_plus_1_general_eps} in the setup of sparse Wigner matrices.   
This proves the analogue of Lemma~\ref{lem:concentration_S_i}.

\paragraph{Estimate on $\f w_1$} 

We remark that the $i=0$ contribution in the definition of $\f w_1$ vanishes as 
$N_0(y) = \scalar{\f 1_y}{X \f g_0} = \scalar{\f 1_y}{\f g_1}$ for any $y$ as $A_{xx} = 0$. Moreover,  $A_{xx} =0$ also 
implies that $N_0(y) = 0$ for $y \in S_0$. 
Hence, 
\[ \norm{\f w_1}^2 = \normbb{ \sum_{i=1}^r \frac{u_i}{\norm{\f g_i}} \sum_{y \in S_i} N_i(y) \f 1_y}^2 = \sum_{i=1}^r \frac{u_i^2}{\norm{\f g_i}^2} \sum_{y \in S_i} N_i(y)^2. \] 
Here, we also used that 
$ N_i(y) - \scalar{\f 1_y}{\f g_{i+1}} = \scalar{\f 1_y}{(X \f g_i)\vert_{[N] \setminus S_{i+1}(x)}} = 0$
for any $y \in S_{i+1}$ due to the fact that $(X\f g_i)\vert_{[N] \setminus S_{i+1}(x)}$ vanishes on $S_{i+1}(x)$. 

Thus, in order to estimate $\norm{\f w_1}^2$, we use the following version of \eqref{eq:Nloops} in Corollary \ref{cor:corollary_tree_approximation}. 
Namely, for all $i \geq 1$, the bounds
\[ \sum_{y \in S_i(x)} N_i(y)^2 = \sum_{y \in S_i(x)} \scalar{\f 1_y}{X \f g_i}^2 = \sum_{y \in S_i(x)} \bigg( \sum_{y_1 \in S_i(x)} \scalar{X \f 1_y}{\f 1_{y_1}} \scalar{\f 1_{y_1}}{\f g_i} \bigg)^2 \leq \Cnu \norm{\f g_i}_\infty^2 \] 
hold with very high probability.
In the last step, we used Lemma~\ref{lem:tree_approximation} to conclude that there are at most $\ord(1)$ many nonzero terms. 
Therefore, \eqref{eq:sup_norm_bound_g_i} yields \eqref{eq:estimate_w_1} in the current setup due to the growth of $\norm{\f g_i}$ by the analogue of \eqref{eq:concentration_S_i}.

\paragraph{Estimate on $\f w_2$} 
Here, we follow the proof of \eqref{eq:estimate_w_2}. By the Pythagorean theorem, we have 
\[ \norm{\f w_2}^2 = \sum_{i=1}^{r - 1} \frac{u_{i+1}^2}{\norm{\f g_{i+1}}^2} \sum_{y \in S_{i}} \bigg( N_{i + 1}(y) - \frac{\norm{\f g_{i + 1}}^2}{\norm{\f g_{i}}^2} \scalar{\f 1_y}{\f g_{i}} \bigg)^2, \] 
where we used that $S_0 = \{ x\}$, $N_1(x) = \scalar{X \f 1_x}{\f g_1} = \norm{\f g_1}^2$ as $\f g_1 = X \f 1_x$ due to $A_{xx}=0$.  
As $\sum_{i=1}^{r-1} u_{i+1}^2 \leq 1$, we obtain 
 \[
 \norm{\f w_2}^2 \leq \max_{i \in [r-1]} \frac{1}{\norm{\f g_{i+1}}^2} \sum_{y \in S_{i}}\bigg( N_{i+1}(y) - \frac{\norm{\f g_{i+1}}^2}{\norm{\f g_{i}}^2} \scalar{\f 1_y}{\f g_{i}} \bigg)^2 \leq \,  \frac{4}{d} \max_{i \in [r-1]}
\big( Z_i + \tilde{Y}_i\big) 
\] 
where we introduced
\[ \begin{aligned} 
Z_i & \defeq \frac{1}{\norm{\f g_i}^2}\sum_{y \in S_{i}} \bigg( N_{i+1}(y) - \E[ N_{i+1}(y) \vert X_{(B_{i-1})} ] \bigg)^2, \\ 
 \tilde{Y}_i & \defeq \frac{1}{\norm{\f g_i}^2}\sum_{y \in S_i} \bigg( \E[ N_{i+1}(y) \vert X_{(B_{i-1})}] - \frac{\norm{\f g_{i+1}}^2}{\norm{\f g_i}^2} \scalar{\f 1_y}{\f g_i} \bigg)^2. 
\end{aligned} \]  
We first estimate $\tilde{Y}_i$. As $\E[ N_{i+1}(y) \vert X_{(B_{i-1})}] = \scalar{\f 1_y}{\f g_i} d (1 - \frac{\abs{B_i}}{N})$, we conclude   
\[ \tilde{Y}_i = \frac{1}{\norm{\f g_i}^2}\bigg ( d - \frac{\norm{\f g_{i+1}}^2}{\norm{\f g_i}^2} - \frac{d \abs{B_i}}{N} \bigg)^2   \sum_{y \in S_i} \scalar{\f 1_y}{\f g_i}^2 
  \leq 2  \bigg( \bigg( d - \frac{\norm{\f g_{i+1}}^2}{\norm{\f g_i}^2} \bigg)^2 + 1 \bigg) 
\leq \Cnu  \bigg( \frac{d \log N}{D_x} + 1 \bigg) 
\] 
with very high probability for all $i \in [r-1]$. 

In order to estimate $Z_i$, we follow the proof of \eqref{eq:Z_claim} and explain the necessary changes. 
We redefine 
\[ E_y \defeq \frac{1}{\scalar{\f 1_y}{\f g_i}} \Big( N_{i+1}(y) - \E [ N_{i+1}(y) \vert X_{(B_{i-1})} ]\Big). \] 
and use Bennett's inequality to obtain 
\[ \P \Big( E_y^2 > (s \kappa d)^2 \Bigm\vert X_{(B_{i-1})}  \Big) \leq \exp \big( - c \kappa d ( s \wedge s^2) \big), \] 
where $\kappa = \E W_{yz}^4$ for some $y, z \in [N]$. 
Hence, using this bound in the proof of \eqref{eq:L_est} yields  
\[ \P \big( L_s^i \geq \ell  \bigm\vert X_{(B_{i-1})} \big) \leq \begin{pmatrix} \norm{\f g_i}^2\\ \ell \end{pmatrix} \exp \big ( - c \kappa d \ell ( s \wedge s^2) \big).\] 
Applying this estimate in the remainder of the proof of \eqref{eq:Z_claim}, we deduce 
\[ \abs{Z_i} \leq \Cnu (1 + K)^{2i} d \bigg( 1 + \frac{\log N}{\norm{\f g_i}^2} \bigg) \bigg( \log d + \frac{\log N}{d} \bigg). \]
Here, we employed that $\norm{\f g_i}_\infty \leq \Cnu (1 + K)^i$. 
Therefore, using the growth of $\norm{\f g_i}$, we obtain the same bound on $\norm{\f w_2}$ as in \eqref{eq:estimate_w_2}. 

When following the proof of \eqref{eq:estimate_w_2} in the proof of Proposition~\ref{pro:approx_trimatrix}, the same adjustments yield the bound used there.

\appendix 

\section{Tridiagonalization} \label{sec:tridiagonalization}

Let $X \in \R^{N \times N}$ be a symmetric matrix and $x \in [N]$. Let $m(x) \deq \dim \op{Span}\{X^n \f 1_x \col n \in \N_0\}$. For $i \in \qq{m - 2}$ define by induction
\begin{equation*}
\f g_0 \deq \f 1_x, \qquad \f g_{i+1} \deq Q_i X \f g_i,
\end{equation*}
where $Q_i$ is the orthogonal projection onto the orthogonal complement of $\op{Span}\{ X^j \f 1_x \col j \in \qq{i}\}$. We call $(\f g_i)_{i \in \qq{m-1}}$ the orthogonal basis associated with $X$ and $x$. Note that this basis is in general not normalized. For convenience, if $m < N - 1$, i.e.\ $\f 1_x$ is not a cyclic vector of $X$, we complete the basis $(\f g_i)_{i \in \qq{m-1}}$ to an orthogonal basis $(\f g_i)_{i \in \qq{N-1}}$ of $\R^N$ in an arbitrary fashion.

We define $M \equiv M^{(X,x)}$ as the matrix $X$ in the orthonormal basis $(\f g_i / \norm{\f g_i})_{i \in \qq{N-1}}$; that is,
\begin{equation*}
M_{ij} \deq \scalarbb{\frac{\f g_i}{\norm{\f g_i}}}{X \frac{\f g_j}{\norm{\f g_j}}}
\end{equation*}
for $i,j \in \qq{N-1}$.

\begin{remark}
It is easy to check that the matrix $M_{\qq{m-1}}$ is tridiagonal, i.e.\ $M_{ij} = 0$ if $\abs{i - j} > 1$ and $i,j \in \qq{m-1}$. Hence, we call $M$ the \emph{tridiagonal matrix} associated with $X$ around $x$.
\end{remark}

\begin{lemma} \label{lem:tridiag_norm_expression}
If $M$ is the tridiagonal matrix of $X$ then
\begin{equation*}
M_{i \, i+1} = \frac{\norm{\f g_{i+1}}}{\norm{\f g_{i}}}
\end{equation*}
for $i \in \qq{m - 2}$.
\end{lemma}
\begin{proof}
We have
\[ M_{i \, i+1} = \frac{\scalar{X \f g_i}{\f g_{i+1}}}{\norm{\f g_i} \norm{\f g_{i+1}}}= \frac{\scalar{X \f g_i}{Q_i X \f g_i}}{\norm{\f g_i}\norm{\f g_{i+1}}} = \frac{\scalar{Q_i X \f g_i}{Q_i X \f g_i}}{\norm{\f g_i}\norm{ \f g_{i+1}}}
= \frac{\norm{\f g_{i+1}}}{\norm{\f g_i}} . \qedhere \] 
\end{proof}

\begin{lemma} \label{lem:tridiagonal_bipartite_graph} 
Let $X = A$ be the adjacency matrix of a bipartite graph (e.g.\ a tree) with vertex set 
$V_0 \cup V_1$ such that $A_{V_0} = 0$ and $A_{V_1} = 0$. Then the following holds. 
\begin{enumerate}[label=(\roman*)]
\item \label{item:support_g_i_bipartite} If $x \in V_0$ then $\f g_i |_{V_{(i+1)\, \mathrm{mod}\, 2}}$ vanishes for all $i$. 
\item \label{item:tridiagonal_bipartite_vanishing_diagonal} The diagonal of the associated tridiagonal matrix $M$ vanishes. 
\end{enumerate} 
\end{lemma}
\begin{proof}
Part (i) follows directly from the bipartite structure of the graph. 
Part (ii) is immediate from $M_{ii} = \scalar{A \f g_i}{\f g_i}/\norm{\f g_i}^2$ and the first part. 
\end{proof}

\begin{lemma} \label{lem:g_i_for_tree} 
If $X= A$ is the adjacency matrix of a tree then 
\[ \f g_i = \f 1_{S_i} + \f d_i \] 
with some $\f d_i \in \R^N$ satisfying $\f d_0 = \f d_1 = 0$ and $\supp \f d_i \subset B_{i-2}$ for all $2 \leq i \leq m$.   
\end{lemma} 

\begin{proof} 
We prove the lemma by induction on $i$. For the induction start, we note that 
$\f g_0 = \f 1_x$, $Q_0$ is the projection onto the complement $\op{Span}\{ \f 1_x \}$ and, thus, 
$\f g_1 = \f 1_{S_1}$.
The induction step follows from 
\[ A \f g_i = \f 1_{S_{i+1}} + ( A \f 1_{S_i} - \f 1_{S_{i+1}} ) + A \f d_i \] 
since $\f 1_{S_{i+1}}$ is invariant under of $Q_i$ and $\supp Q_i (A \f 1_{S_i} - \f 1_{S_{i+1}} )$, $\supp Q_i A \f d_i \subset B_{i-1}$. 
Here, to show the inclusions of the supports, we used that $\supp Q_i\f y\subset B_{i-1}$ if $\supp \f y \subset B_{i-1}$ as well as that $A$ is the adjacency matrix of a tree. 
\end{proof} 

\section{Tridiagonal matrix associated with a regular tree}  \label{sec:tridiagonal_tree}

In this appendix, we compute the tridiagonal matrix representation of $A = \op{Adj}(G)$ if, in the vicinity of some vertex, $G$ has the idealized graph structure 
described in Section~\ref{sec:ideas_proof}. The section complements the explanations in Section~\ref{sec:ideas_proof} and the results are not used in the rest of the paper. 

Throughout this section, we assume that there are $x \in [N]$ and $r \in \N$ such that $G$ has the following structure in $B_{r}(x)$. 
\begin{enumerate}[label=(\roman*)]
\item The induced subgraph $G\vert_{B_{r}(x)}$ on $B_{r}(x)$ is a tree with root $x$. 
\item The root $x$ has $D_x$ children and the vertices in $B_{r}(x)\setminus \{x\}$ have $d$ children. 
\end{enumerate}

\begin{lemma}[Basis and tridiagonal representation] \label{lem:Gram_Schmidt}
Let $\f s_0, \ldots, \f s_{r}$ be the Gram-Schmidt orthonormalization of $\f 1_{x}, A \f 1_{x}, \ldots, A^{r}\f 1_{x}$. Then the following hold true 
\begin{enumerate}
\item For all $i = 0, \ldots, r$,  we have 
\[ \f s_i = \abs{S_i(x)}^{-1/2} \f 1_{S_i(x)}.\] 
\item Let $\f s_{r + 1}, \ldots, \f s_{N-1}$ be a completion of $\f s_0, \ldots, \f s_{r}$ to an 
orthonormal basis of $\R^N$ and 
\[ M \defeq S^* A S, \qquad S \defeq (\f s_0, \ldots, \f s_{N-1}), \] 
the representation of $A$ in this basis. 
Then the upper-left $(r+1) \times (r+1)$ block $M_{\qq{r}}$ 
 of $M$ is independent of $\f s_{r+1}, \ldots, \f s_{N-1}$ and has the tridiagonal form
 \begin{equation} \label{eq:block_structure_M} 
M_{\qq{r}} = 
\begin{pmatrix}
0 & \sqrt{D_x} &&&&
\\
\sqrt{D_x} & 0 & \sqrt{d} &&&
\\
& \sqrt{d} & 0 & \sqrt{d} &&
\\&&\sqrt{d} & 0 & \ddots&
\\
&&&\ddots & \ddots & \sqrt{d}
\\
&&&& \sqrt{d} & 0
\end{pmatrix}.
\end{equation}
\end{enumerate}
\end{lemma} 

Note that the spectra of $A$ and $M$ coincide. We stress that, for all our arguments in the rest of the paper motivated by the construction of $M$ above, only $M_{\qq{r}}$ plays a role. 
Therefore, the special choice of the basis vectors $\f s_{r + 1}, \ldots, \f s_{N-1}$ has no influence on these arguments.

\begin{proof}[Proof of Lemma~\ref{lem:Gram_Schmidt}] 
For the proof of (i), we show inductively that 
\begin{equation} \label{eq:Gram_Schmidt_induction} 
 \f 1_{S_i(x)} = Q_i (A^{i}\f 1_{x}) 
\end{equation}
for $i = 0, 1, \ldots, r$, 
where $Q_i$ is the orthogonal projection onto the orthogonal complement of $\f 1_{x}, \ldots, A^{i-1} \f 1_{x}$ for $i \geq 1$ and $Q_0 = \id$. 
The initial step is trivial as $S_0(x) = \{ x\}$. 

For all $i \geq 1$, we have 
\[ Q_i( A^i \f 1_{x})  = Q_i (A \f 1_{S_{i-1}(x)}) \] 
as well as 
\begin{equation} \label{eq:A_1_S_l} 
 A \f 1_{S_{l}(x)} = \begin{cases} \f 1_{S_1(x)}, & \text{if } l=0, \\ 
 \f 1_{S_2(x)} + D_x \f 1_{x},  &  \text{if } l=1, \\ 
 \f 1_{S_{l+1}(x)} + d \f 1_{S_{l-1}(x)},  & \text{if } l \in [r-1] \setminus \{1\}. 
\end{cases} 
\end{equation} 
Therefore, \eqref{eq:Gram_Schmidt_induction} follows immediately as $\f 1_{S_i(x)}$ and $\f 1_{S_j(x)}$ are orthogonal for $i \neq j$ and $\norm{\f 1_{S_i(x)}} = \abs{S_i(x)}^{1/2}$. 

We start the proof of (ii) by concluding 
\begin{equation} \label{eq:computation_scalar_product} 
 \scalar{\f s_i}{A \f s_j} = \abs{S_i(x)}^{-1/2} \abs{S_j(x)}^{-1/2} \scalar{\f 1_{S_i(x)}}{A \f 1_{S_j(x)}} 
\end{equation} 
for $i,j = 0, \ldots, r$ from (i). 
If $\abs{ i - j} \neq 1$ then this immediately yields $\scalar{\f s_i}{ A \f s_j} = 0$.
Moreover $\abs{S_0(x)} =1$ and $\abs{S_i(x)} = D_x d^{i-1}$ for $i \geq 1$ to \eqref{eq:computation_scalar_product}. For all $i, j = 0, 1, \ldots, r$, we have  
\[ \scalar{\f s_i}{A \f s_j} = \begin{cases} \sqrt{D_x}, & \text{if } \abs{i-j} = 1 \text{ and } ( i=0 \text{ or } j = 0), \\ 
 \sqrt{d}, & \text{if } \abs{i - j} = 1 \text{ and } ( i >0 \text{ and } j > 0), \\ 0, & \text{if } \abs{i - j } \neq 1.  \end{cases} \] 

This yields (ii) and, thus, completes the proof of Lemma~\ref{lem:Gram_Schmidt}. 
\end{proof}

\section{Spectral properties of tridiagonal matrices} \label{app:tridiagonal}

In this section we analyse the spectral properties of the tridiagonal matrices $M(\alpha)$. 
These $(r+1) \times (r+1)$-matrices were defined in \eqref{eq:def_M_alpha} for 
$\alpha >0$ and $r \in \N$.

In Lemma~\ref{lem:transfermatrix} below, we collect and prove a few spectral properties of $M(\alpha)$ for large $r$, in particular about its extreme eigenvalues 
and corresponding approximate eigenvectors. Although we shall not need Lemma~\ref{lem:transfermatrix}, it serves as motivation for the approximate eigenvectors introduced in Section~\ref{sec:large_degree_induces_eigenvalues} for large eigenvalues of the Erd{\H o}s-R\'enyi graph. Moreover, the key concepts behind the proof of Lemma~\ref{lem:transfermatrix} will be needed for the proof of Proposition \ref{pro:bound_eigenvector_perturbation_M_alpha}, and they are most transparent in the simple setting of Lemma~\ref{lem:transfermatrix}. 

\begin{lemma}[Eigenvalues and approximate eigenvectors of $M(\alpha)$]\label{lem:transfermatrix} 
If $\alpha>2$ then the following holds.
\begin{enumerate}[label=(\roman*)]
\item (Extreme eigenvalues) The largest and smallest eigenvalues of $M(\alpha)$, $\lambda_1(M(\alpha))$ and $\lambda_{r+1}(M(\alpha))$, converge to $\Lambda(\alpha)$ and $- \Lambda(\alpha)$, respectively, as $r\rightarrow \infty$.  
\label{item:transfert_1}
\item \label{item:transfermatrix_bulk} (Bulk eigenvalues) The eigenvalues $\lambda_2(M(\alpha)), \ldots, \lambda_{r}(M(\alpha))$ lie in $[-2,2]$. 
\item (Approximate eigenvectors) \label{item:transfert_2} 
Let $\f u=(u_i)_{i=0}^{r}$ and $\f u_-=((-1)^i u_i)_{i=0}^{r}$ have components
\begin{equation} \label{eq:candidate_eig}
u_0 \in \R\setminus \{0\}, \qquad u_1 \defeq \bigg(\frac{\alpha}{\alpha-1}\bigg)^{1/2} u_0, \qquad u_i \defeq \bigg(\frac{1}{\alpha-1}\bigg)^{(i-1)/2} u_1 \quad (i=2, 3, \ldots,r).
\end{equation}
Then $\f u$ and $\f u_-$ are approximate eigenvectors, as $r \to \infty$, of $M(\alpha)$ corresponding to its largest and smallest eigenvalue, respectively. 
\end{enumerate}
\end{lemma} 

The eigenvectors of $M(\alpha)$ can be analysed by a transfer matrix approach. Let $\eta$ be an eigenvalue of $M(\alpha)$ and $\f u = (u_i)_{i=0}^{r}$ a corresponding 
eigenvector. The components of the eigenvalue-eigenvector relation $M(\alpha) \f u = \eta \f u$ read
\begin{equation} \label{eq:eigenvector_M} 
 \sqrt{\alpha} u_1  = \eta u_0, \qquad 
\sqrt{\alpha} u_0 + u_2  = \eta u_1,  \qquad 
u_{i-1} + u_{i+1} = \eta u_i , \qquad u_{r-1} = \eta u_r  
\end{equation}
for $i=2, \ldots,r -1$. Hence, for $i=1, \ldots, r-1$, these relations are equivalent to 
\begin{equation}\label{eq:eigenvector_relation_transfer_matrix} 
\begin{pmatrix} u_{i+1} \\ u_i \end{pmatrix}  = T(\eta)^{i-1} \begin{pmatrix} u_2 \\ u_1 \end{pmatrix},  
\end{equation} 
where we introduced the $2 \times 2$ transfer matrix $T(\eta)$ defined through 
\begin{equation} \label{eq:def_transfer_matrix} 
 T(\eta) \deq \begin{pmatrix} \eta & - 1 \\ 1 & 0 \end{pmatrix}. 
\end{equation}

From now on we suppose that $\abs{\eta}>2$. In this case, we compute the spectrum and the eigenspaces of $T(\eta)$. 
The eigenvalues of $T(\eta)$ are $\gamma(\eta)$ and $\gamma(\eta)^{-1}$, where  we defined 
\begin{equation} \label{eq:def_gamma_eta}
 \gamma(\eta)  \defeq \frac{1}{2}\Big( \eta - \sign(\eta) \sqrt{ \eta^2 - 4} \Big). 
\end{equation}
Note that $\abs{\gamma(\eta)} < 1$.
Moreover, the eigenspaces of $T(\eta)$ associated to $\gamma(\eta)$ and $\gamma(\eta)^{-1}$ are given by 
\begin{equation} \label{eq:span_eigenspace}
 \ker( T(\eta) -\gamma(\eta) I_2) = \mathrm{span}\, \Bigg\{ \begin{pmatrix} \gamma(\eta) \\ 1 \end{pmatrix} \Bigg\}, 
\quad
\ker( T(\eta) -\gamma(\eta)^{-1} I_2) = \mathrm{span}\, \Bigg\{ \begin{pmatrix} 1 \\ \gamma(\eta) \end{pmatrix} \Bigg\}.
 \end{equation}

In the following, we denote the standard basis vectors of $\R^{r+1}$ by $\f e_0, \ldots, \f e_r$.

\begin{proof}[Proof of Lemma \ref{lem:transfermatrix}]
We first prove \ref{item:transfermatrix_bulk}. To that end, we consider $M(\alpha)$ as a rank-two perturbation of $M(1)$. It is well-known that 
\[ \spec ( M(1)) = \bigg\{ 2 \cos \bigg(\frac{\pi k}{r + 2}\bigg) \col k =1, \ldots, r + 1 \bigg\} \subset [-2,2]. \] 
This implies \ref{item:transfermatrix_bulk} by Weyl's interlacing inequalities, since the matrix $M(\alpha)-M(1)$ has rank two with one positive eigenvalue and one negative eigenvalue.

We now show \ref{item:transfert_1} and \ref{item:transfert_2} simultaneously. Let $\f u$ and $\f u_-$ be defined as in \ref{item:transfert_2}. 
We only focus on the largest eigenvalue of $M(\alpha)$ and $\f u$. The same arguments work for the smallest eigenvalue and~$\f u_-$. 

We set $\eta = \Lambda(\alpha)$ and obtain $\gamma(\eta) = (\alpha-1)^{-1/2}$. Thus, $(u_{i+1}, u_i)^* \in \ker(T(\eta) - \gamma(\eta)I_2)$ for all $i=1, \ldots, r-1$. 
Hence, the equivalence between \eqref{eq:eigenvector_M} for $i=2, \ldots, r-1$ and \eqref{eq:eigenvector_relation_transfer_matrix} implies that 
\[ (M(\alpha) - \eta I_{r+1})\f u = (u_{r-1} - \eta u_{r}) \f e_{r}.  \] 
Here, we also used $\eta=\Lambda(\alpha) = \alpha/\sqrt{\alpha-1}$ and the relation between $u_1$ and $u_0$. 
Therefore, since $\alpha > 2$, we find $\norm{(M(\alpha)-\Lambda(\alpha) I_{r+1})\f u}  \to 0$ as $r \to \infty$. This completes the proof of Lemma~\ref{lem:transfermatrix}. 
\end{proof}

The following proposition provides an eigenvector delocalization bound for tridiagonal matrices whose
structure is similar to the one of $M(\alpha)$ in the sense that, starting from the second row and column, the diagonal entries are small while the 
offdiagonal entries are close to one. 
For its formulation, we need some notation which we define now. 
For $\eta >2$, we recall the definition of $\gamma(\eta)$ from \eqref{eq:def_gamma_eta} and introduce 
\begin{equation} \label{eq:def_function_delta} 
 \delta(m_{00}, m_{01}, m_{11}, m_{12}, \eta) = \frac{\gamma(\eta) (\eta - m_{00})(\eta - m_{11}) - \gamma(\eta) m_{01}^2  - m_{12}(\eta - m_{00})}{m_{01}^2  + \gamma(\eta)(\eta - m_{00})m_{12} - (\eta - m_{00})(\eta - m_{11})} 
\end{equation}
whenever the denominator on the right-hand side is different from zero. 
For $\eta >2$ and $\eps >0$, we also define 
\begin{equation} \label{eq:gamma_larger}
\gamma_\geq(m_{00}, m_{01}, m_{11}, m_{12}, \eta,\eps) \defeq \gamma(\eta)^{-1} - \frac{8(3+\eta)\eps}{1- \gamma(\eta)^2} \Big( 1 + 1\vee ({\delta(m_{00}, m_{01}, m_{11}, m_{12}, \eta)})^2\Big)^{1/2}. 
\end{equation}

\begin{proposition}[Delocalization bound for tridiagonal matrices] \label{pro:bound_eigenvector_perturbation_M_alpha} 
Let $\wt{M}$ be a symmetric tridiagonal $(r+1) \times (r+1)$ matrix and $\f b = (b_i)_{i \in \qq{r}} \in \R^{r+1}$.   
Let $\eta >2$. 
We set 
$\eps \defeq \max_{i \in [r-1]} \big(\abs{\wt{M}_{ii}} \vee \abs{\wt{M}_{i\, i+1} - 1} \big)$,  
$\delta \defeq \delta(\wt{M}_{00},\wt{M}_{01},\wt{M}_{11},\wt{M}_{12}, \eta)$ and $\gamma_\geq = \gamma_\geq(\wt{M}_{00},\wt{M}_{01},\wt{M}_{11},\wt{M}_{12}, \eta,\eps)$. 
If $\eps \leq 1/2$ and the condition
\begin{equation} \label{eq:bound_eigenvector_condition} 
 \eta^2 \geq 4 + \frac{4^5(3+\eta)^2\eps^2}{(1 - \gamma(\eta)^2)^2} \big ( 1 + 1 \vee \delta^2\big)  
\end{equation}
is satisfied and $(\wt{M} - \eta I_{r+1})\f b \in \op{Span} \{\f e_r\}$ then 
\[ \frac{(b_0)^2}{\norm{\f b}^2} \leq \frac {8(\wt{M}_{01}\wt{M}_{12})^2} 
{((\wt{M}_{01})^2 - (\eta - \wt{M}_{11})(\eta - \wt{M}_{00}) + \gamma(\eta) \wt{M}_{12} ( \eta -\wt{M}_{00} ))^2 } \bigg[ (\gamma_\geq)^{-2r} \wedge \frac{1}{r-1} \bigg] 
\] and 
\begin{equation} \label{eq:lower_bound_gamma_geq} 
 \gamma_\geq \geq 1 + \frac{1}{2} \big(\gamma(\eta)^{-1} - 1 \big) \geq 1. 
\end{equation}
\end{proposition} 

\begin{proof} 
As $(\wt{M} - \eta I_{r+1}) \f b \in \op{Span}\{\f e_r\}$ we have 
\begin{equation} \label{eq:relations_b_i_perturbed} 
 \wt{M}_{00}  b_0 + \wt{M}_{01}  b_1 = \eta  b_0, \quad \wt{M}_{i\, i-1}  b_{i-1} + \wt{M}_{ii}  b_i + \wt{M}_{i \, i + 1}  b_{i+1} = \eta  b_i
\end{equation}
for any $i=1,\ldots, r-1$.  
Hence, 
\begin{equation} \label{eq:b_i_transfer_matrix} 
\begin{pmatrix} b_{i+1} \\ b_{i} \end{pmatrix} = T_i \begin{pmatrix} b_{i} \\ b_{i-1} \end{pmatrix}, \qquad T_i \defeq \frac{1}{\wt{M}_{i \, i+1}} \begin{pmatrix} \eta - \wt{M}_{i\, i}& - \wt{M}_{i\, i-1} \\ \wt{M}_{i \, i+1} & 0 \end{pmatrix} 
\end{equation}
for $i=1, \ldots, r-1$. For $i \geq 2$, we define $R_i \defeq T_i - T$, where $T = T(\eta)$ is defined as in \eqref{eq:def_transfer_matrix}. As $\eps = \norm{\wt{M} - M(\alpha)} \leq 1/2$, we have $\norm{R_i} \leq 2(3+ \eta)\eps$ uniformly for $i \geq 2$. 

In the rest of the proof, we write $\gamma \equiv \gamma(\eta)$ for $\gamma(\eta)$ defined in \eqref{eq:def_gamma_eta}. 
For each $i\geq 1$, we denote by $p_i$ and $q_i$ the first and second component of $({b}_{i+1},{b}_i)^*$ in the eigenbasis of $T$, respectively. That is 
\begin{equation} \label{eq:def_p_i_q_i} 
 \begin{pmatrix} p_i \\ q_i \end{pmatrix} \defeq V^{-1} \begin{pmatrix} {b}_{i+1} \\ {b}_i \end{pmatrix}, \qquad V = \begin{pmatrix} \gamma & 1 \\ 1 & \gamma \end{pmatrix}. 
\end{equation}
The fact that $V^{-1} T V$ is diagonal can be easily read off from \eqref{eq:span_eigenspace}. 

We shall now show that 
\begin{equation} \label{eq:p_i_div_q_i_bounded} 
\frac{\abs{p_i}}{\abs{q_i}}\leq 1 \vee \bigg(\frac{\abs{p_1}}{\abs{q_1}}\bigg), \qquad \abs{q_i} \geq (\gamma_\geq)^{i-1} \abs{q_1} 
\end{equation} 
for all $i\geq 1$ by induction on $i$. The assertion is trivial for $i=1$. 
From \eqref{eq:def_p_i_q_i} and \eqref{eq:b_i_transfer_matrix}, we conclude 
\[ \begin{pmatrix} p_{i+1}\\ q_{i+1} \end{pmatrix} = V^{-1} \bigg( V \begin{pmatrix} \gamma & 0 \\ 0 & \gamma^{-1} \end{pmatrix} V^{-1} + R_{i+1} V V^{-1} \bigg) \begin{pmatrix} b_{i+1} \\ b_i \end{pmatrix} = 
 \bigg( \begin{pmatrix} \gamma & 0 \\ 0 & \gamma^{-1} \end{pmatrix} + V^{-1} R_{i+1} V \bigg) \begin{pmatrix} p_i \\ q_i \end{pmatrix}. \] 
 Estimating the first component of this relation implies
\[ \begin{aligned} 
\abs{p_{i+1}} & \leq \gamma \abs{p_i} + \normbb{V^{-1} R_{i+1} V \begin{pmatrix} p_i \\ q_i \end{pmatrix}} \\ 
 & \leq \bigg(\gamma \frac{\abs{p_i}}{\abs{q_i}} + \norm{V^{-1} R_{i+1} V} \bigg( 1 + \bigg(\frac{p_i}{q_i}\bigg)^2\bigg)^{1/2}\bigg) \bigg)\abs{q_i} \\
& \leq \bigg(1 \vee \frac{\abs{p_1}}{\abs{q_1}} \bigg) \Big( \gamma + \sqrt{2} \norm{V^{-1} R_{i+1} V} \Big) \abs{q_i}, 
\end{aligned} \] 
where we used that $1 + p_i^2/q_i^2 \leq 2 (1 \vee p_1^2/q_1^2)$ by the induction hypothesis in the last step. 
Similarly, we bound the second component from below and obtain 
\begin{equation} \label{eq:q_i_plus_1_lower_bound} 
\begin{aligned} 
 \abs{q_{i+1}} & \geq \bigg(\gamma^{-1} - \norm{V^{-1} R_{i+1} V} \bigg( \bigg(\frac{p_i}{q_i}\bigg)^{2} + 1\bigg)^{1/2} \bigg) \abs{q_i} \\ 
 & \geq \bigg( \gamma^{-1} - \norm{V^{-1} R_{i+1} V} \bigg(1 + 1 \vee \bigg(\frac{p_1}{q_1} \bigg)^2\bigg)^{1/2} \bigg)\abs{q_i}
\end{aligned} 
\end{equation}
due to the induction hypothesis. 

By dividing the upper bound on $\abs{p_{i+1}}$ by the lower bound on $\abs{q_{i+1}}$, we see that the induction step for the first bound in \eqref{eq:p_i_div_q_i_bounded} is shown if 
\begin{equation} \label{eq:bound_eigenvector_condition_induction} 
  \gamma + \sqrt{2} \norm{V^{-1} R_{i+1} V} \leq \gamma^{-1} - \norm{V^{-1} R_{i+1} V} \bigg(1 + 1 \vee \bigg(\frac{p_1}{q_1}\bigg)^{2} \bigg)^{1/2}. 
\end{equation}
We now deduce this bound from \eqref{eq:bound_eigenvector_condition}. 
 To that end, we first compute $p_1/q_1$. The definition of $p_1$ and $q_1$ in \eqref{eq:def_p_i_q_i} yields $p_1 = (-b_1 + \gamma b_2)/(\gamma^2 - 1)$ and $q_1= (- b_2 + \gamma b_1)/(\gamma^2 - 1)$. 
We use the first relation in \eqref{eq:relations_b_i_perturbed} and the second relation in \eqref{eq:relations_b_i_perturbed} with $i=1$ to express $b_1$ and $b_2$ in terms of $b_0$. Then an easy computation shows that 
\begin{subequations} 
\begin{align} 
p_1 & = \frac{\gamma b_2 - b_1}{\gamma^2 -1} = \frac{b_0 \big( \gamma ( \eta - \wt{M}_{11})(\eta - \wt{M}_{00}) - \gamma ( \wt{M}_{01} )^2 - \wt{M}_{12}( \eta - \wt{M}_{00} ) \big)}{(\gamma^2-1)\wt{M}_{01}\wt{M}_{12}} 
, \\ 
 q_1 & = \frac{- b_2 + \gamma b_1}{\gamma^2 -1}  = \frac{b_0\big( (\wt{M}_{01})^2 - (\eta - \wt{M}_{11})(\eta - \wt{M}_{00}) + \gamma \wt{M}_{12} ( \eta -\wt{M}_{00} ) \big) }{(\gamma^2-1)\wt{M}_{01}\wt{M}_{12}} 
.  \label{eq:representation_q_1} 
\end{align} 
\end{subequations}
Therefore,  we obtain 
\[ \frac{p_1}{q_1} = \frac{-b_1 + \gamma b_2}{-b_2 + \gamma b_1} = \delta \defeq \delta (\wt{M}_{00},\wt{M}_{01},\wt{M}_{11},\wt{M}_{12}, \eta),  \]  
where we used the function $\delta$ defined in \eqref{eq:def_function_delta}.
Thus, as $\norm{V^{-1} R_{i+1} V} \leq 4 \norm{R_{i+1}}/(1-\gamma^2) \leq 8(3+\eta)\eps /(1-\gamma^2)$ and  $\gamma^{-1}  - \gamma = \sqrt{\eta^2 - 4}$, 
the definition of $\gamma(\eta)$ in \eqref{eq:def_gamma_eta} shows that \eqref{eq:bound_eigenvector_condition_induction} is a consequence of \eqref{eq:bound_eigenvector_condition}. 
This completes the induction step for the first estimate in \eqref{eq:p_i_div_q_i_bounded}. 

From \eqref{eq:q_i_plus_1_lower_bound}, $p_1/q_1 = \delta$ and $\norm{V^{-1} R_{i+1} V} \leq 8(3 +\eta)\eps/(1-\gamma^2)$, 
we deduce $\abs{q_{i+1}} \geq \gamma_\geq \abs{q_i}$. Thus, we have completed the proof of \eqref{eq:p_i_div_q_i_bounded}.

We now prove that \eqref{eq:bound_eigenvector_condition} also implies the lower bound on $\gamma_\geq$ in \eqref{eq:lower_bound_gamma_geq}. 
The definition of $\gamma(\eta)$ in \eqref{eq:def_gamma_eta} yields $2(\gamma^{-1}-1) \geq (\gamma^{-1}-\gamma) = \sqrt{\eta^2 - 4} $. 
Thus, we obtain from \eqref{eq:bound_eigenvector_condition} that 
\[\begin{aligned} 
 \gamma_\geq - 1 - \frac{1}{2} \big( \gamma^{-1} - 1 \big) & = \frac{1}{2} \bigg[ \gamma^{-1} - 1 - \frac{16(3 +\eta)\eps}{1-\gamma^2} \bigg(1 + 1 \vee \bigg(\frac{p_1}{q_1}\bigg)^{2} \bigg)^{1/2} \bigg]  \\
& \geq \frac{1}{4} \bigg[ \sqrt{\eta^2 - 4} -\frac{32(3 +\eta)\eps}{1-\gamma^2} \bigg(1 + 1 \vee \bigg(\frac{p_1}{q_1}\bigg)^{2} \bigg)^{1/2} \bigg]. 
\end{aligned} \] 
Owing to \eqref{eq:bound_eigenvector_condition}, the right-hand side is positive. This immediately implies \eqref{eq:lower_bound_gamma_geq}.

Owing to the second bound in \eqref{eq:p_i_div_q_i_bounded}, we have 
\begin{equation}\label{eq:bound_norm_b_aux1} 
\begin{aligned} 
 2 \sum_{i=0}^r (b_i)^2 & \geq \sum_{i=0}^{r-1} \normbb{\begin{pmatrix} b_{i+1} \\ b_i \end{pmatrix}}^2 \\ 
 & = \sum_{i=0}^{r-1} \normbb{ VV^{-1} \begin{pmatrix} b_{i+1} \\ b_i \end{pmatrix}}^2 \\ 
 & \geq \frac{1}{\norm{V^{-1}}^{2}} \sum_{i=1}^{r-1} \normbb{\begin{pmatrix} p_i \\ q_i \end{pmatrix}}^2 \\ 
 & \geq \frac{(\gamma^2-1)^2}{4} \abs{q_1}^2 \Big[ (\gamma_\geq)^{2r} \vee (r-1) \Big]. 
\end{aligned} 
\end{equation}  
Here, we pulled $V$ out of the norm in the third step and used \eqref{eq:def_p_i_q_i}. The fourth step is a consequence of $\norm{V^{-1}} \leq 2 (1 - \gamma^2)^{-1}$, estimating the norm by its second component and using 
the second bound in \eqref{eq:p_i_div_q_i_bounded} as well as $\gamma_\geq \geq 1$ due to \eqref{eq:lower_bound_gamma_geq}.  

Finally, applying \eqref{eq:representation_q_1} to \eqref{eq:bound_norm_b_aux1} completes the proof of Proposition~\ref{pro:bound_eigenvector_perturbation_M_alpha}.
\end{proof}

\section{Degree distribution of the Erd{\H{o}}s-R\'enyi graph} \label{sec:degrees}
The content of this section is standard, and we include it for completeness and the reader's convenience. It is essentially contained in \cite[Chapter 3]{Bol01}. We do not aim for sharp estimates of the error probabilities; instead, our goal here is to collect basic qualitative facts about the behaviour of the largest degrees of an Erd{\H o}s-R\'enyi graph, which, using Theorem \ref{thm:correspondence_eigenvalue_large_degree}, can be used to understand the key properties of the extremal eigenvalues. We recall the normalized degree \eqref{eq:def_alpha_x} and the random permutation \eqref{eq:def_sigma_permutation}.

To formulate qualitative statements conveniently, we use the symbol $o(1)$ to denote any function of $N$ that converges to zero, and say that an $N$-dependent event $\Xi \equiv \Xi_N$ holds with high probability if $\P(\Xi) = 1 - o(1)$.

The distribution of the largest degrees is best analysed using the function
\begin{equation} \label{eq:def_f_d} 
 f_d(\alpha) \defeq d( \alpha \log \alpha - \alpha + 1) + \frac{1}{2} \log (2 \pi \alpha d) 
\end{equation}
for $\alpha \geq 1$. For its interpretation, we note that if $Y \eqdist \op{Poisson}(d)$ then by Stirling's formula we have for any $k \in \N$
\begin{equation*}
\P(Y = k) = \exp\pbb{-f_d(k/d) + O \pbb{\frac{1}{k}}}.
\end{equation*}

It is easy to see that the function $f_d : [1,\infty) \to \big[\frac{1}{2} \log (2 \pi d), \infty\big)$ is bijective and increasing. Therefore there is a universal constant $C > 0$ such that for $1 \leq l \leq \frac{N}{C \sqrt{d}}$ the equation
\begin{equation*}
f_d(\beta) = \log (N/l)
\end{equation*}
has a unique solution $\beta \equiv \beta_l(d)$.
The interpretation of $\beta$ is the typical location of $\alpha_{\sigma(l)}$. By the implicit function theorem, we find that $\beta_l$  on the interval $\bigl(0, \frac{N^2}{C l^2}\bigr]$ is a decreasing bijective function.

We are interested in normalized degrees greater than or equal to $2$. This motivates the definition
\begin{equation*}
\cal L_0(d) \deq \max \{l \geq 1 \col \beta_l(d) \geq 2\},
\end{equation*}
whose interpretation is the typical number of normalized degrees greater than or equal to $2$. By definition, $\beta_l(d) \geq 2$ for all $l \leq \cal L_0(d)$. Note that $\cal L_0(d)$ is nonzero if and only if $d \leq d_*$, where $d_*$ is defined as the unique solution of $\beta_1(d_*) = 2$. More explicitly, $d_*$ satisfies $f_{d_*}(2) = \log N$.

\begin{proposition} \label{lem:degree_distr}
Let $\xi \equiv \xi_N$ be a positive sequence tending to $\infty$.
If $1 \leq d \leq d_*$ and $1 \leq l \leq \cal L_0(d)$ then with high probability we have
\begin{equation} \label{deg_est_1}
\abs{\alpha_{\sigma(l)} - \beta_l(d)} \leq \frac{1 \vee (\xi / \log \beta_l(d))}{d}.
\end{equation}
If $d > d_*$ then with high probability we have
\begin{equation} \label{deg_est_2}
\alpha_{\sigma(1)} \leq 2 + \frac{\xi}{d}.
\end{equation}
\end{proposition}

\begin{proof}
Throughout the proof we suppose that $2 \leq \alpha \leq \frac{\sqrt{N}}{C d}$ for some large enough universal constant $C$. From the definition of $\beta_l(d)$, it is easy to check that this condition is satisfied for $\alpha = \beta_l(d)$ for $1 \leq d \leq d_*$ and $1 \leq l \leq \cal L_0(d)$. The proof of \eqref{deg_est_1} consists of an upper and a lower bound. The former is proved using a first moment method and the latter using a second moment method. We make use of the counting function $\cal N_t \deq  \sum_{x \in [N]} \ind{D_x \geq t}$. Note that by Poisson approximation of the binomial random variable $D_x = d \alpha_x$, see \cite[Lemma 3.3]{BBK2}, there is a universal constant $C$ such that
\begin{equation} \label{Poisson_approx}
C^{-1} N \ee^{-f_d(\alpha)} \leq \E \cal N_{\alpha d} \leq C N \ee^{-f_d(\alpha)}.
\end{equation}

Let $1 \leq d \leq d_*$ and $1 \leq l \leq \cal L_0(d)$.
We begin by proving an upper bound on $\alpha_{\sigma(1)} = D_{\sigma(1)}/d$. Using \eqref{Poisson_approx} we get
\begin{equation} \label{upper_bound_deg}
\P (\alpha_{\sigma(l)} \geq \alpha)  = \P \pb{\cal N_{\alpha d} \geq l} \leq \frac{\E \cal N_{\alpha d}}{l} \leq \frac{C N}{l} \ee^{-f_d(\alpha)}.
\end{equation}

Next, we prove a lower bound on $\alpha_{\sigma(l)}$. Suppose that $\E \cal N_{\alpha d} \geq 2l$. Then using a second moment method, we find
\begin{equation*}
\P (\alpha_{\sigma(l)} \geq \alpha)  = \P \pb{\cal N_{\alpha d} \geq l} \geq \P \pb{\abs{\cal N_{\alpha d} - \E \cal N_{\alpha d}} < \E \cal N_{\alpha d} / 2} \geq 1 - \frac{4 \var (\cal N_{\alpha d})}{(\E \cal N_{\alpha d})^2}.
\end{equation*}
By \cite[Lemma 3.11]{Bol01} we have $\var (\cal N_{\alpha d}) \leq C \E \cal N_{\alpha d}$ for some universal constant $C$, which yields
\begin{equation} \label{lower_bound_deg}
\P (\alpha_{\sigma(l)} \geq \alpha) \geq 1 - \frac{C}{\E \cal N_{\alpha d}} \geq 1 - \frac{C}{N \ee^{-f_d(\alpha)}},
\end{equation}
where we used \eqref{Poisson_approx}.
From \eqref{upper_bound_deg} we conclude that $\alpha_{\sigma(l)} \leq \alpha$ with high probability if
\begin{equation}
f_d(\alpha) - \log (N/l) \to \infty \quad \text{as} \quad N \to \infty.
\end{equation}
From \eqref{lower_bound_deg} we conclude that $\alpha_{\sigma(l)} \geq \alpha$ with high probability if $\E \cal N_{\alpha d} \geq 2l$ and $f_d(\alpha) - \log N \to -\infty$ as $N \to \infty$. By \eqref{Poisson_approx}, both of these conditions are satisfied if
\begin{equation}
f_d(\alpha) - \log (N/l) \to -\infty \quad \text{as} \quad N \to \infty.
\end{equation}
Now \eqref{deg_est_1} follows easily by choosing $\alpha = \beta_l(d) \pm \frac{1 \vee (\xi / \log \beta_l(d))}{d}$, using that $f_d'(\alpha) \geq d \log \alpha$.

The proof of \eqref{deg_est_2} is analogous to the the proof of the upper bound in \eqref{deg_est_1}, and we omit  the details.
\end{proof}

\bibliography{bibliography} 
\bibliographystyle{amsplain}

\paragraph{Acknowledgements}
We gratefully acknowledge funding from the European Research Council (ERC) under the European Union’s Horizon 2020 research and innovation programme (grant agreement No.\ 715539\_RandMat) and from the Swiss National Science Foundation through the SwissMAP grant.

\bigskip

\noindent
Johannes Alt (\href{mailto:johannes.alt@unige.ch}{johannes.alt@unige.ch})
\\
Rapha\"el Ducatez (\href{mailto:raphael.ducatez@unige.ch}{raphael.ducatez@unige.ch})
\\
Antti Knowles (\href{mailto:antti.knowles@unige.ch}{antti.knowles@unige.ch})
\\
University of Geneva, Section of Mathematics, 2-4 Rue du Li\`evre, 1211 Gen\`eve 4, Switzerland.

\end{document}